\documentclass[10pt,a4paper]{amsart}

\usepackage{amstext,amssymb,amsopn,amsthm,mathrsfs,amsfonts,amsmath,amsxtra}

\usepackage{dsfont,multirow,comment,enumerate,bm,float}

\usepackage{graphicx,color}

\allowdisplaybreaks

\newtheorem{theorem}{Theorem}[section]
\newtheorem{corollary}[theorem]{Corollary}
\newtheorem{proposition}[theorem]{Proposition}

\newtheorem{lemma}[theorem]{Lemma}
\newtheorem{remark}[theorem]{Remark}

\DeclareMathOperator{\lin}{lin}
\DeclareMathOperator{\spann}{span}
\DeclareMathOperator{\domain}{Dom}

\DeclareMathOperator{\support}{supp}
\DeclareMathOperator{\sign}{sign}

\def\e{e}

\def\W{\mathbf W}
\def\L{\mathbf L}
\def\a{\alpha}
\def\b{\beta}

\title[On derivatives and related objects for Fourier-Bessel expansions]
	{On derivatives, Riesz transforms and Sobolev spaces \\ for Fourier-Bessel expansions} 

\author[B.\ Langowski]{Bartosz Langowski}

\address{Bartosz Langowski, \newline
Indiana University, Department of Mathematics, \newline
831 East 3rd St., Bloomington, IN 47405, USA \newline
\indent and \newline
			Faculty of Pure and Applied Mathematics, 
      Wroc\l{}aw University of Science and Technology \newline
      Wyb.\ Wyspia\'nskiego 27,
      50--370 Wroc\l{}aw, Poland      
      }
\email{balango@iu.edu}

\author[A.\ Nowak]{Adam Nowak}

\address{Adam Nowak, \newline
			Institute of Mathematics,
      Polish Academy of Sciences, \newline
      \'Sniadeckich 8,
      00--656 Warsaw, Poland 
}
\email{adam.nowak@impan.pl}

\begin{document}


\begin{abstract}
We study the problem of an appropriate choice of derivatives associated with discrete Fourier-Bessel expansions.
We introduce a new so-called essential measure Fourier-Bessel setting, where the relevant derivative is
simply the ordinary derivative. Then we investigate Riesz transforms and Sobolev spaces in this context.
Our main results are $L^p$-boundedness of the Riesz transforms (even in a multi-dimensional situation)
and an isomorphism between the Sobolev and Fourier-Bessel potential spaces.
Moreover, throughout the paper we collect various comments concerning two other closely related
Fourier-Bessel situations that were considered earlier in the literature. We believe that our
observations shed some new light on analysis of Fourier-Bessel expansions.
\end{abstract}

\maketitle
\thispagestyle{empty}

\footnotetext{

Research supported by the National Science Centre of Poland within the project OPUS 2017/27/B/ST1/01623.
}

\section{Introduction} \label{sec:intro}

Discrete Fourier-Bessel expansions have been present in the literature for a long time, see e.g.\ Watson's classic monograph
\cite[Chapter XVII]{watson} and references given there.
Harmonic analysis related to these expansions has been widely developed over many decades, and
a good deal of this activity falls on the 21st century. For instance, convergence of Fourier-Bessel series
in various senses was studied by Wing \cite{wing}, Hochstadt \cite{H}, Benedek and Panzone \cite{BePa1,BePa2,BePa3},
Markett \cite{Mar}, Guadalupe et al.\ \cite{GPRV}, Balodis and C\'ordoba \cite{BaCo}, Stempak \cite{St2},
and Ciaurri and Roncal \cite{CiRo1,CiRo15,CiRo3}. Transplantation problems related to Fourier-Bessel expansions were
investigated by Gilbert \cite{Gi}, Stempak \cite{St1}, and Ciaurri and Stempak \cite{CiSt1,CiSt2}.
Transference principles involving Fourier-Bessel expansions were developed by Betancor and Stempak \cite{BeSt1,BeSt2},
Betancor \cite{Bet1}, and Betancor et al.\ \cite{BFRS}. Conjugacy and Riesz transforms for Fourier-Bessel expansions
were considered by Muckenhoupt and Stein \cite{MuSt}, Ciaurri and Stempak \cite{CiSt3}, Nowak and Stempak \cite{NoSt1},
Ciaurri and Roncal \cite{CiRo4}, and Wr\'obel \cite{Wr}.
Fourier-Bessel heat kernel bounds were obtained by Nowak and Roncal \cite{NoRo1,NoRo2},
and Ma\l{}ecki et al.\ \cite{MSZ}. Further harmonic analysis topics in the context of Fourier-Bessel expansions
can be found in Ciaurri and Roncal \cite{CiRo2} (square functions), Nowak and Roncal \cite{NoRo3} (potential operators),
and Dziuba\'nski et al.\ \cite{DPRS} (Hardy spaces). A list of selected open questions in the Fourier-Bessel analysis
can be found in \cite{Bet}. These accounts are by no means complete.

The aim of the present paper is to advance knowledge about and understanding of Fourier-Bessel expansions and the related analysis.
We offer a new insight into several fundamental definitions existing in the literature on the subject.
Moreover, we introduce and study new objects, which seem to be philosophically more appropriate than their prototypes
considered earlier. The crucial observation that inspired this research is the following (see \cite[p.\,1340]{NoRo1} for
an earlier and more vague formulation).
\vspace{5pt}\\
\noindent \emph{A considerable part of the literature devoted to Fourier-Bessel expansions ignores in certain ways existence
of the right boundary of the interval $(0,1)$ where the expansions naturally live.}
\vspace{5pt}\\
\noindent
The problem pertains to Riesz transforms and the conjugacy scheme for Fourier-Bessel expansions since, roughly speaking,
the derivatives involved, hence also the emerging Riesz transforms, should ``recognize'' or even reflect presence of the right boundary.
This also has an influence on appropriate definitions of Sobolev spaces associated with Fourier-Bessel expansions.
Actually, such spaces were not considered before, and the reason is, apparently, the aforementioned issue.
Another aspect of analysis related to Fourier-Bessel expansions, where the right boundary was not taken into account
properly, are weighted estimates for various operators. The literature commonly assumes that a power weight in this context
has the form $w(x) = x^{a}$, while naturally it should be $w(x) = x^{a}(1-x)^{b}$. A similar remark
concerns also more general weights. For instance, instead of local $A_p$ weights (with locality referring to the left endpoint
of $(0,1)$) one should rather consider double local $A_p$ weights (cf.\ \cite{CNS})
with locality referring to both endpoints of $(0,1)$. All this affects optimality of various results.
A plausible explanation of this situation is that the case of discrete Fourier-Bessel expansions on the interval $(0,1)$
was, to an excessive extent, viewed through the perspective of continuous Fourier-Bessel expansions on $(0,\infty)$ (the context
of the Hankel transform), where there is no right boundary. For instance, the associated Laplacians in both situations
are the same as differential operators, but their factorizations in terms of suitable derivatives are not necessarily inherited
in a similar trivial way, contrary to what one can find in the literature.

The main achievements of this paper are the following.
\begin{itemize}
\item[(A1)]
We introduce new \emph{derivatives} associated with natural and Lebesgue measures Fourier-Bessel settings.
These derivatives are philosophically more appropriate than those used so far in the literature, since they take into
account the existence of the right boundary of $(0,1)$ and the behavior of the Fourier-Bessel systems at the right endpoint of $(0,1)$.
Moreover, they lead to Riesz transforms and a conjugacy scheme that are more consistent with many, if not practically all, other cases
of classical discrete orthogonal expansions, cf.\ \cite[Section 7]{NoSt1}. Furthermore, the new derivatives are suitable for defining
the associated Sobolev spaces, while the previously used derivatives are not.
\item[(A2)]
We derive qualitatively new orthogonal and complete systems called \emph{differentiated Fourier-Bessel systems}, which are defined
in terms of the standard oscillatory Bessel function $J_{\nu}$ and its successive zeros. Expansions with respect
to these new systems seem to be a kind of novelty in the circle of classical long known expansions based
on the Bessel function, like the series of Fourier-Bessel, Fourier-Dini, Fourier-Neumann, Kapteyn, and Schl\"omilch; see \cite{watson}.
\item[(A3)]
We introduce a new Fourier-Bessel setting related to another measure that we call \emph{essential measure}.
This framework turns out to be the most convenient for studying a number of fundamental harmonic analysis problems.
One of its most significant aspects is relative simplicity, since the associated derivative is just the usual derivative,
and the bottom eigenfunction is a constant function $\boldsymbol{1}$. Moreover, there is a probabilistic variant of this
setting, which may be regarded as a Fourier-Bessel counterpart of probabilistic contexts related to the classical orthogonal
systems of Hermite, Laguerre and Jacobi polynomials. Furthermore, as we shall explain in Remark \ref{rem:CZ}, the essential measure
Fourier-Bessel framework is the most effective one from the perspective of application of the Calder\'on-Zygmund operator
theory.
\item[(A4)]
We obtain $L^p$ estimates for vectorial Riesz transforms in the multi-dimensional essential measure Fourier-Bessel context with
an explicit constant independent of the dimension and of the parameter of type, and with explicit quantitative dependence on $p$.
Each of these aspects, i.e.\ dimensionless and parameterless constant, explicit dependence on $p$,
are novel compared to previous studies of Riesz transforms for Fourier-Bessel expansions, see e.g.\ \cite{CiSt3,CiRo4,Wr}.
\item[(A5)]
We define and study Sobolev space of the first order in the essential measure Fourier-Bessel setting. In particular, we
prove that the Sobolev space is isomorphic to the Fourier-Bessel potential space, which is an analogue of the celebrated
classical result of A.P.\, Calder\'on. According to our best knowledge, Sobolev spaces for Fourier-Bessel
expansions were not investigated so far, in contrast to many other situations of classic orthogonal expansions,
see e.g.\ \cite{BFRTT,BT1,BT2,GLLNU,La1,La2,La3}.
\end{itemize}

The paper is organized as follows (item (Sn) stands for Section n).
\begin{itemize}
\item[(S2)]
In this section
we first summarize the general notation used throughout. Then we present various facts and formulae connected with the
Bessel function $J_{\nu}$. In particular, we introduce auxiliary functions $R_n^{\nu}$, $n \ge 1$ (defined in terms
of $J_{\nu}$) and establish their basic properties.
\item[(S3)]
We describe briefly the well-known natural and Lebesgue measures Fourier-Bessel contexts. Then we present the new
essential measure Fourier-Bessel situation, including its probabilistic variant. For technical reasons, we also
invoke the Jacobi trigonometric Lebesgue measure setting scaled to the interval $(0,1)$.
\item[(S4)]
This section is devoted to derivatives in the situations mentioned in (S3) above. In particular, we recall the old derivatives,
which have been used so far, and introduce the new ones in the natural and Lebesgue measures Fourier-Bessel settings.
\item[(S5)]
We study Fourier-Bessel differentiated systems that emerge from the original Fourier-Bessel systems by differentiating with
the derivatives introduced in (S4). We describe the boundary behavior of these systems (Proposition \ref{prop:asd}) and
show that they form orthogonal bases in the corresponding $L^2$ spaces (Proposition \ref{prop:orth} and Theorem \ref{thm:L2dense}).
Here the most difficult issue is completeness of the systems. The reasoning is rather involved and essentially follows
the strategy from Hochstadt's work \cite{H}, though the details are largely independent of \cite{H}.
\item[(S6)]
We investigate Riesz transforms for Fourier-Bessel expansions. We focus mostly on the essential measure setting, where
we prove the results announced in (A4) above (Theorem \ref{thm:Riesz} and Theorem~\ref{thm:Rieszd}). This is done
by applying a general powerful machinery due to Wr\'obel \cite{wrobel}, which requires verification of several
technical conditions. Here the trickiest part, besides the completeness established earlier, is finding uniform
pointwise bounds for the essential measure Fourier-Bessel system and its differentiated counterpart (Lemma \ref{lem:ues}).
Still in this section, we also consider a modification of the essential measure Fourier-Bessel situation and get a result on
the associated Riesz transform (Theorem \ref{thm:Rieszm}). Finally, we comment on the natural and Lebesgue measures
Fourier-Bessel settings.
\item[(S7)]
We study heat semigroups associated with the original and differentiated Fourier-Bessel systems. We aim at obtaining maximal
theorems in the essential measure situation (Theorem \ref{thm:maxhess} and Corollary \ref{cor:map}), since these
results are needed to study Sobolev spaces in the next section. To get that, the key step is to establish a relation
between kernels of heat semigroups associated with differentiated Jacobi and Lebesgue measure Fourier-Bessel systems, where in the
former case sharp bounds of the kernel are known. This is inspired by the ideas from \cite{NoRo2}.
\item[(S8)]
We introduce Sobolev and potential spaces related to the essential measure Fourier-Bessel setting.
As the main result here, we prove that the spaces just mentioned are isomorphic as Banach spaces (Theorem \ref{thm:isoCal}),
as they should be according to a general philosophy. This isomorphism is of importance since, in particular, it provides
an effective way of verification whether a concrete function belongs to the potential space.
A noticeable technical obstacle in this section is showing that the subspace
spanned by the Fourier-Bessel system is dense in the Sobolev space (Proposition~\ref{prop:Sdense}). We finish
this section with some crucial observations concerning natural and Lebesgue measures Fourier-Bessel situations.
Namely, we infer that previously used derivatives are not suitable for defining the associated Sobolev spaces.
On the other hand, we indicate that our new derivatives are appropriate choices for this purpose.
\end{itemize}
For the reader's convenience, as an Appendix at the end
we provide a table summarizing basic notation in the contexts appearing in this paper.

\subsection*{Acknowledgment.}
The authors thank the anonymous referee for insightful comments and suggestions that helped them to improve the presentation.

\section{Preliminaries} \label{sec:pre}

In this section we first briefly describe the general notation used in the paper.
Then we present for further reference various facts and formulae connected with the Bessel function $J_{\nu}$.

\subsection{Notation} \label{ssec:not}

Throughout the paper we use fairly standard notation. In particular, $\mathbb{N} = \{0,1,2,\ldots\}$ is
the set of natural numbers, and $\mathbb{N}_+ = \{1,2,\ldots\}$ is its positive subset.
As usual, for $1 \le p \le \infty$, $p'$ denotes its conjugate exponent, $1/p+1/p' = 1$.
All $L^p$ spaces in this paper are over the interval $(0,1)$ or a product of such intervals,
so in our notation we skip $(0,1)$ or $(0,1)^d$ and indicate only the relevant
measure. The standard inner product in $L^2(d\mu)$ is denoted by $\langle \cdot, \cdot\rangle_{d\mu}$, where we
shall skip the subscript when $d\mu=dx$. We will also write $\langle f, g \rangle_{d\mu}$ for functions that
are not necessarily in $L^2(d\mu)$, but the integral defining the pairing converges.

We write $X \lesssim Y$ to indicate that $X \le CY$ with a positive constant $C$ independent
of significant quantities. We shall write $X \simeq Y$ when simultaneously $X \lesssim Y$ and $Y \lesssim X$.

\subsection{Facts and formulae concerning the Bessel function $J_{\nu}$} \label{sec:Bes}

The most standard Bessel function $J_{\nu}$ of the first kind and order $\nu$ (see e.g.\ \cite{watson})
will appear frequently in what follows. In all the occurrences below $\nu > -1$ and $J_{\nu}$ is considered
as a function on the positive half-line $(0,\infty)$. It is well known that $J_{\nu}$ is smooth on $(0,\infty)$
and has infinitely many oscillations and zeros there. We denote the successive positive zeros of $J_{\nu}$
by $\lambda_{n,\nu}$, $n \in \mathbb{N}_+$. In general, $J_{\nu}$ cannot be expressed directly via elementary
functions; this is possible if and only if $\nu$ is half-integer, but not integer. In particular,
\begin{equation} \label{specJ}
J_{-1/2}(z) = \sqrt{\frac{2}{\pi z}} \cos z, \qquad J_{1/2}(z) = \sqrt{\frac{2}{\pi z}} \sin z,
\qquad J_{3/2}(z) = \sqrt{\frac{2}{\pi z}} \frac{\sin z - z \cos z}{z}.
\end{equation}
The function $J_\nu$ has the series expansion
\begin{equation*}
J_\nu(z)=\Big(\frac{z}{2}\Big)^\nu\sum_{k=0}^{\infty}\frac{(-\frac{1}{4}z^2)^k}{k!\Gamma(\nu+k+1)},
\end{equation*}
from which it follows that 
\begin{equation}\label{J0}
J_\nu(z)=\frac{1}{\Gamma(\nu+1)}\Big(\frac{z}{2}\Big)^\nu+\mathcal{O}\big(z^{\nu+2}\big),\qquad z\rightarrow 0^+.
\end{equation}
For large arguments one has
\begin{equation} \label{Jinf}
J_{\nu}(z) = \mathcal{O}\big( z^{-1/2} \big), \qquad z \to \infty.
\end{equation}
In particular, $J_{\nu}$ is bounded on each half-line separated from zero and contained in $(0,\infty)$.
A more precise asymptotic for large arguments is given by, cf.\ \cite[Chapter VII, Section 7$\cdot$21]{watson},
\begin{equation} \label{Jinfp}
J_{\nu}(z) = \sqrt{\frac{2}{\pi z}} \bigg( \cos\Big( z - \frac{2\nu+1}4 \pi \Big) + \mathcal{O}\big( z^{-1} \big) \bigg),
\qquad z \to \infty.
\end{equation}

Given $n \ge 1$, the behavior of $J_{\nu}(\lambda_{n,\nu}x)$ close to $x=1$ is
\begin{equation}\label{J}
J_\nu(\lambda_{n,\nu} x)=(-1)^{n+1} |J_{\nu+1}(\lambda_{n,\nu})| \lambda_{n,\nu} (1-x) +\mathcal{O}\big((1-x)^2\big),
\qquad x\rightarrow 1;
\end{equation}
this follows from Taylor's formula and \eqref{diff_J} below.
Note that $\sign J_{\nu+1}(\lambda_{n,\nu}) = (-1)^{n+1}$.
In fact, the zeros of $J_{\nu+1}$ interlace those of $J_{\nu}$,
see e.g.\ \cite[Chapter XV, Section 15$\cdot$22]{watson},
\begin{equation} \label{zer_int}
0 < \lambda_{1,\nu} < \lambda_{1,\nu+1} < \lambda_{2,\nu} < \lambda_{2,\nu+1} < \lambda_{3,\nu} < \ldots .
\end{equation}
It is known that
\begin{equation} \label{eigvas}
\lambda_{n,\nu} = \pi n + D_{\nu} + \mathcal{O}\big(n^{-1}\big), \qquad n \ge 1,
\end{equation}
where $D_{\nu} = \pi (2\nu-1)/4$. More precise asymptotics of the zeros are also available, see
\cite[Chapter XV, Section 15$\cdot$53]{watson}; see also \cite[p.\,4445]{CiSt1}.

One has the differentiation formulae
\begin{align}
\frac{d}{dz} J_{\nu}(z) & = \frac{\nu}{z} J_{\nu}(z) - J_{\nu+1}(z), \label{diff_J} \\
\frac{d}{dz} J_{\nu+1}(z) & = J_{\nu}(z) - \frac{\nu+1}{z} J_{\nu+1}(z). \label{diff_J_b}
\end{align}
Notice that \eqref{diff_J} implies
\begin{equation} \label{diff_J_c}
\frac{d}{dz} \big( z^{-\nu}J_{\nu}(z)\big) = -z^{-\nu} J_{\nu+1}(z).
\end{equation}

An important role in our developments is played by the functions
$$
R^{\nu}_n(x) := \lambda_{n,\nu} \frac{J_{\nu+1}(\lambda_{n,\nu}x)}{J_{\nu}(\lambda_{n,\nu}x)}, \qquad x \in (0,1), \quad
	n \ge 1,
$$
and particularly by the initial one $R^{\nu}_1 =: R^{\nu}$ (for the sake of brevity from now on we omit the subscript in $R^{\nu}_1$).
Observe that by \eqref{J0} and \eqref{J}
\begin{equation} \label{Ras}
R^{\nu}(x) \simeq \frac{x}{1-x}, \qquad x \in (0,1).
\end{equation}
Using the Mittag-Leffler expansion (see \cite[p.\,498]{watson})
\begin{equation} \label{MLe}
\frac{J_{\nu+1}(z)}{J_{\nu}(z)} = \sum_{k=1}^{\infty} \frac{2z}{\lambda_{k,\nu}^2 - z^2}
\end{equation}
of the Bessel ratio we see that
\begin{equation} \label{R1}
R^{\nu}(x) = \frac{1}{1-x} - \frac{1}{1+x} + S^{\nu}(x),
\end{equation}
where $S^{\nu}$ is a smooth bounded function on $(0,1)$ given by $S^{\nu} := S_1^{\nu}$, where
\begin{equation}\label{eq:S}
S^{\nu}_n(x) := \sum_{k\ge 1, k\neq n}^{\infty}  \frac{2x}{(\frac{\lambda_{k,\nu}}{\lambda_{n,\nu}})^2 - x^2}, \qquad n \ge 1.
\end{equation}
Clearly, $S^{\nu}(x) \simeq x$ as $x \to 0^+$, in particular $S^{\nu}(0) = 0$. Furthermore, $S^{\nu}(1) = \nu+1$, which
follows from Calogero's formula \cite{Cal}
\begin{equation} \label{Calo}
\sum_{k \ge 1, \; k \neq n} \frac{2 \lambda_{n,\nu}^2}{\lambda_{k,\nu}^2-\lambda_{n,\nu}^2} = \nu+1, \qquad n \in \mathbb{N}_+.
\end{equation}

We will need the following result concerning boundary behavior of the difference $R^{\nu}-R^{\nu}_n$.
\begin{lemma} \label{lem:diffR}
Let $\nu > -1$ and $n \ge 2$ be fixed. Then
\begin{align*}
R^{\nu}(x) - R_n^{\nu}(x) \simeq
	\begin{cases}
		-x, & x \to 0^+, \\
		1-x, & x \to 1^-.
	\end{cases}
\end{align*}
\end{lemma}

\begin{proof}
We first use the Mittag-Leffler expansion \eqref{MLe} to extend \eqref{R1} to arbitrary $n \ge 1$, getting
\begin{equation} \label{RtoS}
R_n^{\nu}(x) = \frac{1}{1-x} - \frac{1}{1+x} + S_n^{\nu}(x).
\end{equation}
Then we can write
\begin{align*}
\lim_{x \to 0^+} \frac{R^{\nu}(x) - R_n^{\nu}(x)}{x} & =
	2 \sum_{k \ge 1, k \neq 1} \bigg(\frac{\lambda_{1,\nu}}{\lambda_{k,\nu}}\bigg)^2
- 2 \sum_{k \ge 1, k \neq n} \bigg(\frac{\lambda_{n,\nu}}{\lambda_{k,\nu}}\bigg)^2 \\
& = 2 \big( \lambda_{1,\nu}^2 - \lambda_{n,\nu}^2\big) \sum_{k \ge 1} \frac{1}{\lambda_{k,\nu}^2}
= - \frac{\lambda_{n,\nu}^2 - \lambda_{1,\nu}^2}{2\nu +2}.
\end{align*}
The series that is summed in the last identity is a particular case of Rayleigh's function, in particular it is known that,
see \cite[Chapter XV, Section 15$\cdot$51]{watson},
$$
\sum_{k \ge 1}\frac{1}{\lambda_{k,\nu}^2} = \frac{1}{4\nu+4}.
$$
From the above limit evaluation it follows that $R^{\nu}(x) - R_n^{\nu}(x) \simeq -x$ as $x \to 0^+$.

It remains to verify that $R^{\nu}(x) - R_n^{\nu}(x) \simeq 1-x$ as $x \to 1^-$.
We have
$$
\lim_{x \to 1^-} \frac{R^{\nu}(x) - R_n^{\nu}(x)}{1-x} = \lim_{x \to 1^-} \frac{S_1^{\nu}(x)- S_n^{\nu}(x)}{1-x}
= \lim_{x \to 1^-} 2x \frac{G(x)}{1-x},
$$
where
$$
G(x) := \sum_{k\ge 1, k \neq 1} \frac{1}{(\frac{\lambda_{k,\nu}}{\lambda_{1,\nu}})^2 - x^2}
	- \sum_{k \ge 1, k \neq n} \frac{1}{(\frac{\lambda_{k,\nu}}{\lambda_{n,\nu}})^2 - x^2}.
$$
Clearly, $G(x)$ is differentiable at $x = 1$ and, by Calogero's formula \eqref{Calo}, $G(1) = 0$.
Further,
$$
\lim_{x \to 1^-} G'(x) = G'(1) = \sum_{k \ge 1, k \neq 1} \frac{2\lambda_{1,\nu}^4}{(\lambda_{k,\nu}^2-\lambda_{1,\nu}^2)^2}
	- \sum_{k \ge 1, k \neq n} \frac{2\lambda_{n,\nu}^4}{(\lambda_{k,\nu}^2-\lambda_{n,\nu}^2)^2}
	= \frac{\lambda_{1,\nu}^2 - \lambda_{n,\nu}^2}{6}.
$$
The last identity is a consequence of another formula by Calogero \cite{Cal2}, namely
\begin{equation} \label{Calo2}
\sum_{k \ge 1, k \neq n} \frac{\lambda_{n,\nu}^4}{(\lambda_{n,\nu}^2-\lambda_{k,\nu}^2)^2}
	= \frac{\lambda_{n,\nu}^2 - (\nu+1)(\nu+5)}{12}, \qquad n \in \mathbb{N}_+.
\end{equation}
Using now L'H\^opital's rule we can evaluate the relevant limit,
$$
\lim_{x \to 1^-} \frac{R^{\nu}(x) - R_n^{\nu}(x)}{1-x} = 2\lim_{x \to 1^-} \frac{G'(x)}{-1}
	= \frac{\lambda_{n,\nu}^2-\lambda_{1,\nu}^2}{3}.
$$
The conclusion follows.
\end{proof}

\begin{remark} \label{rem:asdif}
The proof of Lemma \ref{lem:diffR} gives in fact more precise asymptotics
$$
R^{\nu}(x) - R^{\nu}_{n}(x) =
	\begin{cases}
		- \frac{1}{2\nu+2} (\lambda_{n,\nu}^2 - \lambda_{1,\nu}^2) x + \mathcal{O}(x^2), & x \to 0^+, \\
		\frac{1}{3}(\lambda_{n,\nu}^2 - \lambda_{1,\nu}^2) (1-x) + \mathcal{O}((1-x)^2), & x \to 1^-.
	\end{cases}
$$
\end{remark}

\section{Fourier-Bessel and Jacobi contexts} \label{sec:prel}

In this paper we consider mainly three contexts of discrete Fourier-Bessel expansions depending on the choice
of the associated measure, and also, for technical reasons, an auxiliary context of Jacobi expansions.
We always assume that parameters of type $\nu,\alpha,\beta > -1$. Traditionally, the Fourier-Bessel
systems we deal with are enumerated by the elements of $\mathbb{N}_+$ rather than $\mathbb{N}$.

\subsection{Natural measure Fourier-Bessel setting} \label{ssec:FBnat}

Let
$$
\phi_n^{\nu}(x) = \frac{\sqrt{2}}{|J_{\nu+1}(\lambda_{n,\nu})|} x^{-\nu} J_{\nu}(\lambda_{n,\nu}x), \qquad
		x \in (0,1), \quad n \ge 1.
$$
The Fourier-Bessel system $\{\phi_n^{\nu} : n \ge 1\}$ is an orthonormal basis in $L^2(d\mu_{\nu})$,
where
$$
d\mu_{\nu}(x) = x^{2\nu+1}\, dx, \qquad x \in (0,1).
$$
Moreover, these are eigenfunctions of the Bessel operator
$$
\mathcal{L}_{\nu} = - \frac{d^2}{dx^2} - \frac{2\nu+1}{x} \frac{d}{dx}
$$
considered on $(0,1)$, $\mathcal{L}_{\nu} \phi_n^{\nu} = \lambda_{n,\nu}^2 \phi_n^{\nu}$, $n \ge 1$.

When $\nu = d/2-1$, $d \ge 1$, $\mathcal{L}_{\nu}$ is the radial part of the Euclidean Laplacian in
$\mathbb{R}^d$, and expansions with respect to the system $\{\phi_n^{\nu}\}$ correspond to radial analysis in the
$d$-dimensional unit ball in $\mathbb{R}^d$; see e.g.\ \cite{NoRo1} for more details.

\subsection{Lebesgue measure Fourier-Bessel setting} \label{ssec:FBleb}

This context arises from the previous one by transforming it to the Lebesgue measure situation via the mapping
$$
U_{\nu}f(x) = x^{\nu+1/2} f(x).
$$
This leads to the Fourier-Bessel system
\begin{equation*} 
\psi_n^{\nu} = U_{\nu} \phi_n^{\nu}, \qquad n \ge 1,
\end{equation*}
which forms an orthonormal basis in $L^2(dx)$. The corresponding Laplacian is$^{\dag}$
$$
\mathbb{L}_{\nu} = U_{\nu} \mathcal{L}_{\nu} U_{\nu}^{-1} = -\frac{d^2}{dx^2} + \frac{(\nu+1/2)(\nu-1/2)}{x^2} 
$$
on $(0,1)$, and one has $\mathbb{L}_{\nu} \psi_n^{\nu} = \lambda_{n,\nu}^2 \psi_n^{\nu}$, $n \ge 1$.
\footnotetext{\noindent $\dag$ Note that $\mathbb{L}_{\nu}$, as well as $\mathcal{L}_{\nu}$, $\mathfrak{L}_{\nu}$ and
$\mathbb{J}_{\alpha,\beta}$, can be written in a divergence form. For instance, in the case of $\mathbb{L}_{\nu}$ one has
$$
\mathbb{L}_{\nu}f =
- \frac{1}{\psi_1^{\nu}}\frac{d}{dx}\bigg(\big(\psi_1^{\nu}\big)^2 \frac{d}{dx}\bigg(\frac{1}{\psi_1^{\nu}}f\bigg)\bigg).
$$
}

\subsection{Essential measure Fourier-Bessel setting} \label{FBneu}

We now introduce another Fourier-Bessel context that has not appeared in the literature yet.
Here we choose the reference measure so that the bottom eigenfunction is constant. Then the
orthogonal system satisfies the Neumann boundary condition
at the right endpoint of $(0,1)$ (notice that the systems $\{\phi_n^{\nu}\}$ and $\{\psi_n^{\nu}\}$
satisfy the Dirichlet condition at this endpoint) and, moreover, the associated
derivative is the standard one (this will be discussed in detail in Section \ref{sec:der} below). Thus let
$$
\widetilde{U}_{\nu}f = f / \phi_1^{\nu}
$$
and consider the system
\begin{equation*} 
\varphi_n^{\nu} = \widetilde{U}_{\nu} \phi_n^{\nu}, \qquad n \ge 1,
\end{equation*}
which is an orthonormal basis in $L^2(d\eta_{\nu})$, with
$$
d\eta_{\nu}(x) = x^{2\nu+1}\big[\phi_1^{\nu}(x)\big]^2 \, dx, \qquad x \in (0,1).
$$
The corresponding Laplacian now is
$$
\mathfrak{L}_{\nu} = \widetilde{U}_{\nu} \mathcal{L}_{\nu} \widetilde{U}_{\nu}^{-1}
	= -\frac{d^2}{dx^2} - \bigg( \frac{2\nu+1}{x} - 2R^{\nu}(x)\bigg) \frac{d}{dx} + \lambda_{1,\nu}^2
$$
considered on $(0,1)$, being $\mathfrak{L}_{\nu} \varphi_n^{\nu} = \lambda_{n,\nu}^2 \varphi_n^{\nu}$, $n \ge 1$.
Note that the measure $d\eta_{\nu}$ is doubling, and $(0,1)$ equipped with $d\eta_{\nu}$ and the ordinary distance
is a space of homogeneous type in the sense of Coifman and Weiss \cite{CW}.

The setting just defined is sub-probabilistic in the sense that the semigroup generated by a natural in this context
self-adjoint extension of $\mathfrak{L}_{\nu}$ is sub-Markovian. Nevertheless, a simple modification leads to a genuine
probabilistic setting. Indeed, consider
$$
\mathfrak{L}_{\nu}^M = \mathfrak{L}_{\nu} - \lambda^2_{1,\nu}
	= -\frac{d^2}{dx^2} - \bigg( \frac{2\nu+1}{x} - 2R^{\nu}(x)\bigg) \frac{d}{dx}.
$$
Then the associated orthonormal basis in $L^2(d\eta_{\nu})$ is still $\{\varphi_n^{\nu} : n \ge 1\}$,
but the eigenvalues are shifted,
$$
\mathfrak{L}_{\nu}^M \varphi_n^{\nu} = \big( \lambda_{n,\nu}^2-\lambda_{1,\nu}^2\big) \varphi_n^{\nu},
\qquad n \ge 1.
$$
In particular, the bottom eigenvalue is $0$.
The operator $\mathfrak{L}_{\nu}^M$ admits a natural self-adjoint extension in $L^2(d\eta_{\nu})$ defined
spectrally in terms of the orthonormal system, and the associated semigroup is Markovian.

The framework of $\mathfrak{L}_{\nu}^M$ just introduced may be regarded as a Fourier-Bessel analogue of
the classic probabilistic settings related to expansions into classical Hermite, Laguerre and Jacobi polynomials.
We remark that the natural measure Fourier-Bessel setting is strictly sub-probabilistic in the sense that the semigroup
generated by a natural in that context self-adjoint extension of $\mathcal{L}_{\nu}$ is sub-Markovian, but not Markovian.
See e.g.\ the discussion in \cite[Section 4.3]{NoRo1}; see also \cite{MSZ}.
The case of the Lebesgue measure Fourier-Bessel setting is more complicated.
We strongly believe that the semigroup generated by a natural self-adjoint extension of $\mathbb{L}_{\nu}$ is
sub-Markovian only for some $\nu > -1$, but never Markovian; see again \cite[Section 4.3]{NoRo1} for some comments in this direction.
However, a complete detailed treatment of these questions seems to require an analysis which is far beyond the scope of this paper.

\begin{remark} \label{rem:CZ}
The essential measure Fourier-Bessel context appears to be an optimal environment among the other Fourier-Bessel situations
for a direct application of the Calder\'on-Zygmund operator theory in the following sense.
Note that the three Fourier-Bessel settings considered in this section
are interconnected via the mappings $U_{\nu}$ and $\widetilde{U}_{\nu}$, and similarly are connected harmonic analysis
operators in these contexts. In particular, any weighted $L^p$ inequality for such an operator, e.g.\ Riesz transform,
in one of the contexts can be transferred to the corresponding weighted $L^p$ bounds in the other contexts,
with suitable modifications of the weight.
The standard Calder\'on-Zygmund theory on spaces of homogeneous type can potentially be applied in each of the three
situations, in each case producing weighted results in all the situations via the aforementioned transference.
The point is that a successful application in the essential measure framework delivers the largest classes of
weights, hence the most complete results, in all the three situations, compared to the two other applications.
\end{remark}

Certain modification of the essential measure Fourier-Bessel setting, which might be of interest, is considered in
Section \ref{ssec:mod_e}.

\subsection{Jacobi trigonometric Lebesgue measure setting scaled to $(0,1)$}\label{sec:jac}

This context will serve us as a source of both intuition and results to be transferred to the
Lebesgue measure Fourier-Bessel framework. This idea originates in \cite{NoRo2}, it was also explored in \cite{NoRo3};
the latter article reveals its further usefulness.

We consider the functions
$$
\Phi_k^{\alpha,\beta}(x) = c_k^{\alpha,\beta} \Big(\sin \frac{\pi x}2 \Big)^{\alpha+1/2}
	\Big( \cos \frac{\pi x}2 \Big)^{\beta+1/2} P_k^{\alpha,\beta}(\cos \pi x), \qquad x \in (0,1), \quad k \ge 0,
$$
where $\alpha, \beta > -1$, $P_k^{\alpha,\beta}$ are the classical Jacobi polynomials (cf.\ \cite{Sz}) and $c_k^{\alpha,\beta}>0$ are
suitable (explicitly known, see e.g.\ \cite{NoRo2}) normalizing constants.
The system $\{\Phi_k^{\alpha,\beta}:k \ge 0\}$ is an
orthonormal basis in $L^2(dx)$ consisting of eigenfunctions of the Jacobi Laplacian
$$
\mathbb{J}_{\alpha,\beta} = -\frac{d^2}{dx^2} + \frac{\pi^2 (\alpha+1/2)(\alpha-1/2)}{4\sin^2 \frac{\pi x}2}
	+ \frac{\pi^2(\beta+1/2)(\beta-1/2)}{4\cos^2\frac{\pi x}2};
$$
one has $\mathbb{J}_{\alpha,\beta} \Phi_k^{\alpha,\beta} = \pi^2 (k+\frac{\alpha+\beta+1}2)^2 \Phi_k^{\alpha,\beta}$, $k \ge 0$.

Note that (see e.g.\ \cite[Chapter IV]{Sz} for the relevant formulae)
\begin{align*}
\Phi_k^{-1/2,1/2}(x) & = \sqrt{2} \cos\big( \pi(k+1/2)x \big), \qquad k \ge 0, \\
\Phi_k^{1/2,1/2}(x) & = \sqrt{2} \sin\big( \pi(k+1)x \big), \qquad k \ge 0.
\end{align*}
Further, $\mathbb{J}_{\pm 1/2, 1/2} = -\frac{d^2}{dx^2}$. On the other hand, see \eqref{specJ},
\begin{align*}
\psi_n^{-1/2}(x) & = \sqrt{2} \cos\big( \pi(n-1/2)x \big), \qquad n \ge 1, \\
\psi_n^{1/2}(x) & = \sqrt{2} \sin\big( \pi n x \big), \qquad n \ge 1,
\end{align*}
and $\mathbb{L}_{\pm 1/2} = -\frac{d^2}{dx^2}$. Thus, in the cases $\nu = \pm 1/2$, the Lebesgue measure Fourier-Bessel
setting coincides with the Jacobi setting with $(\alpha,\beta) = (\pm 1/2, 1/2)$, respectively. That means that
the systems, the Laplacians and the eigenvalues are exactly the same (only enumeration is slightly different).

Finally, notice that for $\nu=-1/2$ the natural and Lebesgue measures Fourier-Bessel settings coincide,
both being the same as the Jacobi one with $(\alpha,\beta) = (-1/2,1/2)$.

\section{Derivatives} \label{sec:der}

A crucial aspect of harmonic analysis related to Fourier-Bessel expansions is a proper choice of the associated derivatives.
These should be, in principle, first order differential operators decomposing the corresponding Laplacians.
Definitions of several important objects depend heavily on that choice, and this in particular
pertains to Fourier-Bessel Riesz transforms and Sobolev spaces.

However, the Fourier-Bessel derivatives used so far are not really adequate.
In the natural and Lebesgue measures Fourier-Bessel settings they are taken from the decompositions
$$
\mathcal{L}_{\nu} = \mathfrak{d}_{\nu}^*\mathfrak{d}_{\nu}, \qquad \mathbb{L}_{\nu} = \mathfrak{D}_{\nu}^* \mathfrak{D}_{\nu},
$$
where $\mathfrak{d}_{\nu} = \frac{d}{dx}$, $\mathfrak{D}_{\nu} = \frac{d}{dx}-\frac{\nu+1/2}{x}$, and $\mathfrak{d}_{\nu}^*$
and $\mathfrak{D}_{\nu}^*$ are the formal adjoints in $L^2(d\mu_{\nu})$ and $L^2(dx)$, respectively.
Such derivatives are appropriate in the theory of continuous Fourier-Bessel expansions on the whole half-line $(0,\infty)$,
in the context of the modified Hankel transform, where they originate from; then the related differential Laplacians are again
$\mathcal{L}_{\nu}$ and $\mathbb{L}_{\nu}$, but considered on $(0,\infty)$, and the above decompositions hold on $(0,\infty)$
(for more details, see e.g.\ \cite{BHNV} or \cite{CaSz0} and references therein).
But $\mathfrak{d}_{\nu}$ and $\mathfrak{D}_{\nu}$ are no longer good when restricting to $(0,1)$ and considering discrete
Fourier-Bessel expansions since, roughly speaking, they do not reflect presence of the right boundary of $(0,1)$
and the Dirichlet boundary condition there. Consequently, these derivatives do not allow to develop
(or develop properly) important harmonic
analysis aspects of discrete Fourier-Bessel expansions like, for instance, the theory of the associated Sobolev spaces.
Actually, it seems that for this very reason such Sobolev spaces have not been studied effectively yet.
The problem also interferes with existing
works dealing with Riesz transforms and conjugacy for Fourier-Bessel expansions, for instance \cite[Section 18]{MuSt},
\cite{CiSt3,CiRo4} as well as \cite{NoSt1} and indirectly the related paper \cite{NoSt2}.

We think that the derivatives that philosophically and practically match better the two Fourier-Bessel contexts in question
should factorize, in a fashion indicated above, the Laplacians shifted by the bottom eigenvalue, so that spectra of the
operators being decomposed start at the origin. This precisely happens in most, if not all, so far studied classical
frameworks related to classical discrete orthogonal expansions, see e.g.\ \cite[Sections 7.1--7.7]{NoSt1}, except
for Fourier-Bessel expansions \cite[Section 7.8]{NoSt1}$^{\ddag}$. Apparently, the latter case still requires a deeper understanding.
\footnotetext{\noindent $\ddag$ Note that the Riesz transform associated with $\mathcal{L}_{\nu}$
that is postulated in \cite[Section 7.8]{NoSt1}
coincides with the conjugacy mapping introduced by Muckenhoupt and Stein \cite[Section 18]{MuSt}, contrary to an
erroneous claim finishing \cite[Section 7.8]{NoSt1}. In fact, the conjugate mapping of Muckenhoupt and Stein is incorrectly
written in terms of the system $\phi_k^{\nu}$ in \cite[p.\,698]{NoSt1}.}

In view of what was just said, we are looking for derivatives (first order differential operators)
$\delta_\nu$ and $\mathbb{D}_{\nu}$ giving the decompositions
$$
\mathcal{L}_{\nu} = \lambda^2_{1,\nu} + {\delta}^*_{\nu} {\delta}_{\nu}, \qquad
	\mathbb{L}_{\nu} = \lambda_{1,\nu}^2 + {\mathbb{D}}_{\nu}^* {\mathbb{D}}_{\nu}.
$$
This immediately boils down to solving a Riccati-type ODE, which is not that trivial.
Thus we proceed in a more elementary way.
First we switch to the essential measure Fourier-Bessel setting, where the analogous question is easy, and then transfer
the result to the other settings via the mappings $U_{\nu}$ and~$\widetilde{U}_{\nu}$.

\subsection{Essential measure Fourier-Bessel setting} \label{sec:ess}

We are looking for the first order differential operator $d_{\nu}$ satisfying
$$
\mathfrak{L}_{\nu} = \lambda_{1,\nu}^2 + d_{\nu}^* d_{\nu} \quad \textrm{and} \quad
\mathfrak{L}^M_{\nu} = d_{\nu}^* d_{\nu},
$$
the formal adjoint being taken in $L^2(d\eta_{\nu})$. Since $d_{\nu}$ annihilates the bottom eigenfunction, which
is identically $1$, it follows that $d_{\nu}$ must be the standard derivative. This is because, see \cite[Section 2]{NoSt1},
$$
\big\langle d_{\nu} \varphi_1^{\nu}, d_{\nu} \varphi^{\nu}_1 \big\rangle_{d\eta_{\nu}} = 
\big\langle d^*_{\nu} d_{\nu} \varphi_1^{\nu}, \varphi_1^{\nu} \big\rangle_{d\eta_{\nu}} =
\big\langle \mathfrak{L}^M_{\nu} \varphi_1^{\nu}, \varphi_1^{\nu} \big\rangle_{d\eta_{\nu}} = 0.
$$
Thus
$$
d_{\nu} = \frac{d}{dx}, \qquad d_{\nu}^* = -\frac{d}{dx} - \frac{2\nu+1}{x} + 2R^{\nu}(x).
$$
Here we keep the subscript even though $d_{\nu}$ does not depend on $\nu$.
A simple computation shows that
\begin{equation} \label{eq:us1}
d_{\nu}\varphi_n^{\nu} = (R^{\nu}-R_n^{\nu}) \varphi_n^{\nu}, \qquad n \ge 1.
\end{equation}
Note that 
\begin{equation}\label{eq:Rder}
R_n^{\nu} = - (\phi_n^{\nu})'/\phi_n^{\nu},
\end{equation}
 see \eqref{diff_J_c}.

\subsection{Natural measure Fourier-Bessel setting} \label{sec:FBdernat}

To find $\delta_{\nu}$ in the decomposition
$$
\mathcal{L}_{\nu} = \lambda_{1,\nu}^2 + \delta_{\nu}^* \delta_{\nu}
$$
observe that  $\delta_{\nu} = \widetilde{U}^{-1}_{\nu} d_{\nu} \widetilde{U}_{\nu}$. This leads to the expressions
$$
\delta_{\nu} = \frac{d}{dx} + R^{\nu}(x), \qquad \delta_{\nu}^* = -\frac{d}{dx} - \frac{2\nu+1}x + R^{\nu}(x),
$$
the formal adjoint $\delta_{\nu}^*$ being taken in $L^2(d\mu_{\nu})$. Furthermore,
\begin{equation} \label{eq:us2}
\delta_{\nu} \phi_n^{\nu} = (R^{\nu}-R_n^{\nu}) \phi_n^{\nu}, \qquad n \ge 1.
\end{equation}

\subsection{Lebesgue measure Fourier-Bessel setting} \label{sec:FBderleb}

By the relation $\mathbb{D}_{\nu} = U_{\nu} \delta_{\nu} U_{\nu}^{-1}$ we have
$$
\mathbb{L}_{\nu} = \lambda_{1,\nu}^2 + \mathbb{D}_{\nu}^*\mathbb{D}_{\nu}
$$
with
\begin{equation}\label{eq:derleb}
\mathbb{D}_{\nu} = \frac{d}{dx} - \frac{\nu+1/2}x + R^{\nu}(x), \qquad
\mathbb{D}_{\nu}^* = -\frac{d}{dx} - \frac{\nu+1/2}x + R^{\nu}(x),
\end{equation}
the formal adjoint being taken in $L^2(dx)$. Further,
\begin{equation}\label{eq:useful}
\mathbb{D}_{\nu} \psi_n^{\nu} = (R^{\nu}-R_n^{\nu}) \psi_n^{\nu}, \qquad n \ge 1.
\end{equation}

\subsection{Jacobi trigonometric Lebesgue measure setting scaled to $(0,1)$} \label{sec:jacder}

The Jacobi operator admits the decomposition (cf.\ \cite[Section 7.7]{NoSt1})
$$
\mathbb{J}_{\alpha,\beta} = \pi^2 \bigg(\frac{\alpha+\beta+1}2\bigg)^2 + D_{\alpha,\beta}^* D_{\alpha,\beta},
$$
where
\begin{align*}
D_{\alpha,\beta} & = \frac{d}{dx} - \pi \frac{2\alpha+1}4 \cot\frac{\pi x}2 + \pi \frac{2\beta+1}4 \tan\frac{\pi x}2, \\
D^*_{\alpha,\beta} & = -\frac{d}{dx} - \pi \frac{2\alpha+1}4 \cot\frac{\pi x}2 + \pi \frac{2\beta+1}4 \tan\frac{\pi x}2.
\end{align*}
Here the adjoint is taken in $L^2(dx)$. One has
\begin{equation} \label{jdfe}
D_{\alpha,\beta}\Phi_k^{\alpha,\beta} = - \pi \sqrt{k(k+\alpha+\beta+1)} \Phi_{k-1}^{\alpha+1,\beta+1}, \qquad k \ge 1,
\end{equation}
and $D_{\alpha,\beta}\Phi_0^{\alpha,\beta} \equiv 0$.

Note that the Jacobi derivatives coincide with the Lebesgue measure Fourier-Bessel setting derivatives when
$(\alpha,\beta) = (\pm 1/2, 1/2)$. Indeed, since (see \eqref{specJ})
\begin{equation} \label{Rexp}
R^{-1/2}(x) = \frac{\pi}2 \tan\frac{\pi x}2, \qquad R^{1/2}(x) = \frac{1}{x} - \pi \cot(\pi x),
\end{equation}
we have
\begin{align*}
\mathbb{D}_{-1/2} & = \frac{d}{dx} + \frac{\pi}2 \tan\frac{\pi x}2 = D_{-1/2,1/2}, \\
\mathbb{D}_{1/2} & = \frac{d}{dx} - \pi \cot(\pi x) = \frac{d}{dx} - \frac{\pi}2 \cot\frac{\pi x}2 + \frac{\pi}2\tan\frac{\pi x}2
	= D_{1/2,1/2}.
\end{align*}

\section{Differentiated systems} \label{sec:diff}

The main aim of this section is to show that the differentiated systems
$$
\big\{d_{\nu}\varphi_n^{\nu} : n \ge 2\big\}, \qquad \big\{\delta_{\nu}\phi_n^{\nu} : n \ge 2\big\}, \qquad
\big\{\mathbb{D}_{\nu}\psi_n^{\nu} : n \ge 2\big\}
$$
are orthogonal bases in the corresponding $L^2$ spaces.
Recall that $d_{\nu}\varphi_{1}^{\nu} = \delta_{\nu} \phi_1^{\nu} = \mathbb{D}_{\nu}\psi_1^{\nu} = 0$.

Let us start with noting basic asymptotics of functions belonging to the differentiated systems.
Notice that for the original systems we have, for each $\nu > -1$ and $n \ge 1$ fixed,
\begin{align}
\varphi_n^{\nu}(x) & \simeq
	\begin{cases}
		1, & x \to 0^+, \\
		(-1)^{n+1}, & x \to 1^-,
	\end{cases} \label{asy1} \\
\phi_n^{\nu}(x) & \simeq
	\begin{cases}
		1, & x \to 0^+, \\
		(-1)^{n+1}(1-x), & x \to 1^-,
	\end{cases} \label{asy2}  \\
\psi_n^{\nu}(x) & \simeq
	\begin{cases}
		x^{\nu+1/2}, & x \to 0^+, \\
		(-1)^{n+1}(1-x), & x \to 1^-,
	\end{cases} \label{asy3}
\end{align}
which is an instant consequence of \eqref{J0} and \eqref{J}. 
\begin{proposition} \label{prop:asd}
Let $\nu > -1$ and $n \ge 2$ be fixed.
Each of the functions $d_{\nu} \varphi_n^{\nu}$, $\delta_{\nu} \phi_n^{\nu}$ and $\mathbb{D}_{\nu} \psi_n^{\nu}$
is smooth on $(0,1)$. Further, one has
\begin{align*}
d_{\nu}\varphi_n^{\nu}(x) & \simeq
	\begin{cases}
		-x, & x \to 0^+, \\
		(-1)^{n+1}(1-x), & x \to 1^-,
	\end{cases} \\
\delta_{\nu}\phi_n^{\nu}(x) & \simeq
	\begin{cases}
		-x, & x \to 0^+, \\
		(-1)^{n+1}(1-x)^2, & x \to 1^-,
	\end{cases} \\
\mathbb{D}_{\nu}\psi_n^{\nu}(x) & \simeq
	\begin{cases}
		-x^{\nu+3/2}, & x \to 0^+, \\
		(-1)^{n+1}(1-x)^2, & x \to 1^-.
	\end{cases}
\end{align*}
\end{proposition}

\begin{proof}
To justify the smoothness simply invoke the relevant definitions and write \eqref{eq:us1}, \eqref{eq:us2} and \eqref{eq:useful}
in terms of the Bessel function to see that singularities of the difference $R^{\nu}-R^{\nu}_n$ are canceled out by the
Bessel zeros.
To get the asymptotics, combine \eqref{asy1}, \eqref{asy2} and \eqref{asy3} with \eqref{eq:us1}, \eqref{eq:us2},
\eqref{eq:useful}, respectively, and use Lemma \ref{lem:diffR}.
\end{proof}

\begin{remark} 
With the aid of the asymptotics from Remark \ref{rem:asdif} one can describe more precisely than
Proposition \ref{prop:asd} the endpoint behavior of functions from the differentiated systems.
We leave the details to interested readers.
\end{remark}

\begin{proposition} \label{prop:orth}
Let $\nu > -1$.
\begin{itemize}
\item[(a)] The system $\{d_{\nu}\varphi_n^{\nu} : n \ge 2\}$ is orthogonal in $L^2(d\eta_{\nu})$.
\item[(b)] The system $\{\delta_{\nu}\phi_n^{\nu} : n \ge 2\}$ is orthogonal in $L^2(d\mu_{\nu})$.
\item[(c)] The system $\{\mathbb{D}_{\nu}\psi_n^{\nu} : n \ge 2\}$ is orthogonal in $L^2(dx)$.
\end{itemize}
\end{proposition}

\begin{proof}
Orthogonality in each case is essentially a consequence of orthogonality of the original system and the decomposition of
the associated Laplacian in terms of the derivative and its formal adjoint; see \cite[Lemma 2]{NoSt1}.
To be more precise, also some technical assumptions, like \cite[(2.7) and (2.8)]{NoSt1}, must be verified.
For instance, in case of the system $\{d_{\nu}\varphi_n^{\nu}\}$ the two assumptions just mentioned read as
$$
d_{\nu}\varphi_n^{\nu} \in L^2(d\eta_{\nu}), \qquad n \ge 2,
$$
and
\begin{equation} \label{T1}
\big\langle d_{\nu}\varphi_n^{\nu}, d_{\nu} \varphi_m^{\nu}\big\rangle_{d\eta_{\nu}}
= \big\langle d_{\nu}^* d_{\nu}\varphi_n^{\nu}, \varphi_m^{\nu}\big\rangle_{d\eta_{\nu}}, \qquad n,m \ge 2,
\end{equation}
respectively.
But all this is rather straightforward with the aid of basic asymptotics for functions of the differentiated and non-differentiated
systems given in \eqref{asy1}, \eqref{asy2}, \eqref{asy3} and Proposition \ref{prop:asd}.
To show \eqref{T1} and its analogues one integrates by parts where
an additional fact needed is \eqref{Ras}. The details are left to the reader.
\end{proof}

\begin{remark} \label{rem:norm}
The differentiated systems are not orthonormal. In fact, see \cite[Lemma 2]{NoSt1},
$$
\|d_{\nu}\varphi_n^{\nu}\|_{L^2(d\eta_{\nu})} = \|\delta_{\nu}\phi_n^{\nu}\|_{L^2(d\mu_{\nu})}
= \|\mathbb{D}_{\nu}\psi_n^{\nu}\|_{L^2(dx)} = \sqrt{\lambda_{n,\nu}^2 - \lambda_{1,\nu}^2}, \qquad n \ge 1.
$$
\end{remark}

We now turn to proving completeness of the differentiated systems, which is a more complicated issue.
\begin{theorem} \label{thm:L2dense}
Let $\nu > -1$.
\begin{itemize}
\item[(a)] The system $\{d_{\nu}\varphi_n^{\nu} : n \ge 2\}$ is complete in $L^2(d\eta_{\nu})$.
\item[(b)] The system $\{\delta_{\nu}\phi_n^{\nu} : n \ge 2\}$ is complete in $L^2(d\mu_{\nu})$.
\item[(c)] The system $\{\mathbb{D}_{\nu}\psi_n^{\nu} : n \ge 2\}$ is complete in $L^2(dx)$.
\end{itemize}
\end{theorem}

It is straightforward to see that items (a)--(c) in this theorem are equivalent. So it is enough to prove one of them
and for technical reasons it is convenient to focus on (c). We will adapt the method used by Hochstadt \cite{H}.

Observe that each $\mathbb{D}_{\nu}\psi_n^{\nu}$ is an eigenfunction of the operator
\begin{align}\label{eq:DD*new}
\mathbb{D}_{\nu}\mathbb{D}_{\nu}^{*} = \Big( \frac{d}{dx}+q_{\nu}\Big)\Big( -\frac{d}{dx}+q_{\nu}\Big)
	= -\frac{d^2}{dx^2} +q'_{\nu} + q_{\nu}^2,
\end{align}
with the corresponding eigenvalue $\lambda_{n,\nu}^2-\lambda_{1,\nu}^2$;
here
$$
q_{\nu}(x) = -\frac{\nu+1/2}{x} + R^{\nu}(x),
$$
see \eqref{eq:derleb}.
Hochstadt's strategy is to represent the (formal) inverse of $\mathbb{D}_{\nu}\mathbb{D}_{\nu}^*$
as an integral operator and analyze its
properties. Then the density will follow by a functional analysis argument.

In the first step we construct a corresponding Green function. Denote by $\bm{\delta}$ the Dirac delta.
\begin{lemma} \label{lem:green}
Let $\nu > -1$.
There exists a function $K_{\nu} \colon (0,1)\times (0,1) \mapsto \mathbb{R}$ such that for each $\xi \in (0,1)$ fixed
$K_{\nu}(\cdot,\xi)$ is continuous on $(0,1)$ and the identity
\begin{equation} \label{eq:green}
\mathbb{D}_{\nu}\mathbb{D}_{\nu}^{*} K_{\nu}(\cdot,\xi) = \bm{\delta}(\cdot - \xi)
\end{equation}
holds in the sense of distributions. Moreover,
\begin{equation} \label{KL2}
\int_0^1\int_0^1 K_{\nu}^2(x,\xi)\, d\xi dx < \infty.
\end{equation}
\end{lemma}

To prove Lemma \ref{lem:green} we need to invoke a classic result from ODE theory. Recall that \emph{Wronskian} of a pair
of functions $(y_1,y_2)$ defined on some open interval is given by
$$
\mathcal{W}(y_1,y_2) = \det \left( \begin{array}{cc} y_1 & y_2 \\ y'_1 & y'_2 \end{array} \right)
	= y_1 y'_2 - y'_1 y_2.
$$
The lemma below is known as Abel's theorem, see e.g.\ \cite[Lemma 3.11]{T}.
\begin{lemma} \label{lem:abel}
Let $y_1$ and $y_2$ be any solutions of the equation
$$
y''(x) + a(x)y'(x) + b(x) y(x) = 0, \qquad x \in I,
$$
where $a$ and $b$ are continuous functions on an open interval $I$. Then
$$
\mathcal{W}(y_1,y_2)(x) = C \exp\big( - F(x) \big), \qquad x \in I,
$$
where $C$ is a constant and $F$ is a primitive function of $a$.
In particular, either $\mathcal{W}(y_1,y_2)$ is identically zero on $I$ or it has no zeros on $I$.
\end{lemma}

\begin{proof}[{Proof of Lemma \ref{lem:green}}]
We begin the construction of $K_{\nu}(x,\xi)$ with the observation that \eqref{eq:green} forces $K_{\nu}(x,\xi)$
to satisfy, for each $\xi \in (0,1)$ fixed,
$$
\big( \mathbb{D}_{\nu} \mathbb{D}_{\nu}^* K_{\nu}(\cdot,\xi) \big)(x) = 0, \qquad x \in (0,1)\setminus \{\xi\}.
$$
Therefore we shall look for $K_{\nu}(x,\xi)$ of the form
\begin{equation} \label{eq:casgr}
K_{\nu}(x,\xi) = 	\begin{cases}
										z_1(x)A(\xi), & x \in (0,\xi], \\
										z_2(x)B(\xi), & x \in [\xi,1),
									\end{cases}
\end{equation}
where $z_1$ and $z_2$ form a basis of fundamental solutions of the homogeneous equation
\begin{equation} \label{eq:homo}
\mathbb{D}_{\nu} \mathbb{D}_{\nu}^* y = \Big( - \frac{d^2}{dx^2} + q'_{\nu} + q_{\nu}^2 \Big)y = 0,
\end{equation}
and $A,B$ are functions to be determined later.

A particular solution to \eqref{eq:homo} is, cf.\ \cite[Sec.\ 2.1.9, Eq.\ 30]{handbook},
$$
y_1(x) = \exp\bigg( \int_{1/2}^x q_{\nu}(u)\, du\bigg);
$$
this fact can easily be verified directly.
Using \eqref{eq:Rder} and the identity $\psi^{\nu}_1(x) = x^{\nu+1/2}\phi_1^{\nu}(x)$ one computes
$$
y_1(x) = \exp\bigg( \int_{1/2}^x \bigg[-(\nu+1/2)\log u+\log \frac{1}{\phi_1^\nu(u)}\bigg]'\, du\bigg)
= \exp\bigg( \int_{1/2}^x \bigg[\log \frac{1}{\psi_1^{\nu}(u)}\bigg]'\, du\bigg)=\frac{1}{\psi_1^{\nu}(x)},
$$
where in the last step we have neglected a multiplicative constant.

To complete $y_1$ to a basis of solutions of \eqref{eq:homo} we employ the method known as the \emph{reduction of order}.
Assume that $y_2 = v y_1$ for some differentiable function $v$. A straightforward computation shows that
$$
\mathcal{W}(y_1,y_2) = v' y_1^2.
$$
On the other hand, by Lemma \ref{lem:abel},
$$
\mathcal{W}(y_1,y_2) = C
$$
for some constant $C$. Combining the above equations we get
$$
v' = C y_1^{-2}.
$$
Integrating this identity produces
$$
v(x) = C \int_0^x y_1^{-2}(u)\, du + D,
$$
with $D$ being another constant. Setting here $C=1$ and $D=0$ we obtain
$$
y_2 = F y_1,
$$
where
$$
F(x) = \int_0^{x} \big[ \psi_1^{\nu}(u)\big]^2\, du.
$$

We now express explicitly the auxiliary function $F$. This is not necessarily needed to prove the lemma, but it
is of general interest to have an explicit formula for $K_{\nu}(x,\xi)$ in the end. Using the explicit form
of $\psi_1^{\nu}$ we see that
$$
F(x) = \frac{2}{J_{\nu+1}^2(\lambda_{1,\nu})} \int_0^x u J_{\nu}^2(\lambda_{1,\nu}u)\, du.
$$
By means of a direct differentiation and formulae \eqref{diff_J} and \eqref{diff_J_b} one checks that
$$
2\int z J_{\nu}^2(z)\, dz = z^2 J_{\nu}^2(z) - 2\nu z J_{\nu}(z) J_{\nu+1}(z) + z^2 J_{\nu+1}^2(z) + \mathcal{C},
$$
where $\mathcal{C}$ is a constant of integration.
In view of \eqref{J0}, the above primitive function has value $\mathcal{C}$ at $z=0$, thus we get
\begin{align*}
F(x) & = \frac{2}{\lambda_{1,\nu}^2 J_{\nu+1}^2(\lambda_{1,\nu})} \int_0^{\lambda_{1,\nu}x} z J_{\nu}^2(z)\, dz \\
	& = \frac{1}{J_{\nu+1}^2(\lambda_{1,\nu})} \Big[ x^2 J_{\nu}^2(\lambda_{1,\nu}x) - \frac{2\nu x}{\lambda_{1,\nu}}
		J_{\nu}(\lambda_{1,\nu}x) J_{\nu+1}(\lambda_{1,\nu}x) + x^2 J_{\nu+1}^2(\lambda_{1,\nu}x) \Big].
\end{align*}
Observe that $F(1)=1$ (since $J_\nu(\lambda_{1, \nu})=0$) and, obviously, $F(0)=0$.

We now consider a modified basis of solutions
\begin{align*}
z_1 & = y_2, \\
z_2 & = y_1 - y_2.
\end{align*}
Observe that $\mathcal{W}(z_1,z_2) = -\mathcal{W}(y_1,y_2) = -1$.
To continue the construction of $K_{\nu}(x,\xi)$ notice that,
for any fixed $\xi \in (0,1)$, using \eqref{eq:DD*new} and then \eqref{eq:green} one obtains
\begin{align*}
\frac{\partial}{\partial x}K_{\nu}(x,\xi)\Big|_{x=\xi^{-}} - \frac{\partial}{\partial x}K_{\nu}(x,\xi)\Big|_{x=\xi^{+}}
& = - \lim_{\epsilon \to 0^+} \int_{\xi-\epsilon}^{\xi+\epsilon} \frac{\partial^2}{\partial x^2} K_{\nu}(x,\xi)\, dx \\
& = \lim_{\epsilon \to 0^+} \int_{\xi-\epsilon}^{\xi+\epsilon} \mathbb{D}_{\nu} \mathbb{D}_{\nu}^{*} K_{\nu}(x,\xi)\, dx = 1.
\end{align*}
This combined with \eqref{eq:casgr} forces
$$
z'_1(\xi) A(\xi) - z'_2(\xi) B(\xi) = 1.
$$
On the other hand, the required continuity at $x=\xi$ implies
$$
z_1(\xi) A(\xi) = z_2(\xi) B(\xi).
$$
These two equations form a system whose solution for $A$ and $B$ is
\begin{align*}
A(\xi) & = \frac{- z_2(\xi)}{\mathcal{W}(z_1,z_2)} = z_2(\xi) = y_1(\xi) - y_2(\xi) = \big(1-F(\xi)\big) y_1(\xi)
	= \frac{1-F(\xi)}{\psi_1^{\nu}(\xi)}, \\
B(\xi) & = \frac{- z_1(\xi)}{\mathcal{W}(z_1,z_2)} = z_1(\xi) = y_2(\xi) = F(\xi) y_1(\xi) = \frac{F(\xi)}{\psi_1^{\nu}(\xi)}.
\end{align*}
Plugging this to \eqref{eq:casgr} we arrive at the formula
\begin{equation} \label{greenf}
K_{\nu}(x,\xi) = 	\begin{cases}
										\frac{F(x)(1-F(\xi))}{\psi_1^{\nu}(x)\psi_1^{\nu}(\xi)}, & x \in (0,\xi], \\
										\frac{(1-F(x))F(\xi)}{\psi_1^{\nu}(x)\psi_1^{\nu}(\xi)}, & x \in [\xi,1).
									\end{cases}
\end{equation}
Note that this formula is explicit, in view of the explicit expressions for $F$ and $\psi_1^{\nu}$.

By the construction, $K_{\nu}(x,\xi)$ possesses the asserted properties, it remains only to verify \eqref{KL2}.
By \eqref{asy3},
\begin{equation} \label{symp1}
\psi_1^{\nu}(x) \simeq x^{\nu+1/2}(1-x), \qquad x \in (0,1).
\end{equation}
Further,
\begin{align}
F(x) & = \int_0^x \big[ \psi_1^{\nu}(u)\big]^2\, du \simeq x^{2\nu+2}, \qquad x \in (0,1), \label{symp2} \\
1-F(x) & = \int_x^1 \big[ \psi_1^{\nu}(u)\big]^2\, du \simeq (1-x)^3, \qquad x \in (0,1). \label{symp3}
\end{align}
Consequently,
$$
K_{\nu}(x,\xi) \simeq \begin{cases}
												\frac{x^{\nu+3/2}(1-\xi)^2}{(1-x)\xi^{\nu+1/2}}, & x \in (0,\xi],\\
												\frac{(1-x)^2 \xi^{\nu+3/2}}{x^{\nu+1/2}(1-\xi)}, & x \in [\xi,1).
											\end{cases}								
$$
Now it is elementary to check \eqref{KL2}. Since $K_{\nu}(x,\xi) = K_{\nu}(\xi,x)$, we may restrict attention to $0 < \xi < x$.
Then
$$
\int_0^1 \int_0^x K_{\nu}^2(x,\xi)\, d\xi dx \simeq \int_0^1 \int_0^x \frac{(1-x)^4 \xi^{2\nu+3}}{x^{2\nu+1}(1-\xi)^2}\, d\xi dx
	\simeq \int_0^1 \frac{(1-x)^4 x^{2\nu+4}}{x^{2\nu+1}(1-x)}\, dx = \int_0^1 x^3(1-x)^3\, dx < \infty.
$$
This finishes the proof.
\end{proof}

Let $K_{\nu}(x,\xi)$ be the Green function defined in \eqref{greenf}.
For $\nu > -1$ consider the integral operator
$$
T_{\nu}f(x) = \int_0^1 K_{\nu}(x,\xi) f(\xi)\, d\xi
$$
acting on $L^2(dx)$. In view of \eqref{KL2} this operator is compact.
Further, since the integral kernel is symmetric, $T_{\nu}$ is self-adjoint.
Moreover, the null space of $T_{\nu}$ is trivial. Indeed, assume that $T_{\nu}f = 0$ for some $f \in L^2(dx)$.
Then, by the explicit formula \eqref{greenf},
$$
\frac{1-F(x)}{\psi_1^{\nu}(x)} \int_0^x \frac{F(\xi)}{\psi_1^{\nu}(\xi)} f(\xi)\, d\xi +
\frac{F(x)}{\psi_1^{\nu}(x)} \int_x^1 \frac{1-F(\xi)}{\psi_1^{\nu}(\xi)} f(\xi)\, d\xi = 0, \qquad x \in (0,1).
$$
Multiplying this equation by $\psi_1^{\nu}(x)/F(x)$ and then differentiating both sides we get
$$
\bigg( \frac{1-F(x)}{F(x)}\bigg)'
\int_0^{x} \frac{F(\xi)}{\psi_1^{\nu}(\xi)} f(\xi)\, d\xi = 
-\bigg( \frac{\psi_1^{\nu}(x)}{F(x)}\bigg)^2
\int_0^{x} \frac{F(\xi)}{\psi_1^{\nu}(\xi)} f(\xi)\, d\xi =
0, \qquad x \in (0,1).
$$
This clearly implies that $f=0$.

From a general theory, see \cite[p.\,213--214]{H}, it follows that the set of eigenfunctions of $T_{\nu}$
forms an orthogonal basis in $L^2(dx)$. Therefore Theorem \ref{thm:L2dense}(c) will be proved once we show the following.

\begin{lemma} \label{lem:eig}
Let $\nu > -1$. Each eigenfunction of $T_{\nu}$ is of the form $c\, \mathbb{D}_{\nu}\psi_n^{\nu}$ for some $n \ge 2$
and some constant $c \neq 0$.
\end{lemma}

To prove this we need an auxiliary result.
\begin{lemma} \label{lem:eigA}
Let $\nu > -1$ and let $f$ be an eigenfunction of $T_{\nu}$. Then
$$
|f(x)| \lesssim \big[ x (1-x)\big]^{3/2}
	\begin{cases}
		1, & \nu > 0, \\
		\log^{1/2}\frac{2}x, & \nu = 0, \\
		x^{\nu}, & \nu < 0,
	\end{cases}
$$
uniformly in $x \in (0,1)$. Furthermore,
$$
|\mathbb{D}_{\nu}^{*}f(x)| \lesssim (1-x)
				\begin{cases}
					x^{5/2}, & \nu > 2, \\
					x^{5/2}\log\frac{2}{x}, & \nu =2, \\					
					x^{\nu+1/2}, & \nu < 2,
				\end{cases}
$$
uniformly in $x \in (0,1)$. In particular, $\mathbb{D}_{\nu}^{*}f(1) := \lim_{x \to 1^{-}} \mathbb{D}_{\nu}^{*}f(x) = 0$.
\end{lemma}

\begin{proof}
Since $T_{\nu}f = \Lambda f$ for some $\Lambda \neq 0$, it is enough we show the lemma with $f$ replaced by $T_{\nu}f$.

By the Schwarz inequality, 
$$
|T_{\nu}f(x)| \le \| K_{\nu}(x,\cdot)\|_{L^2(d\xi)} \|f\|_{L^2(d\xi)}, \qquad x \in (0,1).
$$
We now bound suitably $\|K_{\nu}(x,\cdot)\|_{L^2(d\xi)}^2$.
According to the computation at the end of the proof of Lemma~\ref{lem:green},
$$
\int_0^x K_{\nu}^2(x,\xi)\, d\xi \simeq \big[ x(1-x)\big]^{3}.
$$
The complementary integral is estimated with the aid of \eqref{greenf} and \eqref{symp1}--\eqref{symp3}. We get
$$
\int_x^1 K_{\nu}^2(x,\xi)\, d\xi \simeq \frac{x^{2\nu+3}}{(1-x)^2} \int_x^1 \frac{(1-\xi)^4}{\xi^{2\nu+1}}\, d\xi
	\simeq 	\big[ x (1-x)\big]^3  
					\begin{cases}
						1, & \nu > 0, \\
						\log\frac{2}x, & \nu = 0, \\
						x^{2\nu}, & \nu < 0.
					\end{cases}
$$
Consequently, the first bound of the lemma follows.

To prove the second bound, we simply use the explicit form of $T_{\nu}$ together with \eqref{eq:derleb} and
the identity $\mathbb{D}_{\nu}\psi_1^{\nu}=0$, and compute that
$$
\mathbb{D}_{\nu}^{*}(T_{\nu}f)(x) = \psi_1^{\nu}(x) \Bigg[ \int_0^x \frac{F(\xi)}{\psi_1^{\nu}(\xi)} f(\xi)\, d\xi
	- \int_x^1 \frac{1-F(\xi)}{\psi_{1}^{\nu}(\xi)} f(\xi)\, d\xi \Bigg].
$$
Denote by $\Psi$ the absolute value of the above expression in square brackets. To estimate $\Psi$ we use
the first bound of the lemma and the estimates \eqref{symp1}--\eqref{symp3}.
Thus when $\nu > 0$ we get
$$
\Psi \lesssim 
\int_0^x \frac{\xi^{2\nu+2}\xi^{3/2}(1-\xi)^{3/2}}{\xi^{\nu+1/2}(1-\xi)}\, d\xi +
\int_x^1 \frac{(1-\xi)^3 \xi^{3/2} (1-\xi)^{3/2}}{\xi^{\nu+1/2}(1-\xi)}\, d\xi \simeq
	x^{\nu+4} + (1-x)^{9/2}	\begin{cases}
													 x^{2-\nu}, & \nu > 2, \\
													\log\frac{2}{x}, & \nu = 2, \\
													1, & \nu < 2,
													\end{cases}												
$$
so in this case
$$
\Psi \lesssim 					\begin{cases}
													 x^{2-\nu}, & \nu > 2, \\
													\log\frac{2}{x}, & \nu = 2, \\
													1, & \nu < 2.
												\end{cases}		
$$
If $\nu = 0$, then the additional factor $\log^{1/2}\frac{2}{x}$ appears under the above integrals and the bound we get is
$$
\Psi \lesssim x^4 \log^{1/2}\frac{2}{x} + (1-x)^{9/2} \lesssim 1.
$$
Finally, for $\nu < 0$ we obtain
$$
\Psi \lesssim \int_0^x \frac{\xi^{2\nu+2} \xi^{\nu+3/2}(1-\xi)^{3/2}}{\xi^{\nu+1/2}(1-\xi)}\, d\xi
	+ \int_x^1 \frac{(1-\xi)^3 \xi^{\nu+3/2}(1-\xi)^{3/2}}{\xi^{\nu+1/2}(1-\xi)}\, d\xi
\simeq x^{2\nu+4} + (1-x)^{9/2} \lesssim 1.
$$
These bounds for $\Psi$, together with \eqref{symp1}, imply the second estimate of the lemma.
\end{proof}

\begin{proof}[{Proof of Lemma \ref{lem:eig}}]
Assume that $f$ is an eigenfunction of $T_{\nu}$,
\begin{equation*}
T_{\nu}f = \Lambda f
\end{equation*}
for some $\Lambda \in \mathbb{R}$; here $\Lambda \neq 0$ since the null space of $T_{\nu}$ is trivial.
Note that $f$ is smooth on $(0,1)$, see \eqref{greenf}.

We begin with showing that $f$ is of the form $\mathbb{D}_{\nu} g$, where $g$ is an eigenfunction of the differential
operator $\mathbb{L}_{\nu}$. To verify this claim, notice first that $f$ is also an eigenfunction of
$\mathbb{D}_{\nu}\mathbb{D}_{\nu}^{*}$. Indeed, combining $f = \Lambda^{-1}T_{\nu}f$ with \eqref{eq:green} we have
\begin{equation} \label{eq:52}
\mathbb{D}_{\nu} \mathbb{D}_{\nu}^{*} f = \Lambda^{-1} f.
\end{equation}
Applying $\mathbb{D}_{\nu}^{*}$ to both sides of this identity we get
$$
\mathbb{D}_{\nu}^{*} \mathbb{D}_{\nu} \mathbb{D}_{\nu}^{*}f = \Lambda^{-1} \mathbb{D}_{\nu}^{*}f.
$$
Thus $\mathbb{D}_{\nu}^{*}f$ is an eigenfunction of $\mathbb{D}_{\nu}^{*}\mathbb{D}_{\nu} = \mathbb{L}_{\nu}-\lambda_{1,\nu}^2$,
hence also of $\mathbb{L}_{\nu}$. More precisely,
\begin{equation} \label{eq:53}
\mathbb{D}_{\nu}^{*}f = h,
\end{equation}
with $h$ satisfying $\mathbb{L}_{\nu}h = \widetilde{\Lambda} h$, where $\widetilde{\Lambda} = \lambda_{1,\nu}^2 + \Lambda^{-1}$.
Applying now $\mathbb{D}_{\nu}$ to both sides
of \eqref{eq:53} and then using \eqref{eq:52} we see that
$$
f = \Lambda \mathbb{D}_{\nu}h = \mathbb{D}_{\nu}(\Lambda h).
$$
So the desired conclusion follows with $g=\Lambda \mathbb{D}_{\nu}^{*}f$.
Note that $\mathbb{L}_{\nu}g = \widetilde{\Lambda} g$.

Next, we verify that $\widetilde{\Lambda} > 0$. Since $\mathbb{L}_{\nu} = \lambda_{1,\nu}^2 + \mathbb{D}_{\nu}^*\mathbb{D}_{\nu}$,
we have $\mathbb{D}_{\nu}^* \mathbb{D}_{\nu}g = (\widetilde{\Lambda}-\lambda_{1,\nu}^2)g$.
Further, observe that Lemma \ref{lem:eigA} implies $g,\mathbb{D}_{\nu}g \in L^2(dx)$.
Thus, assuming the identity
\begin{equation} \label{idQ}
\langle \mathbb{D}_{\nu}g, \mathbb{D}_{\nu}g \rangle = \langle \mathbb{D}_{\nu}^*\mathbb{D}_{\nu}g,g\rangle,
\end{equation}
we can write
$$
\big\|\mathbb{D}_{\nu}g\big\|_{L^2(dx)}^2 =
\langle \mathbb{D}_{\nu}g, \mathbb{D}_{\nu}g \rangle = \langle \mathbb{D}_{\nu}^*\mathbb{D}_{\nu}g,g\rangle
= \big(\widetilde{\Lambda}-\lambda_{1,\nu}^2\big) \langle g,g \rangle = \big(\widetilde{\Lambda}-\lambda_{1,\nu}^2\big)
	\|g\|_{L^2(dx)}^2.
$$
It follows that $\widetilde{\Lambda} \ge \lambda_{1,\nu}^2 > 0$.

To justify \eqref{idQ} one integrates by parts. The constant term in this integration will be
$\mathbb{D}_{\nu}g(x) g(x)\big|_0^1$, which is zero in view of the bounds from Lemma \ref{lem:eigA}
(recall that $\mathbb{D}_{\nu}g=f$ and $g = \Lambda \mathbb{D}_{\nu}^* f$).

We have shown that $g$ is an eigenfunction of the differential operator $\mathbb{L}_{\nu}$, the corresponding eigenvalue
being positive; denote it by $\lambda^2$. Each such $g$ is of the form
$$
g(x) = c_1 \sqrt{x} J_{\nu}(\lambda x) + c_2 \sqrt{x} Y_{\nu}(\lambda x)
$$
for some constants $c_1,c_2 \in \mathbb{R}$ and $\lambda > 0$; here $Y_{\nu}$ is the Bessel function of the second kind.
This follows by changing the variable in Bessel's equation (for this classic equation and its fundamental system of solutions see
\cite[Chapter III]{watson}, in particular \cite[{Section 3$\cdot$63}]{watson}; see also \cite[Section 5.4]{Leb} and
\cite[(5.4.11) and (5.4.12)]{Leb} taken with $\a=1/2$ and $\gamma=1$).
Then, in view of what was said above, any eigenfunction $f$ of $T_{\nu}$ is of the form
$$
f(x) = c_1 \mathbb{D}_{\nu} \big( \sqrt{x} J_{\nu}(\lambda x)\big) + c_2 \mathbb{D}_{\nu} \big( \sqrt{x} Y_{\nu}(\lambda x) \big).
$$
We now show that one must have $c_2 = 0$.

A direct computation based on \eqref{diff_J} reveals that
$$
\mathbb{D}_{\nu} \big( \sqrt{x} J_{\nu}(\lambda x)\big)
	= - \lambda \sqrt{x} J_{\nu+1}(\lambda x) + R^{\nu}(x) \sqrt{x} J_{\nu}(\lambda x).
$$
Since $Y_{\nu}$ satisfies the same differentiation rules as $J_{\nu}$ (cf.\ \cite{watson}), we also have
$$
\mathbb{D}_{\nu} \big( \sqrt{x} Y_{\nu}(\lambda x)\big)
	= - \lambda \sqrt{x} Y_{\nu+1}(\lambda x) + R^{\nu}(x) \sqrt{x} Y_{\nu}(\lambda x).
$$
Using now the asymptotic \eqref{J0} and the bound \eqref{Ras} near $x=0$, we see that
\begin{equation} \label{fc1}
\big| \mathbb{D}_{\nu} \big( \sqrt{x} J_{\nu}(\lambda x)\big) \big| \lesssim x^{\nu+3/2}, \qquad x \to 0^+.
\end{equation}
On the other hand, see \cite[Chapter III]{watson}, for each $\nu > -1$
$$
|Y_{\nu}(x)| \simeq x^{-\nu} \Big( 1 + \chi_{\{\nu = 0\}}\log\frac{2}{x}\Big), \qquad x \to 0^+.
$$
This implies
\begin{equation} \label{fc2}
\big| \mathbb{D}_{\nu} \big( \sqrt{x} Y_{\nu}(\lambda x)\big) \big|
	\simeq x^{-\nu-1/2} \Big( 1 + \chi_{\{\nu = 0\}}\log\frac{2}{x}\Big), \qquad x \to 0^+.
\end{equation}
Now, the eigenfunction estimate from Lemma \ref{lem:eigA} says that, as $ x \to 0^+$,
$$
|f(x)| \lesssim x^{3/2}
	\begin{cases}
		1, & \nu > 0, \\
		\log^{1/2}\frac{2}x, & \nu = 0, \\
		x^{\nu}, & \nu < 0.
	\end{cases}
$$
This, in view of \eqref{fc1} and \eqref{fc2}, forces $c_2 = 0$.

Thus we have proved that
$$
f(x) = c_1 \mathbb{D}_{\nu} \big( \sqrt{x} J_{\nu}(\lambda x)\big).
$$
It remains to show that $\lambda$ must be one of the zeros of $J_{\nu}$.
Observe that
$$
\mathbb{D}_{\nu}^* \Big(\mathbb{D}_{\nu} \big( \sqrt{x} J_{\nu}(\lambda x)\big)\Big)
	= \big( \mathbb{L}_{\nu} - \lambda_{1,\nu}^2\big)\big( \sqrt{x} J_{\nu}(\lambda x)\big)
	= \big( \lambda^2 - \lambda_{1,\nu}^2\big)\big( \sqrt{x} J_{\nu}(\lambda x)\big).
$$
This combined with Lemma \ref{lem:eigA} gives
$$
\mathbb{D}_{\nu}^* f(1) = c_1 \big( \lambda^2 - \lambda_{1,\nu}^2\big) J_{\nu}(\lambda) = 0.
$$
It follows that $\lambda = \lambda_{n,\nu}$ for some $n \ge 2$.
The case $n=1$ is excluded since for $\lambda = \lambda_{1,\nu}$ the function $f$ would be a multiple of
$\mathbb{D}_{\nu} \psi_1^{\nu}$, which is identically zero.

The proof of Lemma \ref{lem:eig} is complete.
\end{proof}

\begin{remark} \label{rem:eigp}
A part of the proof of Lemma \ref{lem:eig} can be replaced by another reasoning based on the known completeness
of the system $\{\psi_{n}^{\nu}\}$. From the point where the identity $g = \Lambda \mathbb{D}_{\nu}^* f$ is
obtained one can argue as follows.

We know that $\mathbb{L}_{\nu}g = \widetilde{\Lambda}g$ for some $\widetilde{\Lambda} \in \mathbb{R}$.
Here $g$ is a non-trivial function.
Moreover, by Lemma \ref{lem:eigA} we have $g, \mathbb{L}_{\nu}g \in L^2(dx)$. Integrating twice by parts,
with the aid of Lemma \ref{lem:eigA} and using the formula
$\mathbb{L}_{\nu} = \lambda_{1,\nu}^2 + \mathbb{D}_{\nu}^*\mathbb{D}_{\nu}$, one can show the identity
$$
\langle \mathbb{L}_{\nu}g, \psi_n^{\nu}\rangle = \langle g, \mathbb{L}_{\nu}\psi_n^{\nu}\rangle, \qquad n \ge 1.
$$
This implies $\widetilde{\Lambda}\langle g,\psi_n^{\nu} \rangle = \lambda_{n,\nu}^2 \langle g, \psi_n^{\nu} \rangle$,
$n \ge 1$. Consequently, there is a unique $n$ such that $\langle g, \psi_{n}^{\nu}\rangle \neq 0$.
By completeness of the system $\{\psi_n^{\nu}\}$ in $L^2(dx)$ we must have $g = c \psi_n^{\nu}$ for this $n$
and some constant $c \neq 0$, and $\widetilde{\Lambda} = \lambda_{n,\nu}^2$.
Finally, the case $n=1$ is ruled out as in the proof of Lemma \ref{lem:eig}.
\end{remark}

\section{Riesz transforms} \label{sec:riesz}

Riesz transforms for Fourier-Bessel expansions can be defined according to a general approach presented in \cite{NoSt1}.
However, these definitions depend in an essential way on the choice of derivatives decomposing the associated
Laplacians. In this paper, unlike in the existing literature, we assume decompositions in which the bottom eigenvalue
appears as a constant, see Section \ref{sec:der}, which gives rise to our new derivatives.

The resulting Riesz transforms admit a satisfactory $L^2$-theory \cite{NoSt1}.
The primary aim of this section is to show that, in case of the essential measure Fourier-Bessel setting, one has
a good $L^p$-theory as well. Our main tool is Wr\'obel's general result \cite{wrobel}; see also \cite{FSS}
for an earlier development in this direction.
On the other hand, the results of \cite{wrobel,FSS} in general do not apply in cases of the natural measure and Lebesgue
measure Fourier-Bessel settings (cf.\ \cite[p.\,762]{wrobel}). Therefore these contexts require some other,
most probably technically challenging methods which are beyond the scope of this paper.

\subsection{Essential measure Fourier-Bessel setting} \label{sec:FBriesz}

Let $\nu > -1$. In $L^2(d\eta_{\nu})$ the Riesz transforms (we treat simultaneously the probabilistic variant of the framework)
are formally given by
$$
\mathfrak{R}_{\nu} = d_{\nu} \big(\mathfrak{L}_{\nu}\big)^{-1/2}, \qquad
\mathfrak{R}_{\nu}^M = d_{\nu} \big(\mathfrak{L}^M_{\nu}\big)^{-1/2} \Pi_0,
$$
where $\Pi_0$ is the orthogonal projection onto the orthogonal complement of the bottom eigenfunction of $\mathfrak{L}_{\nu}^M$
(that has zero eigenvalue).
This can be written strictly in terms of expansions with respect to the differentiated system. We have
\begin{align*}
\mathfrak{R}_{\nu}f & = \sum_{n \ge 2} \lambda_{n,\nu}^{-1} \big\langle f, \varphi_n^{\nu}\big\rangle_{d\eta_{\nu}}
	d_{\nu}\varphi_n^{\nu}, \qquad f \in L^2(d\eta_{\nu}), \\
\mathfrak{R}^M_{\nu}f & = \sum_{n \ge 2} \big( \lambda_{n,\nu}^2-\lambda_{1,\nu}^2\big)^{-1/2}
	\big\langle f, \varphi_n^{\nu}\big\rangle_{d\eta_{\nu}} d_{\nu}\varphi_n^{\nu}, \qquad f \in L^2(d\eta_{\nu}).
\end{align*}
In view of Proposition \ref{prop:orth} and Remark \ref{rem:norm}, both $\mathfrak{R}_{\nu}$ and $\mathfrak{R}_{\nu}^M$
are contractions on $L^2(d\eta_{\nu})$; see \cite[Section~3]{NoSt1}.
 
The main result of this section reads as follows.
\begin{theorem} \label{thm:Riesz}
Let $\nu \ge -1/2$ and $1 < p < \infty$. Each of the operators $\mathfrak{R}_{\nu}$ and $\mathfrak{R}_{\nu}^M$ extends
uniquely to a bounded linear operator on $L^p(d\eta_{\nu})$. Moreover,
$$
\big\| \mathfrak{R}_{\nu} \big\|_{L^p(d\eta_{\nu})\to L^p(d\eta_{\nu})} \le 48 (p^*-1), \qquad
\big\| \mathfrak{R}^M_{\nu} \big\|_{L^p(d\eta_{\nu})\to L^p(d\eta_{\nu})} \le 48 (p^*-1),
$$
where $p^*=\max(p,p/(p-1))$.
\end{theorem}
Note that the norm bounds in Theorem \ref{thm:Riesz} are not only quantitative in $p$, but also parameterless (do not depend on $\nu$).

\begin{proof}[{Proof of Theorem \ref{thm:Riesz}}]
We aim at applying \cite[Theorem 1]{wrobel}.
To make this legitimate we need to verify the following conditions (we adopt enumeration from \cite{wrobel}
and specify \cite[Conditions (A1), (A2), (T1)--(T3)]{wrobel} to our context).
\begin{itemize}
\item[(A1)]
The commutator $[d_{\nu},d_{\nu}^*] = d_{\nu} d_{\nu}^* - d_{\nu}^* d_{\nu}$ is a multiplication
	operator by a non-negative function.
\item[(A2)]
There is a non-negative constant $K$ such that $0 \le K r$, where $r=\lambda_{1,\nu}^2$ in case of
$\mathfrak{R}_{\nu}$ and $r=0$ in case of $\mathfrak{R}_{\nu}^M$.
\item[(T1)]
For $n,m \ge 1$,
$$
\big\langle d_{\nu}\varphi_n^{\nu}, d_{\nu} \varphi_m^{\nu}\big\rangle_{d\eta_{\nu}}
= \big\langle d_{\nu}^* d_{\nu}\varphi_n^{\nu}, \varphi_m^{\nu}\big\rangle_{d\eta_{\nu}}.
$$
\item[(T2)]
If $z \in \{0^+,1^-\}$, then
\begin{align}
\lim_{x \to z}\big( 1 + |\varphi_n^{\nu}(x)|^{s_1} + |d_{\nu}\varphi_n^{\nu}(x)|^{s_2}\big) w(x) \big( \varphi_{n}^{\nu} \big)'(x)
& = 0, \nonumber \\
\lim_{x \to z}\big( 1 + |\varphi_n^{\nu}(x)|^{s_1} + |d_{\nu}\varphi_n^{\nu}(x)|^{s_2}\big) w(x) \big(d_{\nu}\varphi_{n}^{\nu} \big)'(x)
& = 0, \label{lip}
\end{align}
for all $n \ge 1$ and $s_1,s_2 > 0$; here $w(x)$ is the density of the measure $\eta_{\nu}$.
\item[(T3)]
Each of the subspaces
$$
\lin\big\{\varphi_n^{\nu} : n \ge 1\big\} \qquad \textrm{and} \qquad \lin\big\{d_{\nu}\varphi_n^{\nu} : n \ge 1\big\}
$$
is dense in $L^p(d\eta_{\nu})$, $1 \le p < \infty$.
\end{itemize}

Observe that (A2) holds trivially with $K=0$, and validity of (T1) was already discussed, see \eqref{T1} (and recall that
$d_{\nu}\varphi_1^{\nu} = (\varphi_{1}^{\nu})'=0$).
Concerning (A1), by a direct computation we obtain
$$
[d_{\nu},d_{\nu}^*] = 2\big(R^{\nu}\big)'(x) + \frac{2\nu+1}{x^2}.
$$
It is enough we show that $(R^{\nu})'$ is non-negative. Using \eqref{R1} and \eqref{eq:S} we get
\begin{equation} \label{Rprim}
\big( R^{\nu}\big)'(x) = \frac{1}{(1+x)^2} + \frac{1}{(1-x)^2} + 2 \lambda_{1,\nu}^2 \sum_{k=2}^{\infty}
	\frac{\lambda_{k,\nu}^2 + \lambda_{1,\nu}^2 x^2}{(\lambda_{k,\nu}^2-\lambda_{1,\nu}^2 x^2)^2},
\end{equation}
and this expression is clearly positive for $x \in (0,1)$.
It remains to verify (T2) and (T3), which are less straightforward.

To show (T2) we first estimate all the component quantities. To this end, we assume that $\nu$ and $n$ are fixed.
We have, see \eqref{asy1} and Proposition \ref{prop:asd},
\begin{equation} \label{ee1}
\big| \varphi_n^{\nu}(x) \big| \lesssim 1, \qquad
\big| d_{\nu}\varphi_n^{\nu}(x)\big| = \big| \big( \varphi_n^{\nu}\big)'(x) \big| \lesssim x(1-x), \qquad x \in (0,1).
\end{equation}
Further, by \eqref{asy2},
\begin{equation} \label{ee2}
w(x) = x^{2\nu+1} \big[ \phi_1^{\nu}(x)\big]^2 \simeq x^{2\nu+1}(1-x)^2, \qquad x \in (0,1).
\end{equation}
Finally, to bound $|(d_{\nu}\varphi_n^{\nu})'|$ we observe that
$d_{\nu} = -d_{\nu}^* - (2\nu+1)/x + 2R^{\nu}(x)$, hence
\begin{align*}
\big( d_{\nu} \varphi_n^{\nu}\big)'(x) & =
	-d_{\nu}^* d_{\nu} \varphi_n^{\nu}(x) + \Big( 2R^{\nu}(x) - \frac{2\nu+1}x \Big) d_{\nu}\varphi_n^{\nu}(x) \\
& = - \big( \lambda_{n,\nu}^2 - \lambda_{1,\nu}^2 \big) \varphi_n^{\nu}(x)
		+ \Big( 2R^{\nu}(x) - \frac{2\nu+1}x \Big) d_{\nu}\varphi_n^{\nu}(x).
\end{align*}
In view of \eqref{ee1} and \eqref{Ras}, we conclude that
\begin{equation} \label{ee3}
\big| \big( d_{\nu}\varphi_n^{\nu}\big)'(x) \big| \lesssim 1, \qquad x \in (0,1).
\end{equation}

Now, using \eqref{ee1}, \eqref{ee2} and \eqref{ee3} we see that
\begin{align*}
\Big|\big( 1 + |\varphi_n^{\nu}(x)|^{s_1} + |d_{\nu}\varphi_n^{\nu}(x)|^{s_2}\big) w(x) \big( \varphi_{n}^{\nu} \big)'(x)\Big|
& \lesssim x^{2\nu+2}(1-x)^3, \qquad x \in (0,1), \\
\Big|\big( 1 + |\varphi_n^{\nu}(x)|^{s_1} + |d_{\nu}\varphi_n^{\nu}(x)|^{s_2}\big) w(x) \big(d_{\nu}\varphi_{n}^{\nu} \big)'(x)\Big|
& \lesssim x^{2\nu+1} (1-x)^2, \qquad x \in (0,1).
\end{align*}
It is clear that these expressions converge to $0$ as $x$ approaches the boundary of $(0,1)$, provided that
$\nu > -1/2$ in the second case (this restriction on the parameter is necessary, otherwise \eqref{lip} does not hold
for $z = 0^+$).

Passing to (T3), it essentially follows from completeness of the systems $\{\varphi_n^{\nu}\}$ and $\{d_{\nu} \varphi_n^{\nu}\}$
in $L^2(d\eta_{\nu})$ (see Theorem \ref{thm:L2dense}(a))
and suitable uniform estimates for functions of the systems that are stated in Lemma~\ref{lem:ues} below.
The arguments are well known, see for instance
\cite[Section 2]{CiSt1} or \cite[Section 2]{CiSt3}, so we will not go into much detail.

We first deal with the system $\{\varphi_n^{\nu}\}$.
Fix an arbitrary $f \in C_c^{\infty}(0,1)$. We will show that $f$ can be approximated in the $L^p(d\eta_{\nu})$ norm
by finite linear combinations of $\varphi_n^{\nu}$. This will imply density of $\lin\{\varphi_n^{\nu}\}$
in $L^p(d\eta_{\nu})$, since $C_c^{\infty}(0,1)$ is dense in $L^p(d\eta_{\nu})$.

Let
$$
S_{N}f = \sum_{n=1}^N \langle f, \varphi_n^{\nu} \rangle_{d\eta_{\nu}} \varphi_n^{\nu}, \qquad N \ge 1.
$$
Since $\{\varphi_n^{\nu} : n \ge 1\}$ is an orthonormal basis in $L^2(d\eta_{\nu})$, we know that $S_N f \to f$
in $L^2(d\eta_{\nu})$ as $N \to \infty$. Hence there is a sub-sequence $S_{N_k}f$ converging to $f$ pointwise almost everywhere on
$(0,1)$. By the dominated convergence theorem, $S_{N_k}f$ converges to $f$ also in $L^p(d\eta_{\nu})$, provided that
we can find a $d\eta_{\nu}$-integrable majorant for the expression $|S_{N_k}f(x)-f(x)|^p$ or, equivalently,
$|S_{N_k}f(x)|^p$.

Given any $M > 0$, we have the bound
\begin{equation} \label{fcb}
\big| \langle f, \varphi_n^{\nu} \rangle_{d\eta_{\nu}} \big| \lesssim n^{-M}, \qquad n \ge 1.
\end{equation}
This follows from the first estimate of Lemma \ref{lem:ues}
and the symmetry of $\mathfrak{L}_{\nu}$. Indeed, given $m \in \mathbb{N}_{+}$, we can write
$$
\big| \langle f, \varphi_n^{\nu} \rangle_{d\eta_{\nu}} \big|
= \frac{1}{(\lambda_{n,\nu}^2)^m} \Big| \big\langle f, \mathfrak{L}^m \varphi_n^{\nu} \big\rangle_{d\eta_{\nu}} \Big|
= \frac{1}{(\lambda_{n,\nu}^2)^m} \Big| \big\langle \mathfrak{L}^m f, \varphi_n^{\nu} \big\rangle_{d\eta_{\nu}} \Big|
\lesssim n^{-2m + \nu + 5/2},
$$
where in the last step we used \eqref{eigvas} and \eqref{phib}.
By choosing $m$ sufficiently large we infer \eqref{fcb}.

Take now $M > \nu + 7/2$. Then, by \eqref{phib} and \eqref{fcb}
$$
|S_{N_k}f(x)| \lesssim \sum_{n=1}^{N_k} n^{-M} n^{\nu+5/2} \lesssim 1,
$$
uniformly in $N_k$ and $x \in (0,1)$. Thus the majorant we are looking for is a sufficiently large constant function.
The conclusion follows.

Next, we focus on the differentiated system $\{d_{\nu}\varphi_{n}^{\nu}\}$. Here the reasoning is in general the same, we
only need to provide the relevant ingredients.
Recall from Section \ref{sec:diff} that $\{d_{\nu}\varphi_n^{\nu} : n \ge 2\}$
is an orthogonal basis in $L^2(d\eta_{\nu})$ with the normalization
$\|d_{\nu}\varphi_n^{\nu}\|_{L^2(d\eta_{\nu})} = (\lambda_{n,\nu}^2-\lambda_{1,\nu}^2)^{1/2} \simeq n$, $n \ge 2$.
Further, $d_{\nu}\varphi_n^{\nu}$ are eigenfunctions of the Laplacian
$\mathfrak{M}_{\nu} = \mathfrak{L}_{\nu} + [d_{\nu},d_{\nu}^*] = \lambda_{1,\nu}^2 + d_{\nu}d_{\nu}^*$,
which is symmetric in $L^2(d\eta_{\nu})$, and one has
$\mathfrak{M}_{\nu}(d_{\nu}\varphi_n^{\nu}) = \lambda_{n,\nu}^2(d_{\nu}\varphi_n^{\nu})$, $n \ge 2$; see \cite[Section 5]{NoSt1}.
Furthermore, a suitable uniform bound for $|d_{\nu}\varphi_n^{\nu}|$ is given in Lemma \ref{lem:ues} below.
With all these facts one concludes the density of $\lin\{d_{\nu}\varphi_n^{\nu}\}$ in $L^p(d\eta_{\nu})$
analogously as in the case of $\lin\{\varphi_n^{\nu}\}$.

Now Conditions (A1), (A2) and (T1)--(T3) are verified, so \cite[Theorem 1]{wrobel} can be applied to get the result
in case $\nu > -1/2$. Here $\nu = -1/2$ is excluded due to the restriction in Condition (T2).
Nevertheless, we can cover $\nu = -1/2$ by a simple continuity argument completely analogous to that used in the proofs
of \cite[Theorems 11 and 12]{wrobel}.
\end{proof}

\begin{remark} \label{rem:Riesz}
The restriction $\nu \in [-1/2,\infty)$ is indispensable in our approach. It is necessary for non-negativity of the
commutator (Condition (A1)) which is a crucial assumption for the machineries of \cite{wrobel,FSS} to work.
Nevertheless, we believe that both $\mathfrak{R}_{\nu}$ and $\mathfrak{R}_{\nu}^M$ are $L^p$-bounded for all $\nu > -1$,
though the norms might depend essentially on $\nu$ close to $-1$.
\end{remark}

\begin{remark} \label{rem:densC0}
In the proof of Theorem \ref{thm:Riesz} it was shown that $\lin\{\varphi_n^{\nu}:n \ge 1\}$ and
$\lin\{d_{\nu}\varphi_n^{\nu} : n \ge 2\}$ are dense in $L^p(d\eta_{\nu})$, $1 \le p < \infty$.
Similar arguments show that closures under $\|\cdot\|_{\infty}$ norm of
these subspaces contain $C_0(0,1)$, the space of continuous functions on $(0,1)$ vanishing at the boundary.
Indeed, it suffices to check that for $f\in C_c^{\infty}(0,1)$ the sequence $S_N f$ is uniformly fundamental,
and this is done with the aid of \eqref{fcb} (or its analogue for the differentiated system) and Lemma \ref{lem:ues}.
Since there is a sub-sequence $S_{N_k}f$ converging to $f$ pointwise a.e., the conclusion follows.
\end{remark}

We now state and prove the rough uniform bounds for the systems $\{\varphi_n^{\nu}\}$ and $\{d_{\nu}\varphi_n^{\nu}\}$
used in verification of (T3) in the proof of Theorem \ref{thm:Riesz}.
We leave aside the question of optimality of the exponents involved, which is not relevant for our needs.
\begin{lemma} \label{lem:ues}
Let $\nu > -1$ be fixed. Then
\begin{equation} \label{phib}
|\varphi_n^{\nu}(x)| \lesssim n^{\nu+5/2}, \qquad n \ge 1, \quad x \in (0,1),
\end{equation}
and
\begin{equation} \label{phibd}
|d_{\nu}\varphi_n^{\nu}(x)| \lesssim n^{\nu+11/2}, \qquad n \ge 2, \quad x \in (0,1).
\end{equation}
\end{lemma}

\begin{proof}
We begin with proving the simpler part \eqref{phib}.
We have
$$
\varphi_n^{\nu}(x) = \frac{J_{\nu+1}(\lambda_{1,\nu})}{J_{\nu+1}(\lambda_{n,\nu})}
	\frac{J_{\nu}(\lambda_{n,\nu}x)}{J_{\nu}(\lambda_{1,\nu}x)}, \qquad n \ge 1.
$$
Observe that here $J_{\nu+1}(\lambda_{1,\nu})$ is just a constant independent of $n$, and
\begin{equation} \label{spe1z}
\inf_{n \ge 1} \sqrt{n} |J_{\nu+1}(\lambda_{n,\nu})| > 0,
\end{equation}
by \eqref{zer_int}, \eqref{eigvas} and \eqref{Jinfp}. Further, by \eqref{J0} and \eqref{J},
$$
J_{\nu}(\lambda_{1,\nu} x) \simeq x^{\nu} (1-x), \qquad x \in (0,1).
$$
All this implies
\begin{equation} \label{ephip}
|\varphi_n^{\nu}(x)| \lesssim \sqrt{n} \frac{J_{\nu}(\lambda_{n,\nu} x)}{x^{\nu}(1-x)}, \qquad n \ge 1, \quad x \in (0,1).
\end{equation}

To proceed, we note that by \eqref{J0}, \eqref{Jinf} and \eqref{eigvas}
$$
|J_{\nu}(\lambda_{n,\nu}x)| \lesssim
	\begin{cases}
		(nx)^{\nu}, & 0 < x < 1/n, \\
		1, & 1/n \le x < 1,
	\end{cases}
$$
uniformly in $n \ge 1$ and $x \in (0,1)$. Combining this with \eqref{ephip} we get
$$
|\varphi_n^{\nu}(x)| \lesssim
	\begin{cases}
		n^{\nu+1/2}, & x \in (0,1/2], \\
		n^{3/2}, & x \in (1/2,1-1/(2n)],
	\end{cases}
$$
uniformly in $n \ge 1$ and $x \in (0,1-1/(2n)]$. To treat $x$ close to $1$ we first use 
$x^{\nu} \simeq 1$ and $J_{\nu}(\lambda_{n,\nu})=0$, and then the mean value theorem,
\eqref{diff_J}, \eqref{Jinf} and \eqref{eigvas}, obtaining
\begin{align*}
\frac{|J_{\nu}(\lambda_{n,\nu}x)|}{x^{\nu}(1-x)} & \simeq
	\frac{|J_{\nu}(\lambda_{n,\nu}x) - J_{\nu}(\lambda_{n,\nu})|}{1-x}
		\le \lambda_{n,\nu} \sup_{y \in (1-1/(2n),1)} \big| J'_{\nu}(\lambda_{n,\nu} y)\big| \\
& \le \sup_{y \in (1-1/(2n),1)} \bigg( \Big| \frac{\nu}{y} J_{\nu}(\lambda_{n,\nu} y) \Big|
		+ \lambda_{n,\nu} |J_{\nu+1}(\lambda_{n,\nu} y)| \bigg) \lesssim n,
\end{align*}
uniformly in $n \ge 1$ and $x \in (1-1/(2n),1)$. Together with \eqref{ephip}, this gives the bound for $|\varphi_n^{\nu}|$.
Summing up, for any fixed $\nu > -1$
$$
|\varphi_n^{\nu}(x)| \lesssim n^{\max(\nu+1/2,3/2)}, \qquad n \ge 1, \quad x \in (0,1),
$$
which obviously implies \eqref{phib}.

It remains to prove \eqref{phibd}. We have, see \eqref{RtoS},
$$
\big| d_{\nu}\varphi_n^{\nu}(x) \big| = \big| \big( R^{\nu}(x) - R_n^{\nu}(x) \big) \varphi_n^{\nu}(x) \big| =
	\big| \big( S^{\nu}(x) - S_n^{\nu}(x) \big) \varphi_n^{\nu}(x) \big| 
\le \big| S^{\nu}(x) \varphi_n^{\nu}(x) \big| + \big| S_n^{\nu}(x) \varphi_n^{\nu}(x)\big|.
$$
Since $S^{\nu}$ is bounded on $(0,1)$, taking into account \eqref{phib} our task reduces to estimating
$|S_n^{\nu}\varphi_n^{\nu}|$ for $n \ge 2$.

In view of \eqref{eigvas},
$$
\big| S_n^{\nu}(x) \big| \lesssim n^2 \bigg| \sum_{k \ge 1, k \neq n}
	\frac{1}{\lambda_{k,\nu}^2 - \lambda_{n,\nu}^2 x^2} \bigg|, \qquad n \ge 2, \quad x \in (0,1).
$$
Then
$$
\big| S_n^{\nu}(x) \varphi_n^{\nu}(x) \big| \lesssim n^2 \big| \varphi_n^{\nu}(x) \big|
	\big( \Upsilon_0 + \Upsilon_{\infty} \big),  \qquad n \ge 2, \quad x \in (0,1),
$$
where
$$
\Upsilon_0 := \bigg| \sum_{1 \le k < n} \frac{1}{\lambda_{k,\nu}^2 - \lambda_{n, \nu}^2 x^2} \bigg|, \qquad
\Upsilon_{\infty} := \sum_{k > n} \frac{1}{\lambda_{k,\nu}^2 - \lambda_{n,\nu}^2 x^2}.
$$

We treat first the easier term containing $\Upsilon_{\infty}$. Clearly,
$$
\Upsilon_{\infty} \le \sum_{k > n} \frac{1}{\lambda_{k,\nu}^2 - \lambda_{n,\nu}^2} \le
	\bigg| \sum_{k \ge 1, k \neq n} \frac{1}{\lambda_{k,\nu}^2 - \lambda_{n,\nu}^2} \bigg|
	+ \sum_{1 \le k < n} \frac{1}{\lambda_{n,\nu}^2 - \lambda_{k,\nu}^2}.
$$
Using now Calogero's formula \eqref{Calo} we get
$$
\Upsilon_{\infty} < \frac{\nu+1}{2\lambda_{n,\nu}^2} + \frac{n}{\lambda_{n,\nu}}
	\sup_{1\le k < n} \frac{1}{\lambda_{n,\nu}- \lambda_{k,\nu}}
\le \frac{\nu + 1}{2\lambda_{n,\nu}^2} + \frac{n}{\lambda_{n,\nu}} \frac{1}{\kappa},
$$
where
$$
\kappa := \inf_{k,l \ge 1, k \neq l} |\lambda_{k,\nu} - \lambda_{l,\nu}| > 0;
$$
strict positivity of $\kappa$ is justified by \eqref{eigvas}. Making again use of \eqref{eigvas} we see that
$\Upsilon_{\infty} \lesssim 1$ uniformly in $x \in (0,1)$ and $n \ge 2$. This together with \eqref{phib} implies
$$
n^2 \big| \varphi_n^{\nu}(x) \big| \Upsilon_{\infty} \lesssim n^{\nu+9/2},  \qquad n \ge 2, \quad x \in (0,1).
$$

We pass to the analysis of the $\Upsilon_0$ term. With the aid of \eqref{eigvas} we can estimate
$$
n^2 \big| \varphi_n^{\nu}(x)\big| \Upsilon_0 \lesssim n^2 \big| \varphi_n^{\nu}(x) \big| n
	\sup_{1 \le k < n} \frac{1}{|\lambda_{k,\nu} - \lambda_{n,\nu}x|}
		\lesssim n^2 \sup_{1 \le k < n} \Bigg| \frac{\varphi_n^{\nu}(x)}{x - \frac{\lambda_{k,\nu}}{\lambda_{n,\nu}}} \Bigg|.
$$
We will show that
\begin{equation} \label{qq}
\sup_{1 \le k < n} \Bigg| \frac{\varphi_n^{\nu}(x)}{x - \frac{\lambda_{k,\nu}}{\lambda_{n,\nu}}} \Bigg|
	\lesssim n^{\nu+7/2},  \qquad n \ge 2, \quad x \in (0,1).
\end{equation}
This will finish the whole reasoning.

Observe that by \eqref{eigvas} one can choose a constant $C > 1$ (depending on $\nu$, but not on $n$) such that
$$
\frac{\lambda_{1,\nu}}{\lambda_{n,\nu}} > \frac{2}{C n} \qquad \textrm{and} \qquad
	1 - \frac{\lambda_{n-1,\nu}}{\lambda_{n,\nu}} > \frac{2}{C n}, \qquad n \ge 2.
$$
Denote
$$
I_{n,k}^{\nu} = \bigg( \frac{\lambda_{k,\nu}}{\lambda_{n,\nu}} - \frac{1}{C n}, 
		\frac{\lambda_{k,\nu}}{\lambda_{n,\nu}} + \frac{1}{C n} \bigg).
$$
Notice that, for $n \ge 2$,
\begin{equation} \label{q3}
I_n^{\nu} := \bigcup_{1 \le k < n} I_{n,k}^{\nu} \subset \bigg( \frac{\lambda_{1,\nu}}{\lambda_{n,\nu}} - \frac{1}{C n}, 
		\frac{\lambda_{n-1,\nu}}{\lambda_{n,\nu}} + \frac{1}{C n} \bigg) \subset
		\bigg( \frac{1}{C n}, 1- \frac{1}{C n} \bigg) \subset (0,1).
\end{equation}
It is straightforward that
$$
\sup_{1 \le k < n} \Bigg| \frac{\varphi_n^{\nu}(x)}{x - \frac{\lambda_{k,\nu}}{\lambda_{n,\nu}}} \Bigg|
\lesssim n \big| \varphi_n^{\nu}(x)\big| \lesssim n^{\nu+7/2}, \qquad x \in \big( I_n^{\nu} \big)^c \cap (0,1), \quad n \ge 2,
$$
with the last bound being a consequence of \eqref{phib}. On the other hand, by the mean value theorem,
$$
\Bigg| \frac{\varphi_n^{\nu}(x)}{x - \frac{\lambda_{k,\nu}}{\lambda_{n,\nu}}} \Bigg| =
\Bigg| \frac{\varphi_n^{\nu}(x) - \varphi_{n}^{\nu}\big(
	\frac{\lambda_{k,\nu}}{\lambda_{n,\nu}}\big)}{x - \frac{\lambda_{k,\nu}}{\lambda_{n,\nu}}} \Bigg|
	\le \sup_{y \in I_{n,k}^{\nu}} \big| \big( \varphi_n^{\nu}\big)'(y) \big|,
		\qquad x \in I_{n,k}^{\nu}, \quad 1 \le k < n, \quad n \ge 2,
$$
where we used $ \varphi_{n}^{\nu}(
	{\lambda_{k,\nu}}\slash{\lambda_{n,\nu}})=0$ to write the equality.
So  \eqref{qq} will follow once we verify that
$$
\sup_{y \in I_{n,k}^{\nu}} \big| \big( \varphi_n^{\nu}\big)'(y) \big| \lesssim n^{\nu+7/2}, \qquad 1 \le k < n, \quad n \ge 2.
$$
It is enough, see \eqref{q3}, to check the stronger bound
\begin{equation} \label{q4}
\sup_{\frac{1}{C n} < y < 1 - \frac{1}{C n}} \big| \big( \varphi_n^{\nu}\big)'(y) \big| \lesssim n^{\nu + 7/2}, \qquad n \ge 2.
\end{equation}

A simple computation based on the explicit form of $\varphi_n^{\nu}$ and \eqref{diff_J} leads to the expression
$$
 \big( \varphi_n^{\nu}\big)'(y) = \Xi_1 - \Xi_2,
$$
with
$$
 \Xi_1 := \frac{J_{\nu+1}(\lambda_{1,\nu})}{J_{\nu+1}(\lambda_{n,\nu})}
	\frac{\lambda_{1,\nu} J_{\nu}(\lambda_{n,\nu} y) J_{\nu+1}(\lambda_{1,\nu} y)}{J_{\nu}^2(\lambda_{1,\nu} y)}, \qquad
\Xi_2 := \frac{J_{\nu+1}(\lambda_{1,\nu})}{J_{\nu+1}(\lambda_{n,\nu})}
	\frac{\lambda_{n,\nu} J_{\nu+1}(\lambda_{n,\nu} y)}{J_{\nu}(\lambda_{1,\nu} y)}.
$$
Observe that here $J_{\nu+1}(\lambda_{1,\nu})$ is just a constant, and $\sqrt{n}|J_{\nu+1}(\lambda_{n,\nu})|$ is separated from zero,
see \eqref{spe1z}. Further, since $\lambda_{n,\nu} y \gtrsim 1$ for $y > \frac{1}{C n}$ and $n \ge 2$,
see \eqref{eigvas}, it follows from \eqref{Jinf} that
$$
|J_{\nu}(\lambda_{n,\nu} y)| \lesssim 1 \qquad \textrm{and} \qquad
	|J_{\nu+1}(\lambda_{n,\nu} y)| \lesssim 1, \qquad y > \frac{1}{C n},
$$
uniformly in $n \ge 2$. Taking into account the above and using \eqref{J0}, \eqref{J}, \eqref{Ras} and \eqref{eigvas}
we conclude that
$$
|\Xi_1| \lesssim \frac{\sqrt{n} R^{\nu}(y)}{J_{\nu}(\lambda_{1,\nu} y)} \lesssim \sqrt{n} y^{1-\nu}(1-y)^{-2}, \qquad
|\Xi_2| \lesssim n^{3/2} y^{-\nu} (1-y)^{-2}, \qquad \frac{1}{C n} < y < 1 - \frac{1}{C n},
$$
uniformly in $n \ge 2$. From this \eqref{q4} can instantly be deduced.
This finishes proving \eqref{phibd}.
\end{proof}

\subsection{Multi-dimensional essential measure Fourier-Bessel setting}

Let $d \ge 1$. Taking a tensor product of $d$ copies of the one-dimensional essential measure Fourier-Bessel situation
leads in a canonical way to a $d$-dimensional essential measure Fourier-Bessel framework. See \cite{NoSt1} for the details.
Then $\nu \in (-1,\infty)^d$ is a multi-parameter, the space is $(0,1)^d$ equipped with the product measure
$\eta_{\nu} = \eta_{\nu_1} \otimes \ldots \otimes \eta_{\nu_d}$, and the orthonormal system in $L^2(d\eta_{\nu})$ is
$$
\varphi_n^{\nu} = \varphi_{n_1}^{\nu_1} \otimes \ldots \otimes \varphi_{n_d}^{\nu_d}, \qquad n \in (\mathbb{N}_+)^d.
$$
Other multi-dimensional objects, like the associated Laplacian, derivatives and Riesz transforms emerge via
a general scheme \cite{NoSt1}.

The $d$-dimensional counterparts of $\mathfrak{L}_{\nu}$ and $\mathfrak{L}_{\nu}^M$
(denoted here by the same symbols) have the decompositions
$$
\mathfrak{L}_{\nu} = \sum_{i=1}^d \Big( \lambda_{1,\nu_i}^2 + d_{\nu_i}^* d_{\nu_i}\Big), \qquad
\mathfrak{L}_{\nu}^{M} = \sum_{i=1}^d d_{\nu_i}^* d_{\nu_i},
$$
where $d_{\nu_i}$ and $d_{\nu_i}^*$ are the one-dimensional operators acting on $i$th axis.
The $d$-dimensional Riesz transforms are formally given by
$$
\mathfrak{R}_{\nu}^i = d_{\nu_i} (\mathfrak{L}_{\nu})^{-1/2}, \qquad
\big(\mathfrak{R}_{\nu}^M \big)^i = d_{\nu_i} \big( \mathfrak{L}_{\nu}^M \big)^{-1/2} \Pi_0, \qquad i = 1,\ldots,d,
$$
with $\Pi_0$ now denoting the orthogonal projection onto $\{\varphi_{(1,\ldots,1)}^{\nu}\}^{\perp}$.
These definitions are made strict in $L^2(d\eta_{\nu})$ in terms of orthogonal expansions, cf.\ \cite{NoSt1}, providing
bounded operators (in fact contractions) on $L^2(d\eta_{\nu})$.

The general result of \cite{wrobel} applies as well to the multi-dimensional essential measure Fourier-Bessel contexts.
It provides dimensionless and parameterless quantitative $L^p$ bounds for vectorial Riesz transforms under the
restriction $\nu \in [-1/2,\infty)^d$ (actually, under slightly bigger restriction $\nu \in (-1/2,\infty)^d$, but
then $\nu \in [-1/2,\infty)^d$ can be covered by a continuity argument).
Verification of all the necessary technical conditions is straightforward given what was already done in the proof
of Theorem \ref{thm:Riesz}, hence we leave the details to the reader. Here we only formulate the result.
\begin{theorem} \label{thm:Rieszd}
Let $d \ge 1$, $\nu \in [-1/2,\infty)^d$ and $1 < p < \infty$. The operators $\mathfrak{R}_{\nu}^i$ and
$(\mathfrak{R}_{\nu}^M)^i$, $i = 1,\ldots,d$, extend uniquely to bounded linear operators on $L^p(d\eta_{\nu})$.
Moreover,
\begin{align*}
\Big\| \big|\big(\mathfrak{R}_{\nu}^1,\ldots,\mathfrak{R}_{\nu}^d\big)\big|_{\ell^2} \Big\|_{L^p(d\eta_{\nu})\to L^p(d\eta_{\nu})}
& \le 48 (p^*-1), \\
\bigg\| \Big|\Big(\big(\mathfrak{R}_{\nu}^M\big)^1,\ldots,\big(\mathfrak{R}_{\nu}^M\big)^d\Big)\Big|_{\ell^2}
	\bigg\|_{L^p(d\eta_{\nu})\to L^p(d\eta_{\nu})}
& \le 48 (p^*-1),
\end{align*}
where $p^*=\max(p,p/(p-1))$.
\end{theorem}

According to our best knowledge, Theorem \ref{thm:Rieszd} and the result of Wr\'obel \cite[Theorem 4.1]{Wr}
having somewhat different flavor are the only $L^p$ results for Riesz transforms in
any multi-dimensional discrete Fourier-Bessel setting.

\subsection{Modified essential measure Fourier-Bessel setting} \label{ssec:mod_e}

In connection with Riesz transforms and Wr\'obel's result \cite{wrobel},
it is perhaps interesting to look briefly at yet another Fourier-Bessel setting that is close to the essential measure
Fourier-Bessel context. Roughly speaking, the idea is to normalize the system (see Section \ref{FBneu}) not by the
bottom eigenfunction, but by the most elementary function having the same boundary behavior. Consequently, the associated
measure and Laplacian have elementary forms (no Bessel functions involved).

Let
$$
\check{U}f(x) = f(x)/ (1-x)
$$
and for $\nu > -1$ consider the system
$$
\check{\varphi}_{n}^{\nu} = \check{U} \phi_n^{\nu}, \qquad n \in \mathbb{N}_+,
$$
which is an orthonormal basis in $L^2(d\rho_{\nu})$, where
$$
d\rho_{\nu}(x) = x^{2\nu+1}(1-x)^2\, dx, \qquad x \in (0,1).
$$
The corresponding Laplacian is given by
$$
\check{\mathfrak{L}}_{\nu} = \check{U} \mathcal{L}_{\nu} \check{U}^{-1}
 = - \frac{d^2}{dx^2} - \bigg( \frac{2\nu+1}{x} - \frac{2}{1-x} \bigg) \frac{d}{dx} + \frac{2\nu+1}{x(1-x)},
$$
and $\check{\mathfrak{L}}_{\nu} \check{\varphi}_n^{\nu} = \lambda_{n,\nu}^2 \check{\varphi}_n^{\nu}$, $n \ge 1$.
We have the decomposition
$$
\check{\mathfrak{L}}_{\nu} = \lambda_{1,\nu}^2 + \check{d}_{\nu}^* \check{d}_{\nu},
$$
with the derivative and its formal adjoint in $L^2(d\rho_{\nu})$
\begin{align*}
\check{d}_{\nu} & = \check{U}\delta_{\nu} \check{U}^{-1} = \frac{d}{dx} + R^{\nu}(x) - \frac{1}{1-x}, \\
\check{d}_{\nu}^* & = \check{U}\delta_{\nu}^* \check{U}^{-1} = - \frac{d}{dx} - \frac{2\nu+1}{x} + R^{\nu}(x) + \frac{1}{1-x}.
\end{align*}
Note that $\check{d}_{\nu}\check{\varphi}_n^{\nu} = (R^{\nu} - R_n^{\nu})\check{\varphi}_n^{\nu}$, $n \ge 1$.

The differentiated system $\{\check{\varphi}_n^{\nu} : n \ge 2\}$ is an orthogonal basis in $L^2(d\rho_{\nu})$
(this is justified by the results of Section \ref{sec:diff}) and
$\|\check{\varphi}_n^{\nu}\|_{L^2(d\rho_{\nu})} = (\lambda_{n,\nu}^2-\lambda_{1,\nu}^2)^{1/2}$.
The Riesz transform in this setting is given by (cf.\ \cite{NoSt1})
$$
\check{\mathfrak{R}}_{\nu}f = \sum_{n=2}^{\infty} \lambda_{n,\nu}^{-1}
	\big\langle f, \check{\varphi}_n^{\nu}\big\rangle_{d\rho_{\nu}} \check{d}_{\nu}\check{\varphi}_n^{\nu}, \qquad
		f \in L^2(d\rho_{\nu}).
$$
This operator is a contraction on $L^2(d\rho_{\nu})$.

It is noteworthy that the result of \cite{wrobel} can be applied also in the present situation, under the restriction
$\nu > -1/2$. This leads to the following.
\begin{theorem} \label{thm:Rieszm}
Let $\nu > -1/2$ and $1 < p < \infty$. The operator $\check{\mathfrak{R}}_{\nu}$ extends uniquely to a bounded
linear operator on $L^p(d\rho_{\nu})$. Moreover,
$$
\big\| \check{\mathfrak{R}}_{\nu} \big\|_{L^p(d\rho_{\nu})\to L^p(d\rho_{\nu})} \le 48 \big(1+\sqrt{K}\big) (p^*-1),
$$
where $K = \frac{1}8 \max(\nu+1/2, 1/(\nu+1/2))$ and $p^* = \max(p,p/(p-1))$.
\end{theorem}

\begin{proof}
Apply \cite[Theorem 1]{wrobel}. Given the facts and analysis presented so far in this paper, verification
of all the technical conditions needed is straightforward. We omit the details, except one thing. Namely, we show
that in \cite[Theorem 1]{wrobel} one can choose the constant $K$ as indicated in the theorem.

A simple computation reveals, see \cite[p.\,751]{wrobel}, that \cite[Condition (A2)]{wrobel} specifies to
\begin{equation} \label{cndA2me}
\bigg( R^{\nu}(x) - \frac{1}{1-x} \bigg)^2 \le K  \frac{2\nu+1}{x(1-x)}, \qquad x \in (0,1)
\end{equation}
(the relevant quantities involved in formulation of the above mentioned condition in \cite{wrobel}
are $a_1=\lambda_{1,\nu}^2$, $p_1(x) = 1$, $q_1(x) = R^{\nu}(x)-\frac{1}{1-x}$, $w_1(x) = x^{2\nu+1}(1-x)^2$
and $r(x) = r_1(x) = (2\nu+1)\slash [x(1-x)]$).
Let $F(x) = R^{\nu}(x) - \frac{1}{1-x}$. By \eqref{R1}, $F(x) = S^{\nu}(x) - \frac{1}{1+x}$, so
$F(0) = -1$ and $F(1) = \nu + 1/2$ (see Section \ref{sec:Bes}). Further, $F$ is increasing on $(0,1)$, as can be
concluded from \eqref{Rprim}. Thus $|F(x)| \le \max(\nu+1/2,1)$, $x \in (0,1)$. On the other hand,
$8(\nu+1/2) \le \frac{2\nu+1}{x(1-x)}$, since the minimum on $(0,1)$ of the right-hand side here is attained for $x = 1/2$.
Therefore, condition \eqref{cndA2me} is satisfied with
$$
K = \frac{\max^2(\nu+1/2,1)}{8(\nu+1/2)} = \frac{1}{8} \max\Big(\nu+1/2, \frac{1}{\nu+1/2}\Big).
$$

Note that any $K$ satisfying \eqref{cndA2me} must blow up as $\nu \to (-1/2)^+$, by the fact that
$R^{-1/2}(x) = \frac{\pi}2 \tan\frac{\pi x}2$ (see \eqref{Rexp}) and continuity of both sides of the inequality \eqref{cndA2me}
with respect to the parameter.
Thus the case $\nu = -1/2$ cannot be covered by a combination of \cite[Theorem 1]{wrobel} and a continuity argument.
\end{proof}

\begin{remark}
The modified essential measure Fourier-Bessel setting admits a canonical multi-dimensional generalization, where
Wr\'obel's result \cite{wrobel} still applies providing dimensionless quantitative $L^p$ bounds
for the associated vectorial Riesz transforms.
\end{remark}

\subsection{Comments on natural and Lebesgue measures Fourier-Bessel settings} \label{ssec:comm_Riesz}

As it was already mentioned, Riesz transforms in the natural and Lebesgue measures Fourier-Bessel
settings are not covered in general by the results of \cite{wrobel,FSS}. However, there is a special case
where the situation is different. Namely, for $\nu = 1/2$ the Lebesgue measure Fourier-Bessel setting coincides with
the Jacobi one with the parameters $\alpha=\beta=1/2$. Riesz transform defined by means of our derivative $\mathbb{D}_{\nu}$
in the Lebesgue measure Fourier-Bessel context is the same as the Jacobi Riesz transform for the parameters $\alpha=\beta=1/2$.
So from Wr\'obel's result (see \cite[Theorem 12]{wrobel}) we infer that the Lebesgue measure Fourier-Bessel Riesz transform
for $\nu = 1/2$ is $L^p$-bounded, $1 < p < \infty$, and its $L^p$ norm is not greater than $48 (p^*-1)$.

Another observation is that for $\nu = -1/2$ the natural and Lebesgue measures Fourier-Bessel settings are the same
and they coincide with the Jacobi situation with the parameters $(\alpha,\beta) = (-1/2, 1/2)$. Again the Riesz transform
is the same operator in all the contexts mentioned (the relevant derivative being $\delta_{-1/2}$),
and its $L^p$-boundedness can be inferred from e.g.\ the result of Stempak \cite{St}.

Concerning $L^p$-boundedness of Riesz transforms defined by means of our new derivatives
in the natural and Lebesgue measures Fourier-Bessel contexts
for general $\nu$, we believe that an efficient approach could be based on the Calder\'on-Zygmund operator theory
and its variants. This means analysis in the spirit of \cite{CiRo4,CiSt2,CiSt3}, but more sophisticated in our
situations since the new derivatives and differentiated systems are more difficult to handle. The problem remains open.

\section{Differentiated semigroups} \label{sec:diff_semi}

In this section we study semigroups related to the differentiated Fourier-Bessel systems, the main focus being put on
the essential measure contexts. Our aims are
to estimate integral kernels of the semigroups and then to establish maximal theorems.
The latter will be needed in investigation of Fourier-Bessel Sobolev spaces in Section \ref{sec:Sob}.
The main idea we shall use is a comparison of the Jacobi and Fourier-Bessel Lebesgue measure
situations which is inspired by \cite{NoRo2}.
Thus we begin with gathering necessary facts in the Jacobi context and then concentrate for a while on the Lebesgue measure
Fourier-Bessel setting.

\subsection{Jacobi trigonometric Lebesgue measure setting scaled to $(0,1)$} \label{ssec:jac_diff}

In this framework the differentiated system $\{D_{\a,\b} \Phi_k^{\a,\b}\}$ coincides with the original one, but with shifted
indices; this is an instant consequence of \eqref{jdfe}. Thus it is clear that $\{D_{\a,\b}\Phi_k^{\a,\b} : k \ge 1\}$
is an orthogonal basis in $L^2(dx)$. Moreover,
$\|D_{\a,\b}\Phi_k^{\a,\b}\|_{L^2(dx)} = \pi \sqrt{k(k+\a+\b+1)}$, see \cite[Lemma 2]{NoSt1}.

The Laplacian associated with the differentiated system is, see \cite[Section 5]{NoSt1},
$M_{\a,\b} = \mathbb{J}_{\a,\b} + [D_{\a,\b},D_{\a,\b}^*]$.
A simple computation shows that
$$
[D_{\a,\b},D_{\a,\b}^*] = \frac{\pi^2 (\a+1/2)}{2 \sin^2 \frac{\pi x}2} + \frac{\pi^2 (\b+1/2)}{2 \cos^2\frac{\pi x}2},
$$
therefore
$$
M_{\a,\b} = - \frac{d^2}{dx^2} + \frac{\pi^2 (\alpha+3/2)(\alpha+1/2)}{4 \sin^2\frac{\pi x}2}
						+ \frac{\pi^2 (\beta+3/2)(\beta+1/2)}{4 \cos^2\frac{\pi x}2}
					= \mathbb{J}_{\a+1,\b+1}.
$$
Furthermore, $M_{\a,\b}(D_{\a,\b}\Phi_{k}^{\a,\b}) = \pi^2(k+\frac{\a+\b+1}2)^2 D_{\a,\b} \Phi_k^{\a,\b}$, $k \ge 1$.

In this context, there exists a natural self-adjoint extension of $M_{\a,\b}$ considered initially on the dense subspace of $L^2(dx)$
spanned by $\{D_{\a,\b}\Phi_k^{\a,\b} : k \ge 1\}$. This extension is defined spectrally in $L^2(dx)$ in a canonical way, and
we denote it by the same symbol $M_{\a,\b}$. As easily verified, the semigroup $\{\exp(-t M_{\a,\b})\}_{t \ge 0}$ has for $t > 0$ an
integral representation in $L^2(dx)$, with the integral kernel given by (see \eqref{jdfe})
\begin{align*}
H_t^{\a,\b}(x,y) & = \sum_{k=1}^{\infty} e^{-\pi^2(k+\frac{\a+\b+1}2)^2 t}
	\frac{D_{\a,\b}\Phi_k^{\a,\b}(x) D_{\a,\b}\Phi_k^{\a,\b}(y)}{\|D_{\a,\b}\Phi_k^{\a,\b}\|_{L^2(dx)}^2} \\
	& = \sum_{k=0}^{\infty} e^{-\pi^2(k+\frac{\a+1+\b+1+1}2)^2 t} \Phi_k^{\a+1,\b+1}(x) \Phi_k^{\a+1,\b+1}(y)
	= \mathbb{G}_t^{\a+1,\b+1}(x,y),
\end{align*}
where $\mathbb{G}_t^{\a,\b}(x,y)$ is the integral kernel of the Jacobi semigroup related to $\mathbb{J}_{\a,\b}$.

Sharp estimates of the latter kernel, thus also of the former one, are known.
We have the following, see \cite[Theorem 3.1]{NSS} and also \cite[Theorem 2.1]{NoRo2}.
\begin{theorem}[{\cite{NSS}}] \label{thm:jachkest}
Assume that $\a,\b \ge -1/2$. Given any fixed $T>0$,
$$
\mathbb{G}_t^{\a,\b}(x,y) \simeq \bigg[ 1 \wedge \frac{xy}t\bigg]^{\a+1/2}
		\bigg[ 1 \wedge \frac{(1-x)(1-y)}t\bigg]^{\b+1/2} \frac{1}{\sqrt{t}} \exp\bigg( -\frac{(x-y)^2}{4t} \bigg),
$$
uniformly in $x,y \in (0,1)$ and $0 < t \le T$.
\end{theorem}

\begin{remark}
Theorem \ref{thm:jachkest} remains true if either $\a > -1$ and $\b=-1/2$ or $\a=-1/2$ and $\b > 1$; see \cite[Remark 3.2]{NSS}
and also \cite{MSZ,NoRo2}. Moreover, the result is known to be true for any $\a,\b>-1$ provided that the constant $4$ in the
exponential factor is replaced by some other, sufficiently small or sufficiently big positive constants in the
lower and upper estimates, respectively; see \cite[Theorem 7.2]{CKP}; see also \cite{NoSj}.
\end{remark}

\begin{remark}
The large time estimates for $\mathbb{G}_t^{\a,\b}(x,y)$ are known as well. For any $\a,\b > -1$ and $T > 0$ fixed one has
$$
\mathbb{G}_t^{\a,\b}(x,y) \simeq e^{-(\frac{\a+\b+1}2)^2 t} (xy)^{\a+1/2}
	\big[(1-x) (1-y)\big]^{\b+1/2}, \qquad x,y \in (0,1), \quad t \ge T.
$$
See e.g.\ \cite[p.\,233 and (3)]{NoSj}.
\end{remark}

\subsection{Fourier-Bessel Lebesgue measure setting} \label{ssec:FBL_diff}

Recall that $\mathbb{D}_{\nu}\psi_n^{\nu} = (R^{\nu}-R_n^{\nu})\psi_n^{\nu}$ and $\{\mathbb{D}_{\nu}\psi_n^{\nu} : n \ge 2\}$
is an orthogonal basis in $L^2(dx)$. Further,
$\|\mathbb{D}_{\nu}\psi_n^{\nu}\|_{L^2(dx)} = (\lambda_{n,\nu}^2-\lambda_{1,\nu}^2)^{1/2}$.

The Laplacian associated with the differentiated system is
$\mathbb{M}_{\nu} = \mathbb{L}_{\nu} + [\mathbb{D}_{\nu},\mathbb{D}_{\nu}^*]$. Since
$$
[\mathbb{D}_{\nu},\mathbb{D}_{\nu}^*] = \frac{2\nu+1}{x^2} + 2\big(R^{\nu}\big)'(x),
$$
it follows that
\begin{equation} \label{Mdif}
\mathbb{M}_{\nu} = -\frac{d^2}{dx^2} + \frac{(\nu+3/2)(\nu+1/2)}{x^2} + 2 \big(R^{\nu}\big)'(x).
\end{equation}
Note that $\mathbb{M}_{\nu}(\mathbb{D}_{\nu}\psi_n^{\nu}) = \lambda_{n,\nu}^2 \mathbb{D}_{\nu}\psi_n^{\nu}$, $n \ge 2$.
We denote by the same symbol $\mathbb{M}_{\nu}$ the natural in this context self-adjoint extension in $L^2(dx)$ of $\mathbb{M}_{\nu}$
acting initially on $\spann\{\mathbb{D}_{\nu}\psi_n^{\nu} : n \ge 2\}$.

The semigroup $\{\exp(-t \mathbb{M}_{\nu})\}_{t \ge 0}$ has for $t > 0$ an integral representation in $L^2(dx)$ with the
integral kernel
$$
\mathbb{H}_t^{\nu}(x,y) = \sum_{n=2}^{\infty} e^{-\lambda_{n,\nu}^2 t}
	\frac{\mathbb{D}_{\nu}\psi_n^{\nu}(x) \mathbb{D}_{\nu}\psi_n^{\nu}(y)}{\|\mathbb{D}_{\nu}\psi_n^{\nu}\|_{L^2(dx)}^2}
= \sum_{n=2}^{\infty} e^{-\lambda_{n,\nu}^2 t}
	\frac{(R^{\nu}-R_n^{\nu})(x)\, (R^{\nu}-R_n^{\nu})(y)}{\lambda_{n,\nu}^2-\lambda_{1,\nu}^2} \psi_n^{\nu}(x) \psi_n^{\nu}(y).
$$
With the aid of Lemma \ref{lem:ues} it is straightforward to check that $\mathbb{H}_t^{\nu}(x,y)$ is
a continuous function of $(x,y,t) \in (0,1)^2 \times (0,\infty)$. A more detailed analysis shows that it is actually
smooth, see Remark \ref{rem:smooth} below.
Unlike in the Jacobi case, this kernel cannot be expressed in a reasonable way in terms of the kernel of
the non-differentiated semigroup $\{\exp(-t\mathbb{L}_{\nu})\}$.

To estimate $\mathbb{H}_t^{\nu}(x,y)$ we employ the method from \cite{NoRo2}. Namely, we will show that the generators
$\mathbb{M}_{\nu}$ and $M_{\nu,1/2}$ differ only by a bounded zero order term and, moreover, their
domains are the same. This together with the Trotter product formula will imply that the corresponding heat kernels
are comparable in finite time intervals.

We shall need some technical preparation. Observe that
$$
\mathbb{M}_{\nu} - M_{\nu,1/2} = \Big( \nu + \frac{3}2\Big) \Big( \nu + \frac{1}2\Big)
	\bigg[ \frac{1}{x^2} - \frac{\pi^2}{4 \sin^2\frac{\pi x}2} \bigg] +
	2\bigg[\big(R^{\nu}\big)'(x) - \frac{\pi^2}{4\cos^2\frac{\pi x}2}\bigg]
	=: F^{\nu}(x).
$$
\begin{lemma} \label{lem:Fbd}
For any $\nu > -1$ fixed the function $F^{\nu}$ is bounded on $[0,1]$, the values $F^{\nu}(\pm 1)$ being understood in
a limiting sense.
\end{lemma}

\begin{proof}
In what follows some expressions are understood in an obvious limiting sense.
It is not hard to check, see \cite[p.\,441]{NoRo2}, that
\begin{equation} \label{pin2}
- \frac{\pi^2}4 + 1 \le  \bigg[ \frac{1}{x^2} - \frac{\pi^2}{4\sin^2\frac{\pi x}2}\bigg]  \le 0, \qquad x \in [0,1].
\end{equation}
Further, using \eqref{Rprim} we can write
$$
\bigg[\big(R^{\nu}\big)'(x) - \frac{\pi^2}{4\cos^2\frac{\pi x}2}\bigg] =
	\bigg(\frac{1}{(1-x)^2} - \frac{\pi^2}{4\cos^2\frac{\pi x}2} \bigg)
	+ \frac{1}{(1+x)^2} + 2\lambda_{1,\nu}^2 \sum_{k=2}^{\infty}
		\frac{\lambda_{k,\nu}^2 + \lambda_{1,\nu}^2 x^2}{(\lambda_{k,\nu}^2 - \lambda_{1,\nu}^2 x^2)^2}.
$$
Here we can easily bound most of the terms,
$$
- \frac{\pi^2}4 + 1 \le  \bigg( \frac{1}{(1-x)^2} - \frac{\pi^2}{4\cos^2\frac{\pi x}2}\bigg)  \le 0, \qquad
\frac{1}4 \le \frac{1}{(1+x)^2} \le 1, \qquad x \in [0,1];
$$
the first part follows from \eqref{pin2} by reflecting $x \mapsto 1-x$, and the second one is trivial.

As for the remaining term involving the series, we observe that it is an increasing function of $x$, hence it is controlled
by its value at $x=1$. Combining this with Calogero's formulae \eqref{Calo} and \eqref{Calo2} we get
\begin{align*}
0 < 2\lambda_{1,\nu}^2 \sum_{k=2}^{\infty}
		\frac{\lambda_{k,\nu}^2 + \lambda_{1,\nu}^2 x^2}{(\lambda_{k,\nu}^2 - \lambda_{1,\nu}^2 x^2)^2}
& \le 2\lambda_{1,\nu}^2 \sum_{k=2}^{\infty}
		\frac{\lambda_{k,\nu}^2 + \lambda_{1,\nu}^2}{(\lambda_{k,\nu}^2 - \lambda_{1,\nu}^2)^2} 
 = \sum_{k=2}^{\infty} \frac{2\lambda_{1,\nu}^2}{\lambda_{k,\nu}^2 - \lambda_{1,\nu}^2} 
	+ \sum_{k=2}^{\infty} \frac{4\lambda_{1,\nu}^4}{(\lambda_{k,\nu}^2-\lambda_{1,\nu}^2)^2} \\
&	= \frac{\lambda_{1,\nu}^2 - (\nu+1)(\nu+2)}{3}, \qquad x \in [0,1].
\end{align*}
The conclusion follows.
\end{proof}

Another crucial technical result we need is the following.
\begin{lemma} \label{lem:mixsym}
Let $\nu > -1$. Then
$$
\Big\langle \mathbb{M}_{\nu}\big(\mathbb{D}_{\nu}\psi_n^{\nu}\big), D_{\nu,1/2}\Phi_k^{\nu,1/2}\Big\rangle
= \Big\langle \mathbb{D}_{\nu}\psi_n^{\nu}, \mathbb{M}_{\nu}\big(D_{\nu,1/2}\Phi_k^{\nu,1/2}\big)\Big\rangle,
\qquad n \ge 2, \quad k \ge 1,
$$
where $\mathbb{M}_{\nu}$ on the right-hand side is the differential operator \eqref{Mdif}.
\end{lemma}

\begin{proof}
What we need to show can be written as, see \eqref{jdfe},
$$
\Big\langle \mathbb{M}_{\nu}
\mathbb{D}_{\nu}\psi_n^{\nu}, \Phi_k^{\nu+1,3/2}\Big\rangle
= \Big\langle \mathbb{D}_{\nu}\psi_n^{\nu}, \mathbb{M}_{\nu}\Phi_k^{\nu+1,3/2}\Big\rangle, \qquad n \ge 2, \quad k \ge 0.
$$
We will use the divergence form of the differential operator $\mathbb{M}_{\nu}$,
\begin{equation} \label{divM}
\mathbb{M}_{\nu}f = - \psi_1^{\nu} \frac{d}{dx} \bigg[ \frac{1}{(\psi_1^{\nu})^2} \frac{d}{dx} \big(\psi_1^{\nu}f\big)\bigg],
\end{equation}
and integrate by parts. Denote
$$
\mathcal{I} = \Big\langle \mathbb{M}_{\nu}\mathbb{D}_{\nu}\psi_n^{\nu} , \Phi_k^{\nu+1,3/2}\Big\rangle
 = \int_0^1 \mathbb{M}_{\nu}\mathbb{D}_{\nu}\psi_n^{\nu}(x) \Phi_k^{\nu+1,3/2}(x)\, dx.
$$
Notice that this integral converges since it represents inner product of two $L^2(dx)$ functions.

Integrating twice by parts, with the aid of \eqref{divM}
(see the proof of \cite[Lemma 3.2]{NoRo2} for a completely analogous computation), we arrive at the identity
$$
\mathcal{I} = \mathcal{I}_1 + \mathcal{I}_2 + \mathcal{I}_3,
$$
where (skipping for the sake of brevity the argument $x$)
\begin{align*}
\mathcal{I}_1 & = -\frac{1}{\psi_1^{\nu}} \frac{d}{dx} \Big( \psi_1^{\nu} 
	\mathbb{D}_{\nu}\psi_n^{\nu}\Big) \Phi_k^{\nu+1,3/2}\Big|_0^1, \\
\mathcal{I}_2 & =  \frac{\mathbb{D}_{\nu}\psi_n^{\nu}}{\psi_1^{\nu}}
	\frac{d}{dx} \Big( \psi_1^{\nu}\Phi_k^{\nu+1,3/2}\Big)\Big|_{0}^1, \\
\mathcal{I}_3 & = \Big\langle \mathbb{D}_{\nu}\psi_n^{\nu}, \mathbb{M}_{\nu}\Phi_k^{\nu+1,3/2}\Big\rangle.
\end{align*}
Since the desired identity is $\mathcal{I} = \mathcal{I}_3$, it is enough we check that
$\mathcal{I}_1 = \mathcal{I}_2 = 0$.

We have, see \eqref{asy3}, Proposition \ref{prop:asd} and the definition of $\Phi_k^{\a,\b}$,
$$
\bigg| \frac{\Phi_k^{\nu+1,3/2}(x)}{\psi_1^{\nu}(x)}\bigg| \lesssim x(1-x), \qquad
\bigg| \frac{\mathbb{D}_{\nu}\psi_n^{\nu}(x)}{\psi_1^{\nu}(x)}\bigg| \lesssim x(1-x), \qquad x \in (0,1).
$$
In view of the above, our task will be done once we prove the bound (this is actually even more than we need)
$$
\bigg| \frac{d}{dx} \Big( \psi_1^{\nu}(x) \mathbb{D}_{\nu}\psi_n^{\nu}(x) \Big) \bigg| +
	\bigg| \frac{d}{dx} \Big( \psi_1^{\nu}(x) \Phi_k^{\nu+1,3/2}(x) \Big) \bigg| \lesssim x^{2\nu+1}(1-x)^2, \qquad x \in (0,1).
$$
Denote the first term on the left-hand side here by $\mathcal{J}_1$, and the second one by $\mathcal{J}_2$.

Taking into account the forms of $\mathbb{D}_{\nu}$ and $\mathbb{D}_{\nu}^*$ we have
\begin{align*}
\mathcal{J}_1 & \le \bigg| \frac{d}{dx} \psi_1^{\nu}(x) \mathbb{D}_{\nu}\psi_n^{\nu}(x)\bigg|
+ \bigg| \psi_1^{\nu}(x) \frac{d}{dx}\mathbb{D}_{\nu}\psi_n^{\nu}(x) \bigg| \\
& \le \big| \mathbb{D}_{\nu}\psi_1^{\nu}(x) \mathbb{D}_{\nu} \psi_n^{\nu}(x) \big|
	+ \big| \psi_1^{\nu}(x) \mathbb{D}_{\nu}^* \mathbb{D}_{\nu}\psi_n^{\nu}(x) \big|
	+ 2\bigg| \bigg( \frac{\nu+1/2}{x} - R^{\nu}(x)\bigg) \psi_1^{\nu}(x) \mathbb{D}_{\nu}\psi_n^{\nu}(x) \bigg|.
\end{align*}
Recall that $\mathbb{D}_{\nu}^* \mathbb{D}_{\nu} \psi_n^{\nu} = (\lambda_{n,\nu}^2-\lambda_{1,\nu}^2)\psi_n^{\nu}$.
Further, see \eqref{Ras},
\begin{equation} \label{eq:pot7}
\bigg| \frac{\nu+1/2}x -R^{\nu}(x)\bigg| \lesssim \frac{1}{x(1-x)}, \qquad x \in (0,1).
\end{equation}
Using this together with \eqref{asy3} and Proposition \ref{prop:asd} we arrive at the bound
$$
\mathcal{J}_1 \lesssim x^{2\nu+1} (1-x)^2, \qquad x \in (0,1).
$$

It remains to treat $\mathcal{J}_2$. We argue similarly as in case of $\mathcal{J}_1$ getting
\begin{align*}
\mathcal{J}_2 & \le \bigg| \frac{d}{dx} \psi_1^{\nu}(x) \Phi_k^{\nu+1,3/2}(x)\bigg|
	+ \bigg| \psi_1^{\nu}(x) \frac{d}{dx} \Phi_k^{\nu+1,3/2}(x)\bigg| \\
& \le \big| \mathbb{D}_{\nu}\psi_1^{\nu}(x) \Phi_k^{\nu+1,3/2}(x)\big|
	+ \bigg| \bigg( \frac{\nu+1/2}{x} - R^{\nu}(x)\bigg) \Phi_k^{\nu+1,3/2}(x)\bigg| \\
& \quad	+ \big| \psi_1^{\nu}(x) \mathbb{D}_{\nu+1,3/2} \Phi_k^{\nu+1,3/2}(x) \big|
	+ \bigg| \psi_1^{\nu}(x) \bigg( \pi \frac{2\nu+3}{4}\cot\frac{\pi x}2 - \pi \tan\frac{\pi x}2\bigg) \Phi_k^{\nu+1,3/2}(x)\bigg|.
\end{align*}
Recall that $|\Phi_k^{\a,\b}(x)| \lesssim x^{\a+1/2}(1-x)^{\b+1/2}$, $x \in (0,1)$.
This combined with \eqref{jdfe}, \eqref{eq:pot7}, \eqref{asy3} and Proposition \ref{prop:asd} leads to the bound
$$
\mathcal{J}_2 \lesssim x^{2\nu+1} (1-x)^2, \qquad x \in (0,1).
$$
This finishes the proof.
\end{proof}

Now, with Lemmas \ref{lem:Fbd} and \ref{lem:mixsym} at our disposal, we can prove the following.
\begin{lemma} \label{lem:dom}
Let $\nu > -1$. The domains of the self-adjoint operators $\mathbb{M}_{\nu}$ and $M_{\nu,1/2}$ coincide,
$$
\domain \mathbb{M}_{\nu} = \domain M_{\nu,1/2}.
$$
\end{lemma}

\begin{proof}
The arguments are exactly the same as in the proof of \cite[Theorem 3.1]{NoRo2}, hence we omit the details.
\end{proof}

As a consequence of Lemmas \ref{lem:Fbd} and \ref{lem:dom} combined with the Trotter product formula
(see \cite[p.\,443]{NoRo2}) we infer that
$$
e^{-c_{\nu}t} H_t^{\nu,1/2}(x,y) \le \mathbb{H}_t^{\nu}(x,y) \le e^{c_{\nu}t} H_t^{\nu,1/2}(x,y), \qquad x,y \in (0,1), \quad t > 0,
$$
where $c_{\nu} = \sup_{x\in (0,1)}|F^{\nu}(x)| < \infty$.
Since $H_t^{\nu,1/2}(x,y) = \mathbb{G}_t^{\nu+1,3/2}(x,y)$, this leads, via Theorem \ref{thm:jachkest}, to sharp short time
estimates of the kernel $\mathbb{H}_t^{\nu}(x,y)$.
\begin{theorem} \label{thm:Hsest}
Let $\nu > -1$. Given any $T > 0$ fixed,
$$
\mathbb{H}_t^{\nu}(x,y) \simeq \bigg[ 1 \wedge \frac{xy}t\bigg]^{\nu+3/2}
		\bigg[ 1 \wedge \frac{(1-x)(1-y)}t\bigg]^{2} \frac{1}{\sqrt{t}} \exp\bigg( -\frac{(x-y)^2}{4t} \bigg),
$$
uniformly in $x,y \in (0,1)$ and $0 < t \le T$.
\end{theorem}
Note that $\mathbb{H}_t^{\nu}(x,y)$ is strictly positive for any $t > 0$ and any $x,y \in (0,1)$, since
$H_t^{\nu,1/2}(x,y)$ has this property.

\subsection{Fourier-Bessel essential measure settings} \label{ssec:6ess}

Let us start with heat semigroups related to $\mathfrak{L}_{\nu}$ and $\mathfrak{L}_{\nu}^M$,
self-adjoint operators naturally associated with the system $\{\varphi_n^{\nu} : n \ge 1\}$, see Section \ref{FBneu}.
The semigroups
$$
\mathfrak{T}_t^{\nu} := \exp(-t\mathfrak{L}_{\nu}) \qquad \textrm{and} \qquad
	\mathfrak{T}_t^{\nu,M} := \exp\big(-t\mathfrak{L}^M_{\nu}\big), \qquad t \ge 0,
$$
have for $t>0$ integral representations in $L^2(d\eta_{\nu})$.
We denote the corresponding integral kernels by
$\mathfrak{G}_t^{\nu}(x,y)$ and $\mathfrak{G}_t^{\nu,M}(x,y)$, respectively, see \eqref{hessint} below.
Further, let $G_t^{\nu}(x,y)$ stand
for the heat kernel in the natural measure Fourier-Bessel setting. It is straightforward to check that
\begin{equation} \label{GK10}
\mathfrak{G}_t^{\nu}(x,y) = \frac{1}{\phi_1^{\nu}(x)\phi_1^{\nu}(y)} G_t^{\nu}(x,y), \qquad 
\mathfrak{G}_t^{\nu,M}(x,y) = e^{\lambda_{1,\nu}^2 t} \mathfrak{G}_t^{\nu}(x,y).
\end{equation}
Recall that $\phi_1^{\nu}(x) \simeq 1-x$, $x \in (0,1)$, see \eqref{asy2}.
Using then sharp estimates for $G_t^{\nu}(x,y)$ that were established in \cite{MSZ}, see also \cite{NoRo1,NoRo2}, we get the
following.
\begin{proposition} \label{prop:heess}
Let $\nu > -1$. Given any $T > 0$ fixed,
\begin{equation} \label{heess}
\mathfrak{G}_t^{\nu}(x,y) \simeq \mathfrak{G}_t^{\nu,M}(x,y) \simeq
	(t \vee xy)^{-\nu-1/2} \big[ t \vee (1-x)(1-y) \big]^{-1} \frac{1}{\sqrt{t}} \exp\bigg( -\frac{(x-y)^2}{4t}\bigg),
\end{equation}
uniformly in $x,y \in (0,1)$ and $0 < t \le T$. Moreover,
$$
\mathfrak{G}_t^{\nu}(x,y) \simeq e^{-\lambda_{1,\nu}^2 t}, \qquad
\mathfrak{G}_t^{\nu,M}(x,y) \simeq 1, \qquad x,y \in (0,1), \quad t \ge T.
$$
\end{proposition}
In view of the above result, the integral representations
\begin{equation} \label{hessint}
\mathfrak{T}_t^{\nu}f(x) = \int_0^1 \mathfrak{G}_t^{\nu}(x,y)f(y)\, d\eta_{\nu}(y), \qquad
\mathfrak{T}_t^{\nu,M}f(x) = \int_0^1 \mathfrak{G}_t^{\nu,M}(x,y)f(y)\, d\eta_{\nu}(y),
\end{equation}
extend the actions of the semigroups to $L^1(d\eta_{\nu})$, hence to all $L^p(d\eta_{\nu})$, $p \ge 1$.
Bring in the corresponding maximal operators
$$
\mathfrak{T}_{*}^{\nu}f = \sup_{t > 0}\big|\mathfrak{T}_t^{\nu}f\big|, \qquad
\mathfrak{T}_{*}^{\nu,M}f = \sup_{t > 0}\big|\mathfrak{T}_t^{\nu,M}f\big|.
$$
\begin{theorem} \label{thm:maxhess}
Let $\nu > -1$. The maximal operators $\mathfrak{T}_{*}^{\nu}$ and $\mathfrak{T}_{*}^{\nu,M}$ are bounded on
$L^p(d\eta_{\nu})$, $1 < p \le \infty$, and from $L^1(d\eta_{\nu})$ to weak $L^1(d\eta_{\nu})$.
\end{theorem}

\begin{proof}
From Proposition \ref{prop:heess} we see that the parts of the maximal operators corresponding to $t \ge T$ can be easily
handled. On the other hand, the desired mapping properties of the complementary parts for $t < T$ are concluded
from maximal results in the Jacobi trigonometric natural measure setting (scaled to $(0,1)$, to be precise), see
\cite[Section 5]{NoSj} and also \cite[Section 5]{NSS0}. This is because the sharp short time bounds \eqref{heess}
are essentially the same as in the natural measure Jacobi case with indices $\alpha = \nu$ and $\beta = 1/2$
(see e.g.\ \cite[Theorem A]{NoSj} or \cite[Theorem 5.1]{NSS0}) and the corresponding measures,
i.e.\ $d\eta_{\nu}$ and the Jacobi measure, are comparable.

Alternatively, here one could also argue more directly by observing that from \eqref{heess} it follows that the kernels satisfy
Gaussian upper bounds with respect to the space of homogeneous type $((0,1),|\cdot|,d\eta_{\nu})$
(with $|\cdot|$ standing for the Euclidean distance), cf.\ comments succeeding \cite[Theorem 5.4]{NSS0}.
We leave the details to interested readers.
\end{proof}

\begin{corollary} \label{cor:maxhess}
Let $\nu > -1$ and $ 1 \le p \le \infty$. For each $f \in L^p(d\eta_{\nu})$,
$$
\lim_{t \to 0^+} \mathfrak{T}_{t}^{\nu}f = \lim_{t \to 0^+} \mathfrak{T}_{t}^{\nu,M}f = f \quad \textrm{a.e.}
$$
\end{corollary}

\begin{proof}
It is enough to consider $p=1$, since the other $L^p$ spaces are contained in $L^1(d\eta_{\nu})$.
Then the arguments are standard. One uses Theorem \ref{thm:maxhess} and the linear density of
$\{\varphi_n^{\nu} : n \ge 1\}$ in $L^1(d\eta_{\nu})$, cf.\ Remark \ref{rem:densC0}.
Further details are well known, see e.g.\ \cite[Chapter 2, Section 2]{Duo}, and hence omitted.
\end{proof}

We now pass to the differentiated semigroups. They are defined analogously as in the Lebesgue measure
Fourier-Bessel setting, see \cite[Section 5]{NoSt1}. Thus the generators are the operators
$$
\mathfrak{M}_{\nu} = \mathfrak{L}_{\nu} + [d_{\nu},d_{\nu}^*], \qquad 
\mathfrak{M}^M_{\nu} = \mathfrak{L}^M_{\nu} + [d_{\nu},d_{\nu}^*],
$$
or rather their natural self-adjoint extensions related to the differentiated system $\{d_{\nu}\varphi_n^{\nu} : n \ge 2\}$
which is an orthogonal basis in $L^2(d\eta_{\nu})$. One has
$$
\mathfrak{M}_{\nu} \big( d_{\nu}\varphi_n^{\nu} \big) = \lambda_{n,\nu}^2 d_{\nu}\varphi_n^{\nu} \qquad \textrm{and} \qquad
\mathfrak{M}^M_{\nu} \big( d_{\nu}\varphi_n^{\nu} \big) = \big(\lambda_{n,\nu}^2-\lambda_{1,\nu}^2\big) d_{\nu}\varphi_n^{\nu},
\qquad n \ge 2.
$$
Recall that $d_{\nu}\varphi_n^{\nu} = (R^{\nu}-R_n^{\nu})\varphi_n^{\nu}$.
The integral kernels of
$$
\mathfrak{H}_t^{\nu} := \exp\big(-t \mathfrak{M}_{\nu}\big) \qquad \textrm{and} \qquad
\mathfrak{H}_t^{\nu,M} := \exp\big(-t \mathfrak{M}^M_{\nu}\big)
$$
are, respectively,
\begin{equation} \label{serHess}
\mathfrak{H}_t^{\nu}(x,y) := \sum_{n=2}^{\infty} e^{-\lambda_{n,\nu}^2 t}
	\frac{d_{\nu} \varphi_n^{\nu}(x) d_{\nu}\varphi_n^{\nu}(y)}{\|d_{\nu}\varphi_n^{\nu}\|_{L^2(d\eta_{\nu})}^2}
	= \frac{1}{\psi_1^{\nu}(x)\psi_1^{\nu}(y)} \mathbb{H}_t^{\nu}(x,y)
\end{equation}
and $\mathfrak{H}_t^{\nu,M}(x,y) := e^{\lambda_{1,\nu}^2 t} \mathfrak{H}_t^{\nu}(x,y)$.
Since these kernels are directly related to $\mathbb{H}_t^{\nu}(x,y)$, their sharp short time estimates follow from
Theorem~\ref{thm:Hsest} and the bounds $\psi_1^{\nu}(x) \simeq x^{\nu+1/2}(1-x)$, $x \in (0,1)$.
\begin{proposition} \label{prop:HsEest}
Let $\nu > -1$. Given any fixed $T>0$,
\begin{align*}
\mathfrak{H}_t^{\nu}(x,y) & \simeq \mathfrak{H}_t^{\nu,M}(x,y) \\ & \simeq \bigg[ 1 \wedge \frac{xy}t\bigg]
	\bigg[ 1 \wedge \frac{(1-x)(1-y)}t\bigg] (t \vee xy)^{-\nu-1/2} \big[ t \vee (1-x)(1-y)\big]^{-1}
	\frac{1}{\sqrt{t}} \exp\bigg( -\frac{(x-y)^2}{4t} \bigg),
\end{align*}
uniformly in $x,y \in (0,1)$ and $0 < t \le T$.
\end{proposition}

Considering the large time bounds, we have the following.
\begin{proposition} \label{prop:HlEest}
Let $\nu > -1$. Given any $T > 0$ fixed,
$$
0 < \mathfrak{H}_t^{\nu}(x,y) \lesssim e^{-\lambda_{2,\nu}^2 t}, \qquad
0 < \mathfrak{H}_t^{\nu,M}(x,y) \lesssim e^{-(\lambda_{2,\nu}^2-\lambda_{1,\nu}^2) t}, \qquad x,y \in (0,1), \quad t \ge T.
$$
\end{proposition}

\begin{proof}
The lower bounds follow from strict positivity of $\mathbb{H}_t^{\nu}(x,y)$, so it remains to show the upper bounds.
Clearly, we can restrict our attention to $\mathfrak{H}_t^{\nu}(x,y)$.

Using \eqref{serHess} together with Lemma \ref{lem:ues} and \eqref{eigvas}
(recall that $\|d_{\nu}\varphi_n^{\nu}\|_{L^2(d\eta_{\nu})}^2 = \lambda_{n,\nu}^2-\lambda_{1,\nu}^2$), we get
$$
\mathfrak{H}_t^{\nu}(x,y) \lesssim e^{-\lambda_{2,\nu}^2 t} + \e^{-\lambda_{2,\nu}^2 t}
	\sum_{n=3}^{\infty} e^{-c n^2 T} n^{2\nu + 9} \lesssim e^{-\lambda_{2,\nu}^2 t}, \qquad x,y \in (0,1), \quad t \ge T,
$$
where $c>0$ is a constant such that $\lambda_{n,\nu}^2 - \lambda_{2,\nu}^2 \ge c n^2$ for $n \ge 3$.
The conclusion follows.
\end{proof}

\begin{remark}
A sharp variant of the bounds from Proposition \ref{prop:HlEest} is
\begin{align*}
\mathfrak{H}_t^{\nu}(x,y) & \simeq e^{-\lambda_{2,\nu}^2 t} x(1-x) y(1-y), \qquad x,y \in (0,1), \quad t \ge T, \\
\mathfrak{H}_t^{\nu,M}(x,y) & \simeq e^{-(\lambda_{2,\nu}^2-\lambda_{1,\nu}^2) t} x(1-x) y(1-y), \qquad x,y \in (0,1), \quad t \ge T.
\end{align*}
However, proving this would require a longer and more subtle analysis.
Since this is not relevant for our developments, we shall not pursue this matter.
\end{remark}

As a straightforward consequence of Propositions \ref{prop:HsEest}, \ref{prop:HlEest} and \ref{prop:heess} we get the following.
\begin{corollary} \label{cor:HsEest}
Let $\nu > -1$. Then
$$
0 \le \mathfrak{H}_t^{\nu}(x,y) \lesssim \mathfrak{G}_t^{\nu}(x,y) \qquad \textrm{and} \qquad
0 \le \mathfrak{H}_t^{\nu,M}(x,y) \lesssim \mathfrak{G}_t^{\nu,M}(x,y)
$$
uniformly in $x,y \in (0,1)$ and $t > 0$.
\end{corollary}

This, in turn, leads immediately to further important corollaries. First of all, the integral representations extend the
actions of the semigroups $\mathfrak{H}_t^{\nu}$ and $\mathfrak{H}_t^{\nu,M}$ to all $L^p(d\eta_{\nu})$, $p \ge 1$.
Moreover, the corresponding maximal operators
$$
\mathfrak{H}_{*}^{\nu}f = \sup_{t > 0}\big|\mathfrak{H}_t^{\nu}f\big|, \qquad
\mathfrak{H}_{*}^{\nu,M}f = \sup_{t > 0}\big|\mathfrak{H}_t^{\nu,M}f\big|,
$$
possess the mapping properties from Theorem \ref{thm:maxhess} and,
furthermore, the semigroups converge almost everywhere on the boundary.
\begin{corollary} \label{cor:map}
Let $\nu > -1$. The maximal operators $\mathfrak{H}_{*}^{\nu}$ and $\mathfrak{H}_{*}^{\nu,M}$ are bounded on
$L^p(d\eta_{\nu})$, $1 < p \le \infty$, and from $L^1(d\eta_{\nu})$ to weak $L^1(d\eta_{\nu})$.
\end{corollary}

\begin{proof}
Combine Corollary \ref{cor:HsEest} with Theorem \ref{thm:maxhess}.
\end{proof}

\begin{corollary} \label{cor:max}
Let $\nu > -1$ and $1 \le p \le \infty$. For each $f \in L^p(d\eta_{\nu})$,
$$
\lim_{t \to 0^+} \mathfrak{H}_t^{\nu}f = \lim_{t \to 0^+} \mathfrak{H}_t^{\nu,M}f = f \quad \textrm{a.e.}
$$
\end{corollary}

\begin{proof}
Use Corollary \ref{cor:map} and the linear density in $L^1(d\eta)$ of the system
$\{d_{\nu} \varphi_n^{\nu} : n \ge 2\}$, cf.\ Remark \ref{rem:densC0}.
See the proof of Corollary \ref{cor:maxhess}.
\end{proof}

\begin{remark} \label{rem:smooth}
Kernels of all semigroups appearing in Section \ref{ssec:6ess} are smooth functions of $(x,y,t) \in (0,1)^2 \times (0,\infty)$.
In cases of $\mathfrak{G}_t^{\nu}(x,y)$ and $\mathfrak{G}_t^{\nu,M}(x,y)$ this follows from the connection with
$G_t^{\nu}(x,y)$ and the smoothness of the latter which is known. For $\mathfrak{H}_t^{\nu}(x,y)$ and $\mathfrak{H}_t^{\nu,M}(x,y)$
(as well as for the previous kernels) the property can be inferred from the series representations.
Here the main facts needed are the uniform bounds from Lemma \ref{lem:ues} and the identities
$d_{\nu}^* d_{\nu} = \mathfrak{L}_{\nu} - \lambda_{1,\nu}^2$, $d_{\nu} d_{\nu}^{*} = \mathfrak{M}_{\nu} - \lambda_{1,\nu}^2$.
The details are rather elementary.
\end{remark}

\begin{remark} \label{rem:serdef}
The semigroups appearing in Section \ref{sec:diff_semi} are given on $L^p(d\eta_{\nu})$, $1 \le p \le \infty$, by their integral
representations. Equivalently, they are also given pointwise by their Fourier-Bessel series which a priori define the semigroups
in $L^2(d\eta_{\nu})$. This can be verified with the aid of Lemma \ref{lem:ues} in a well-known way, see e.g.\ \cite{N}.
The crucial point is that the Fourier-Bessel coefficients of an $L^p(d\eta_{\nu})$ function grow at most polynomially.
Moreover, with the series representation it is straightforward to check that $\mathfrak{T}_t^{\nu}f(x)$ and
$\mathfrak{H}_t^{\nu}f(x)$, hence also $\mathfrak{T}_t^{\nu,M}f(x)$ and $\mathfrak{H}_t^{\nu,M}f(x)$, are
smooth functions of $(x,t) \in (0,1)\times (0,\infty)$ for each $f \in L^p(d\eta_{\nu})$, $p \ge 1$.
Furthermore, the representing Fourier-Bessel series can be multiply differentiated term by term, both in $x$ and $t$.
\end{remark}

\subsection{Comments on natural and Lebesgue measures Fourier-Bessel settings} \label{ssec:6com}

As it was already mentioned, the main focus of Section \ref{sec:diff_semi} was put on the essential measure
Fourier-Bessel settings. Thus Section \ref{ssec:FBL_diff} contains in principle only what is needed to prepare
for Section \ref{ssec:6ess}. Nevertheless, in the Lebesgue measure Fourier-Bessel context studying
maximal operators and boundary convergence of the differentiated semigroups makes sense and is perfectly
available given what was done in Section \ref{sec:diff_semi}.
Similar comments pertain to the natural measure Fourier-Bessel setting that was not considered in Section~\ref{sec:diff_semi}.

\section{Sobolev spaces} \label{sec:Sob}

This section is devoted to Sobolev spaces associated with Fourier-Bessel expansions.
Our main focus is again put on the essential measure Fourier-Bessel setting, where as the main result we establish
an isomorphism between Sobolev and potential spaces.
In the other settings we make some crucial observations, in particular that the derivatives used so far in the literature
are not suitable for defining Sobolev spaces in those contexts.

\subsection{Essential measure Fourier-Bessel setting} \label{ssec:sob_ess}

Let $p \ge 1$. We define the Sobolev space ${\W}^p_{\nu}$ as
$$
{\W}^p_{\nu} := \big\{ f \in L^p(d\eta_{\nu}) : d_{\nu}f \in L^p(d\eta_{\nu}) \big\}.
$$
Here $d_{\nu}$ is understood in a weak sense, thus $d_{\nu}f \in L^p(d\eta_{\nu})$ means that the distribution
$d_{\nu}f$ is represented by an $L^p(d\eta_{\nu})$ function.
Equipped with the norm
$$
{\|f\|}_{{{\W}}^p_{\nu}} := \|f\|_{L^p(d\eta_{\nu})} + \|d_{\nu}f\|_{L^p(d\eta_{\nu})},
$$
${\W}^p_{\nu}$ becomes a Banach space.
We now show that the linear span of the system $\{\varphi_n^{\nu}\}$ is dense in ${\W}^p_{\nu}$.

Denote
$$
{S}_{\nu} = \lin\{ \varphi_n^{\nu} : n \ge 1 \}.
$$
\begin{proposition} \label{prop:Sdense}
Let $\nu > -1$. For each $1 \le p < \infty$, ${S}_{\nu}$ is a dense subspace of $\W^p_{\nu}$.
\end{proposition}
Proving this requires some technical preparation.

\begin{lemma} \label{lem:maxpc}
Let $\nu > -1$ and $1 \le p < \infty$. For each $f \in L^p(d\eta_{\nu})$,
$$
\lim_{t \to 0^+} \big\| \mathfrak{T}_t^{\nu}f - f\big\|_{L^p(d\eta_{\nu})}
= \lim_{t \to 0^+} \big\| \mathfrak{H}_t^{\nu}f - f\big\|_{L^p(d\eta_{\nu})} = 0.
$$
\end{lemma}

\begin{proof}
The arguments are very standard, so we provide only their brief description.
The convergence clearly holds when $f$ belongs to a dense subset of $L^p(d\eta_{\nu})$, either
$S_{\nu}$ or $\lin\{d_{\nu}\varphi_n^{\nu} : n \ge 2\}$, respectively (see Remark \ref{rem:densC0}).
Then to get the convergence for arbitrary $f \in L^p(d\eta_{\nu})$ combine the
above with $L^p$-boundedness of $\mathfrak{T}_*^{\nu}$ and $\mathfrak{H}_*^{\nu}$, see Theorem \ref{thm:maxhess} and
Corollary \ref{cor:map}. This works for $p > 1$.

To cover $p =1$ observe that above, instead of the maximal theorems, uniform $L^p$-boundedness of the semigroups is also
sufficient. Our semigroups are indeed uniformly bounded also in $L^1(d\eta_{\nu})$,
$$
\big\| \mathfrak{T}_t^{\nu}f\big\|_{L^1(d\eta_{\nu})} \lesssim \|f\|_{L^1(d\eta_{\nu})} \qquad \textrm{and} \qquad
\big\| \mathfrak{H}_t^{\nu}f\big\|_{L^1(d\eta_{\nu})} \lesssim \|f\|_{L^1(d\eta_{\nu})}
$$
uniformly in $f \in L^1(d\eta_{\nu})$ and $t > 0$. This follows from symmetry of the kernels of these semigroups and
their uniform boundedness in $L^{\infty}$. The latter is a consequence of the bounds
$$
\mathfrak{T}_t^{\nu} \boldsymbol{1}(x) \le \mathfrak{T}_t^{\nu,M}\boldsymbol{1}(x) = 1 \qquad \textrm{and} \qquad
\mathfrak{H}_t^{\nu} \boldsymbol{1}(x) \lesssim \mathfrak{T}_t^{\nu}\boldsymbol{1}(x) \le 1, \qquad
	x \in (0,1), \quad t > 0,
$$
see Corollary \ref{cor:HsEest}; here $\boldsymbol{1}$ is a constant function equal identically to $1$.
\end{proof}

\begin{lemma} \label{lem:derT}
Let $\nu > -1$ and $1 \le p \le \infty$. For each $f \in \W^p_{\nu}$ and $t > 0$
$$
d_{\nu} \mathfrak{T}_t^{\nu}f(x) = \mathfrak{H}_t^{\nu}(d_{\nu}f)(x), \qquad x \in (0,1).
$$
\end{lemma}

\begin{proof}
Let $f \in \W^p_{\nu}$ for some $p \ge 1$. Assuming that
\begin{equation} \label{claimq}
\big\langle f, d_{\nu}^* d_{\nu} \varphi_n^{\nu} \big\rangle_{d\eta_{\nu}}
 = \big\langle d_{\nu}f, d_{\nu}\varphi_n^{\nu}\big\rangle_{d\eta_{\nu}}, \qquad n \ge 1,
\end{equation}
we can write, see Remark \ref{rem:serdef},
\begin{align*}
d_{\nu} \mathfrak{T}_t^{\nu}f(x) & = \sum_{n=1}^{\infty} e^{-t \lambda_{n,\nu}^2} \big\langle f, \varphi_n^{\nu}
	\big\rangle_{d\eta_{\nu}}
	d_{\nu}\varphi_n^{\nu}(x) \\
& = \sum_{n=2}^{\infty} e^{-t\lambda_{n,\nu}^2} \Big\langle f, \big(\lambda_{n,\nu}^2-\lambda_{1,\nu}^2\big)^{-1}
	d_{\nu}^* d_{\nu} \varphi_{n}^{\nu} \Big\rangle_{d\eta_{\nu}} d_{\nu}\varphi_n^{\nu}(x) \\
& = \sum_{n=2}^{\infty} e^{-t \lambda_{n,\nu}^2} \Big\langle d_{\nu}f, d_{\nu}\varphi_n^{\nu}\slash
	\|d_{\nu}\varphi_n^{\nu}\|_{L^2(d\eta_{\nu})}\Big\rangle_{d\eta_{\nu}} d_{\nu}\varphi_n^{\nu}\slash
	\|d_{\nu}\varphi_n^{\nu}\|_{L^2(d\eta_{\nu})} \\
& = \mathfrak{H}_t^{\nu}(d_{\nu}f)(x), \qquad x \in (0,1), \quad t > 0.
\end{align*}
Therefore it remains to justify \eqref{claimq}.

Choose a sequence of smooth and compactly supported functions  $\{\gamma_{m} : m \ge 2\}$ on $(0,1)$ satisfying
$\support\gamma_m \subset (\frac{1}{2m},1-\frac{1}{2m})$, $\gamma_m(x) = 1$ for $x \in (\frac{1}m,1-\frac{1}m)$,
$0 \le \gamma_m(x) \le 1$ for $x \in (0,1)$, and
\begin{equation} \label{eq:dec1}
\gamma_m'(x) \lesssim \frac{1}{x(1-x)}, \qquad x \in (0,1), \quad m \ge 2.
\end{equation}
Notice that $\gamma_{m} \to 1$ pointwise as $m \to \infty$; actually, $\gamma_m(x) = 1$ for each $x\in (0,1)$ fixed and sufficiently
large $m$.

By Proposition \ref{prop:asd} and H\"older's inequality, the product $d_{\nu}f d_{\nu}\varphi_n^{\nu}$
is integrable with respect to $d\eta_{\nu}$. Using the dominated convergence theorem we get
$$
\big\langle d_{\nu}f, d_{\nu}\varphi_n^{\nu}\big\rangle_{d\eta_{\nu}}
= \lim_{m \to \infty} \int_0^1 d_{\nu}f(x) \gamma_m(x) d_{\nu}\varphi_n^{\nu}(x)\, d\eta_{\nu}(x)
= \lim_{m \to \infty} \int_0^1 f(x) d_{\nu}^* \big[ \gamma_m(x) d_{\nu}\varphi_n^{\nu}(x)\big] \, d\eta_{\nu}(x).
$$
Since for each $x \in (0,1)$
$$
d_{\nu}^* \big[ \gamma_m(x) d_{\nu}\varphi_n^{\nu}(x)\big] \to d_{\nu}^* d_{\nu} \varphi_n^{\nu}(x), \qquad m \to \infty,
$$
in order to arrive at \eqref{claimq} it suffices to justify another application of the dominated convergence theorem
by finding a suitable majorant.

We have
$$
d_{\nu}^* \big[ \gamma_m(x) d_{\nu}\varphi_n^{\nu}(x)\big] = -\gamma_m'(x) d_{\nu}\varphi_n^{\nu}(x)
	+ \gamma_m(x) d_{\nu}^* d_{\nu}\varphi_n^{\nu}(x) = -\gamma_{m}'(x) d_{\nu}\varphi_n^{\nu}(x)
		+ \big( \lambda_{n,\nu}^2-\lambda_{1,\nu}^2\big) \gamma_{m}(x) \varphi_n^{\nu}(x),
$$
so by \eqref{eq:dec1} and Proposition \ref{prop:asd} it follows that
$$
\big| d_{\nu}^* \big[ \gamma_m(x) d_{\nu}\varphi_n^{\nu}(x)\big] \big| \lesssim 1, \qquad x \in (0,1), \quad m \ge 2.
$$
Thus the desired majorant is $C |f| \in L^p(d\eta_{\nu}) \subset L^1(d\eta_{\nu})$ for a sufficiently large constant $C$.

Now \eqref{claimq} follows and this finishes the proof.
\end{proof}

We are now ready to prove Proposition \ref{prop:Sdense}.
\begin{proof}[{Proof of Proposition \ref{prop:Sdense}}]
Let $f \in \W^p_{\nu}$ for some $1 \le p < \infty$. Using Lemma \ref{lem:derT} and then Lemma \ref{lem:maxpc} we infer that
$\mathfrak{T}_t^{\nu}f \in \W^{p}_{\nu}$ for $t > 0$, and
$$
\lim_{t \to 0^+} \big\| f - \mathfrak{T}_t^{\nu}f \big\|_{\W^p_{\nu}} = 0.
$$
Thus the desired conclusion will follow once we verify that, given any $t_0 > 0$ fixed, partial Fourier-Bessel
sums of $\mathfrak{T}_{t_0}^{\nu}f$ (that are elements of $S_{\nu}$) converge to $\mathfrak{T}_{t_0}^{\nu}f$ in $\W^p_{\nu}$.

Define
$$
\mathfrak{T}_{t,N}^{\nu}f(x) := \sum_{n=1}^N e^{-t \lambda_{n,\nu}^2} \langle f, \varphi_n^{\nu} \rangle_{d\eta_{\nu}}
	\varphi_n^{\nu}(x), \qquad x \in (0,1), \quad t > 0.
$$
By \eqref{claimq} one has $d_{\nu} \mathfrak{T}_{t,N}^{\nu}f(x) = \mathfrak{H}_{t,N}^{\nu}(d_{\nu}f)(x)$, where
$\mathfrak{H}_{t,N}^{\nu}f$ is the $N$th partial sum of $\mathfrak{H}_t^{\nu}f$. Then using H\"older's inequality
and Lemma \ref{lem:derT} we can write
\begin{align*}
\big\| \mathfrak{T}_{t_0,N}^{\nu}f - \mathfrak{T}_{t_0}^{\nu}f \big\|_{\W^p_{\nu}} & =
 \big\| \mathfrak{T}_{t_0,N}^{\nu}f - \mathfrak{T}_{t_0}^{\nu}f \big\|_{L^p(d\eta_{\nu})} +
	\big\| \mathfrak{H}_{t_0,N}^{\nu}(d_{\nu}f) - \mathfrak{H}_{t_0}^{\nu}(d_{\nu}f) \big\|_{L^p(d\eta_{\nu})} \\
& \le \|f\|_{L^p(d\eta_{\nu})} \Bigg( \sum_{n=N+1}^{\infty} e^{-t_0 \lambda_{n,\nu}^2} \|\varphi_n^{\nu}\|_{L^p(d\eta_{\nu})}
	 \|\varphi_n^{\nu}\|_{L^{p'}(d\eta_{\nu})} \\
	& \qquad \qquad \qquad+ \sum_{n=N+1}^{\infty} e^{-t_0 \lambda_{n,\nu}^2}
		\frac{\|d_{\nu}\varphi_n^{\nu}\|_{L^p(d\eta_{\nu})}\|d_{\nu}\varphi_n^{\nu}\|_{L^{p'}(d\eta_{\nu})}}
			{\|d_{\nu}\varphi_n^{\nu}\|^2_{L^2(d\eta_{\nu})}} \Bigg).
\end{align*}
In view of Lemma \ref{lem:ues}, the last expression is controlled, up to a multiplicative constant, by
$$
\|f\|_{L^p(d\eta_{\nu})} \sum_{n=N+1}^{\infty} e^{-t_0 \lambda_{n,\nu}^2} n^{2\nu+9},
$$
which clearly tends to $0$ as $N \to \infty$, cf.\ \eqref{eigvas}. It follows that
$\mathfrak{T}_{t_0,N}^{\nu}f \to \mathfrak{T}_{t_0}^{\nu}f$ in $\W^p_{\nu}$, as needed.
\end{proof}

Next, we will introduce potential spaces associated with $\mathfrak{L}_{\nu}$. Define, initially for $f \in S_{\nu}$,
\begin{equation} \label{Lpot}
\mathfrak{L}_{\nu}^{-\sigma}f = \sum_{n=1}^{\infty} \lambda_{n,\nu}^{-2\sigma} \langle f, \varphi_n^{\nu} \rangle_{d\eta_{\nu}}
	\varphi_n^{\nu}.
\end{equation}
When $\sigma < 0$, this is the so-called fractional derivative operator, while for $\sigma > 0$ we obtain
the fractional integral operator. In the latter case, \eqref{Lpot} makes sense for $f \in L^2(d\eta_{\nu})$ and
defines a bounded linear operator there. Formally $\mathfrak{L}_{\nu}^{-\sigma}$, $\sigma > 0$, can be written as
an integral operator which we denote $\mathfrak{I}_{\sigma}^{\nu}$,
$$
\mathfrak{I}_{\sigma}^{\nu}f(x) = \int_0^1 \mathfrak{K}_{\sigma}^{\nu}(x,y) f(y)\, d\eta_{\nu}(y),
$$
where the potential kernel is given via the heat kernel,
$$
\mathfrak{K}_{\sigma}^{\nu}(x,y) = \frac{1}{\Gamma(\sigma)} \int_0^{\infty} \mathfrak{G}_t^{\nu}(x,y) t^{\sigma-1}\, dt;
$$
see e.g.\ \cite[Section 1]{NoRo3}. We call $\mathfrak{I}_{\sigma}^{\nu}$ the potential operator.
Since $\mathfrak{I}_{\sigma}^{\nu}$ is bounded on $L^2(d\eta_{\nu})$ (see Proposition \ref{prop:Lppot} below), standard
arguments show (cf.\ e.g.\ \cite[Section 1]{NoRo3} and the reference given there) that for each $\sigma > 0$,
$\mathfrak{L}_{\nu}^{-\sigma}$ and $\mathfrak{I}_{\sigma}^{\nu}$ coincide on $L^2(d\eta_{\nu})$. Thus
$\mathfrak{I}_{\sigma}^{\nu}$ provides a natural extension of $\mathfrak{L}_{\nu}^{-\sigma}$ to more general spaces.
The natural domain of $\mathfrak{I}_{\sigma}^{\nu}$ consists of all functions for which the defining integral converges
for a.e.\ $x \in (0,1)$.

\begin{proposition} \label{prop:Lppot}
Let $\nu > -1$, $\sigma > 0$ and $1 \le p \le \infty$. The space $L^p(d\eta_{\nu})$ is contained in the natural domain of
$\mathfrak{I}_{\sigma}^{\nu}$ and, moreover, $\mathfrak{I}_{\sigma}^{\nu}$ is bounded on $L^p(d\eta_{\nu})$.
\end{proposition}

\begin{proof}
In view of \eqref{GK10},
$$
\mathfrak{K}_{\sigma}^{\nu}(x,y) = \frac{1}{\phi_1^{\nu}(x)\phi_1^{\nu}(y)} \mathcal{K}_{\sigma}^{\nu}(x,y),
$$
where $\mathcal{K}_{\sigma}^{\nu}(x,y)$ is the potential kernel associated with the natural measure Fourier-Bessel setting.
Taking into account the above, the bounds $\phi_1^{\nu}(x) \simeq 1-x$, $x \in (0,1)$, sharp estimates for
$\mathcal{K}_{\sigma}^{\nu}(x,y)$ from \cite[Theorem 2.6]{NoRo3} and finally sharp estimates for the Jacobi
potential kernel from \cite[Theorem 2.2]{NoRo3}, we conclude that $\mathfrak{K}_{\sigma}^{\nu}(x,y)$ has exactly
the same sharp bounds as the potential kernel in the Jacobi trigonometric polynomial setting with natural measure,
scaled to $(0,1)$, with parameters $\alpha = \nu$ and $\beta = 1/2$. Furthermore, the corresponding measures are
comparable. Thus $\mathfrak{I}_{\sigma}^{\nu}$ has the same $L^p$ mapping properties as the Jacobi potential operator
in the just mentioned circumstances. Then it follows from \cite[Theorem 2.3]{NoRo3} and comments preceding it that,
in particular, $L^p(d\eta_{\nu})$ lies in the natural domain of $\mathfrak{I}_{\sigma}^{\nu}$, and
$\mathfrak{I}_{\sigma}^{\nu}$ is bounded on $L^p(d\eta_{\nu})$, $1 \le p \le \infty$.
\end{proof}

\begin{proposition} \label{prop:1-1}
Let $\nu > -1$, $\sigma > 0$ and $1 < p < \infty$. Then $\mathfrak{I}_{\sigma}^{\nu}$ is injective on $L^p(d\eta_{\nu})$.
\end{proposition}

\begin{proof}
The arguments are standard, see e.g.\ \cite[Proposition 2.4]{La1} and its proof.
Let $f \in L^p(d\eta_{\nu})$ for some $1 < p < \infty$.
First, we note that
\begin{equation} \label{fls}
\big\langle \mathfrak{I}_{\sigma}^{\nu}f,\varphi_n^{\nu}\big\rangle_{d\eta_{\nu}}
= \lambda_{n,\nu}^{-2\sigma} \langle f, \varphi_n^{\nu}\rangle_{d\eta_{\nu}}, \qquad n \ge 1.
\end{equation}
This is clear when $f \in S_{\nu}$. Further, by H\"older's inequality and Proposition \ref{prop:Lppot} the functionals
defined by the left-hand and right-hand sides of \eqref{fls} are bounded from $L^p(d\eta_{\nu})$ to $\mathbb{C}$.
Since $S_{\nu}$ is dense in $L^p(d\eta_{\nu})$, we see that \eqref{fls} holds for $f \in L^p(d\eta_{\nu})$.

Next, we show the following claim:
if $\langle f, \varphi_n^{\nu}\rangle_{d\eta_{\nu}} = 0$ for all $n \ge 1$, then $f \equiv 0$.
To justify this observe that the antecedent of the conditional means that $\langle f, g\rangle_{d\eta_{\nu}} = 0$
for all $g \in S_{\nu}$, and $S_{\nu}$ is dense in the dual space $(L^p(d\eta_{\nu}))^* = L^{p'}(d\eta_{\nu})$,
cf.\ Remark \ref{rem:densC0}.

Finally, taking into account \eqref{fls} and the claim, we infer that $\mathfrak{I}_{\sigma}^{\nu}f \equiv 0$ implies
$f \equiv 0$. It follows that $\mathfrak{I}_{\sigma}^{\nu}$ is one-to-one.
\end{proof}

We now define the Fourier-Bessel potential spaces $\L_{\nu}^{p,\sigma}$.
Keeping in mind Propositions \ref{prop:Lppot} and \ref{prop:1-1}, for $\nu > -1$, $\sigma > 0$ and $1 < p < \infty$ we set
$$
\L_{\nu}^{p,\sigma} := \mathfrak{I}_{\sigma/2}^{\nu} \big( L^p(d\eta_{\nu}) \big)
$$
and equip this space with a norm via the formula
$$
\|f\|_{\L_{\nu}^{p,\sigma}} := \|g\|_{L^p(d\eta_{\nu})}, \qquad f = \mathfrak{I}_{\sigma /2}^{\nu}g, \qquad g \in L^p(d\eta_{\nu}).
$$
It is straightforward to check that $(\L_{\nu}^{p,\sigma}, \|\cdot\|_{\L_{\nu}^{p,\sigma}})$ is a Banach space.
\begin{proposition} \label{prop:Ldense}
Let $\nu > -1$ and $\sigma > 0$. For each $1 < p < \infty$, $S_{\nu}$ is a dense subspace of $\L_{\nu}^{p,\sigma}$.
\end{proposition}

\begin{proof}
Combine the following facts: $\mathfrak{I}_{\sigma/2}^{\nu}$ is bounded and one-to-one on $L^p(d\eta_{\nu})$, $S_{\nu}$ is
a dense subspace of $L^p(d\eta_{\nu})$, and $S_{\nu}$ coincides with its image under the action of $\mathfrak{I}_{\sigma/2}^{\nu}$.
\end{proof}

\begin{remark}
The potential spaces $\L_{\nu}^{p,\sigma}$ possess various expected monotonicity, structural and embedding properties,
which can be established by standard methods, see e.g.\ \cite[Proposition 5.1]{La1} and \cite[Section~3]{La2}.
However, we shall not pursue these matters here to avoid confusion with the main line of the paper.
\end{remark}

We are now ready to formulate the analogue of the classical result of A.P.\ Calder\'on \cite{C}, which is the main result
of this section. This result is a crucial point when it comes to consistency of the theory of Fourier-Bessel Sobolev spaces with
the classical one.
\begin{theorem} \label{thm:isoCal}
Let $\nu \ge -1/2$ and $1 < p < \infty$. Then
$$
\W_{\nu}^p = \L_{\nu}^{p,1}
$$
in the sense of isomorphism of Banach spaces.
\end{theorem}

\begin{proof}
Since $S_{\nu}$ is dense both in $\W_{\nu}^p$ and in $\L_{\nu}^{p,1}$, it suffices to show that the norms are comparable
on~$S_{\nu}$,
$$
\|f\|_{\W_{\nu}^p} \simeq \|f\|_{\L_{\nu}^{p,1}}, \qquad f \in S_{\nu}.
$$
To this end, we always assume that $f \in S_{\nu}$.

Observe that the following identities hold:
$$
 f = \mathfrak{L}_{\nu}^{-1/2} \mathfrak{L}_{\nu}^{1/2}f = \mathfrak{I}_{1/2}^{\nu} \mathfrak{L}_{\nu}^{1/2}f, \qquad
d_{\nu}f = \mathfrak{R}_{\nu} \mathfrak{L}_{\nu}^{1/2}f.
$$
Using this together with the $L^p$-boundedness of $\mathfrak{R}_{\nu}$ (Theorem \ref{thm:Riesz}) and $\mathfrak{I}_{1/2}^{\nu}$
(Proposition \ref{prop:Lppot}) we can write, uniformly in $f \in S_{\nu}$,
\begin{align*}
\|f\|_{\W_{\nu}^p} & = \|f\|_{L^p(d\eta_{\nu})} + \|d_{\nu}f\|_{L^p(d\eta_{\nu})}
	= \big\| \mathfrak{I}_{1/2}^{\nu} \mathfrak{L}_{\nu}^{1/2}f\big\|_{L^p(d\eta_{\nu})}
		+ \big\|\mathfrak{R}_{\nu} \mathfrak{L}_{\nu}^{1/2}f\big\|_{L^p(d\eta_{\nu})} \\
		& \lesssim \big\| \mathfrak{L}_{\nu}^{1/2}f\big\|_{L^p(d\eta_{\nu})} = \|f\|_{\L_{\nu}^{p,1}}.
\end{align*}

On the other hand, in view of the identities
$$
\mathfrak{L}_{\nu}^{1/2}f = \mathfrak{L}_{\nu}^{-1/2}\mathfrak{L}_{\nu} f
	= \mathfrak{L}_{\nu}^{-1/2} \big( \lambda_{1,\nu}^2 + d_{\nu}^* d_{\nu}\big)f
	= \lambda_{1,\nu}^2 \mathfrak{L}_{\nu}^{-1/2}f + \mathfrak{R}_{\nu}^* d_{\nu}f,
$$
with $\mathfrak{R}_{\nu}^{*} = \mathfrak{L}_{\nu}^{-1/2}d_{\nu}^*$ being the adjoint of $\mathfrak{R}_{\nu}$, we have
\begin{align*}
\|f\|_{\L_{\nu}^{p,1}} & = \big\| \mathfrak{L}_{\nu}^{1/2}f \big\|_{L^p(d\eta_{\nu})}
	\le \lambda_{1,\nu}^2 \big\| \mathfrak{I}_{1/2}^{\nu}f \big\|_{L^p(d\eta_{\nu})}
		+ \big\| \mathfrak{R}_{\nu}^* d_{\nu}f\big\|_{L^p(d\eta_{\nu})} \\
	& \lesssim \|f\|_{L^p(d\eta_{\nu})} + \|d_{\nu}f\|_{L^p(d\eta_{\nu})} = \|f\|_{\W_{\nu}^p},
\end{align*}
uniformly in $f \in S_{\nu}$.
Here to write the last estimate we used again the $L^p$-boundedness of $\mathfrak{I}_{1/2}^{\nu}$ and
$L^{p'}$-boundedness of $\mathfrak{R}_{\nu}$;
notice that by duality the latter implies $L^p$-boundedness of $\mathfrak{R}_{\nu}^{*}$.

The conclusion follows.
\end{proof}

\begin{remark}
The slight restriction on the parameter in Theorem \ref{thm:isoCal} is inherited from Theorem \ref{thm:Riesz}, since
$L^p$-boundedness of the Riesz transform is an essential ingredient of the proof. Nevertheless, it is natural to conjecture that
Theorem \ref{thm:isoCal} holds for all $\nu > -1$, see Remark \ref{rem:Riesz}.
\end{remark}

\begin{remark}
In the probabilistic setting of $\mathfrak{L}^M_{\nu}$ zero is the bottom eigenvalue and the negative powers/Riesz potentials
$(\mathfrak{L}_{\nu}^M)^{-\sigma}$ are not well defined even on $S_{\nu}$.
Thus the appropriate potential operators underlying the definition of potential spaces associated with $\mathfrak{L}_{\nu}^M$
are the Bessel type potentials $(c + \mathfrak{L}_{\nu}^M)^{-\sigma}$, $c > 0$ fixed. Here it is natural and convenient
to choose $c = \lambda_{1,\nu}^2$. Since $\lambda_{1,\nu}^2 + \mathfrak{L}_{\nu}^M = \mathfrak{L}_{\nu}$, we see
that potential spaces associated with $\mathfrak{L}_{\nu}^M$ are the same as for $\mathfrak{L}_{\nu}$.
Choosing some other $c$, in particular $c=1$,
does not lead to any qualitative change, only to some additional and mostly parallel analysis
related to the Laplacian $\mathfrak{L}_{\nu} - \lambda_{1,\nu}^2 + c$ in place of $\mathfrak{L}_{\nu}$
in Sections \ref{sec:riesz} and \ref{sec:Sob}. We omit the details.
\end{remark}

The problem of defining higher order Sobolev spaces (that includes defining higher order Fourier-Bessel derivatives)
so that they are isomorphic with the corresponding Fourier-Bessel potential spaces remains open.

\subsection{Comments on natural and Lebesgue measure Fourier-Bessel settings} \label{ssec:sob_nl}

Recall from Section~\ref{sec:der} that the derivatives used commonly in the literature in the natural and Lebesgue measures
Fourier-Bessel settings are, respectively,
$$
\mathfrak{d}_{\nu} = \frac{d}{dx} \qquad \textrm{and} \qquad \mathfrak{D}_{\nu} = \frac{d}{dx} - \frac{\nu+1/2}{x}.
$$
On the other hand, our derivatives introduced in this paper are, respectively,
$$
\delta_{\nu} = \frac{d}{dx} + R^{\nu}(x) \qquad \textrm{and} \qquad \mathbb{D}_{\nu} = \frac{d}{dx} - \frac{\nu+1/2}x + R^{\nu}(x).
$$
Note that $\delta_{\nu}$ and $\mathbb{D}_{\nu}$ are compatible with $d_{\nu}$ in the essential measure Fourier-Bessel context,
see Section~\ref{sec:der}, which is not the case of $\mathfrak{d}_{\nu}$ and $\mathfrak{D}_{\nu}$.
We claim that $\mathfrak{d}_{\nu}$ and $\mathfrak{D}_{\nu}$, in contrast with $\delta_{\nu}$ and $\mathbb{D}_{\nu}$,
are not suitable for defining Sobolev spaces associated with the corresponding settings.

For $\nu > -1$ and $1 \le p < \infty$ define
\begin{align*}
\mathcal{W}_{\nu}^p & := \{ f \in L^p(d\mu_{\nu}) : \delta_{\nu}f \in L^p(d\mu_{\nu}) \}, \\
\widetilde{\mathcal{W}}_{\nu}^p & := \{ f \in L^p(d\mu_{\nu}) : \mathfrak{d}_{\nu}f \in L^p(d\mu_{\nu}) \}, \\
\mathbb{W}_{\nu}^p & := \{ f \in L^p(dx) : \mathbb{D}_{\nu}f \in L^p(dx)\}, \\
\widetilde{\mathbb{W}}_{\nu}^p & := \{ f \in L^p(dx) : \mathfrak{D}_{\nu}f \in L^p(dx) \},
\end{align*}
where the derivatives are understood in a weak sense.
It is straightforward to see that 
\begin{equation} \label{neqSob}
\mathcal{W}_{\nu}^p \neq \widetilde{\mathcal{W}}_{\nu}^p \qquad \textrm{and} \qquad
\mathbb{W}_{\nu}^p \neq \widetilde{\mathbb{W}}_{\nu}^p.
\end{equation}
Indeed, the function $f = \chi_{(1/2,1)}$ is in the tilded spaces, but does not belong to the untilded ones.
With some more effort, with the aid of some known $L^p$ inequalities for a Hardy operator, it can be shown that
actually $\mathcal{W}_{\nu}^p \varsubsetneq \widetilde{\mathcal{W}}_{\nu}^p$
and $\mathbb{W}_{\nu}^p \varsubsetneq \widetilde{\mathbb{W}}_{\nu}^p$, but we shall not go into details here.

Recall now from Section \ref{ssec:comm_Riesz} the connections with the Jacobi setting that occur for $\nu = \pm 1/2$.
When $\nu = 1/2$ the Lebesgue measure Fourier-Bessel context coincides with the Jacobi one with parameters $\a=\b=1/2$.
When $\nu = -1/2$ both the Fourier-Bessel settings under consideration coincide with the Jacobi one with parameters
$\a=-1/2$, $\b = 1/2$. Thus from the results obtained in the Jacobi framework \cite{La1} we know that
$\mathcal{W}_{-1/2}^p$ and $\mathbb{W}_{\pm 1/2}^p$ are for $1<p<\infty$ isomorphic with the corresponding potential spaces,
and because of \eqref{neqSob} the tilded counterparts are not. 
Similar conclusions can be drawn for all $\nu \ge -1/2$ in the special case $p=2$, since then Theorem \ref{thm:isoCal}
can be transferred to the Lebesgue and natural measures Fourier-Bessel settings via the mappings $U_{\nu}$ and
$\widetilde{U}_{\nu}$, see Section \ref{sec:prel}.
All this strongly indicates that for general $\nu$ a crucial aspect of the theory, which is the isomorphism between Sobolev
and potential spaces, holds for the untilded spaces, but fails for the tilded ones.
A strict formal confirmation of this general conjecture is another interesting open problem.

\section*{Appendix: summary of notation}

For readers' convenience, in Table \ref{tab:notation} below we summarize the notation of various objects in
the contexts appearing in this paper, that is 
	\begin{itemize}
	 \item \textbf{F}ourier-\textbf{B}essel \textbf{essential} measure setting / its \textbf{prob}abilistic \textbf{variant}, 
	 \item \textbf{F}ourier-\textbf{B}essel \textbf{natural} measure setting,
	 \item \textbf{F}ourier-\textbf{B}essel \textbf{Lebesgue} measure setting,
	 \item \textbf{Jacobi} trigonometric function setting \textbf{scaled} to the interval $(0,1)$.
	 \end{itemize}
References in square brackets indicate subsections where the objects are defined.
{{
\begin{table*}[htbp]
\centering
\scalebox{0.9}{
\begin{tabular}{c|c|c|c|c|}
\cline{2-5}
 & F-B essential / prob.\ variant
 & F-B natural
 & F-B Lebesgue
 & Jacobi scaled \\ 
\cline{1-5} 
   \multicolumn{1}{|c|}{eigenfunctions} 
   & $\varphi_n^{\nu}$, $n\ge 1$ \; [\textsection \ref{FBneu}]
   & $\phi_n^{\nu}$, $n\ge 1$   \; [\textsection \ref{ssec:FBnat}]
   & $\psi_n^{\nu}$, $n\ge 1$   \; [\textsection \ref{ssec:FBleb}]
   & $\Phi_k^{\alpha,\beta}$, $k\ge 0$ \; [\textsection \ref{sec:jac}]\\
\cline{1-5} 
   \multicolumn{1}{|c|}{reference measure} 
   & $d\eta_{\nu}$ \; [\textsection\ref{FBneu}]
   & $d\mu_{\nu}$ \; [\textsection \ref{ssec:FBnat}]
   & $dx$ \; [\textsection \ref{ssec:FBleb}]
   & $dx$ \; [\textsection \ref{sec:jac}]\\
\cline{1-5} 
   \multicolumn{1}{|c|}{Laplacian} 
   & $\mathfrak{L}_{\nu}$ / $\mathfrak{L}_{\nu}^M$ \; [\textsection\ref{FBneu}]
   & $\mathcal{L}_{\nu}$ \;  [\textsection \ref{ssec:FBnat}]
   & $\mathbb{L}_{\nu}$  \; [\textsection\ref{ssec:FBleb}]
   & $\mathbb{J}_{\alpha,\beta}$ \; [\textsection\ref{sec:jac}]\\
\cline{1-5} 
   \multicolumn{1}{|c|}{derivative} 
   & $d_{\nu}$ \; [\textsection\ref{sec:ess}]
   & $\delta_{\nu}$ \; [\textsection\ref{sec:FBdernat}]
   & $\mathbb{D}_{\nu}$  \; [\textsection\ref{sec:FBderleb}]
   & $D_{\alpha,\beta}$ \; [\textsection\ref{sec:jacder}]\\ 
\cline{1-5} 
   \multicolumn{1}{|c|}{old derivative} 
   & 
   & $\mathfrak{d}_{\nu}$ \; [\textsection\ref{ssec:sob_nl}]
   & $\mathfrak{D}_{\nu}$ \; [\textsection\ref{ssec:sob_nl}]
   &  \\
\cline{1-5} 
   \multicolumn{1}{|c|}{Riesz transform} 
   & $\mathfrak{R}_{\nu}$ / $\mathfrak{R}_{\nu}^M$ \; [\textsection\ref{sec:FBriesz}]
   & 
   & 
   & \\ 
\cline{1-5} 
   \multicolumn{1}{|c|}{heat kernel} 
	 & $\mathfrak{G}_t^{\nu}(x,y)$ / $\mathfrak{G}_t^{\nu,M}(x,y)$  \; [\textsection\ref{ssec:6ess}]
   & $G_t^{\nu}(x,y)$ \; [\textsection\ref{ssec:6ess}]
   & 
   & $\mathbb{G}_t^{\alpha,\beta}(x,y)$ \; [\textsection\ref{ssec:jac_diff}] \\
\cline{1-5} 
   \multicolumn{1}{|c|}{heat semigroup} 
	 & $\mathfrak{T}_t^{\nu}$ / $\mathfrak{T}_t^{\nu,M}$ \; [\textsection\ref{ssec:6ess}]
   & 
   & 
   &  \\ 
\cline{1-5} 
   \multicolumn{1}{|c|}{modified Laplacian} 
   & $\mathfrak{M}_{\nu}$ / $\mathfrak{M}_{\nu}^M$ \; [\textsection\ref{ssec:6ess}]
   & 
   & $\mathbb{M}_{\nu}$ \; [\textsection\ref{ssec:FBL_diff}]
   & $M_{\alpha,\beta}$ \; [\textsection\ref{ssec:jac_diff}]\\ 
\cline{1-5} 
   \multicolumn{1}{|c|}{modified heat kernel} 
	 & $\mathfrak{H}_t^{\nu}(x,y)$ / $\mathfrak{H}_t^{\nu,M}(x,y)$ \; [\textsection\ref{ssec:6ess}]
   & 
   & $\mathbb{H}_t^{\nu}(x,y)$  \; [\textsection\ref{ssec:FBL_diff}]
   & $H_t^{\alpha,\beta}(x,y)$ \; [\textsection\ref{ssec:jac_diff}] \\ 
\cline{1-5} 
   \multicolumn{1}{|c|}{modified heat semigroup} 
	 & $\mathfrak{H}_t^{\nu}$ / $\mathfrak{H}_t^{\nu,M}$ \; [\textsection\ref{ssec:6ess}]
   & 
   & 
   &  \\
\cline{1-5} 
   \multicolumn{1}{|c|}{potential kernel} 
	 & $\mathfrak{K}_{\sigma}^{\nu}(x,y)$  \; [\textsection\ref{ssec:sob_ess}]
   & $\mathcal{K}^{\nu}_{\sigma}(x,y)$  \; [\textsection\ref{ssec:sob_ess}]
   & 
   &  \\
\cline{1-5} 
   \multicolumn{1}{|c|}{potential operator} 
	 & $\mathfrak{I}_{\sigma}^{\nu}$  \; [\textsection\ref{ssec:sob_ess}]
   & 
   & 
   &  \\
\cline{1-5} 
   \multicolumn{1}{|c|}{Sobolev space} 
	 & $\mathbf{W}_{\nu}^p$ \; [\textsection\ref{ssec:sob_ess}]
   & $\mathcal{W}_{\nu}^p$ \; [\textsection\ref{ssec:sob_nl}]
   & $\mathbb{W}_{\nu}^p$ \; [\textsection\ref{ssec:sob_nl}]
   &  \\ 
\cline{1-5}
\noalign{\bigskip}
\end{tabular}
}
\caption{Summary of notation.}
\label{tab:notation}
\end{table*}
}}


\end{document}